\documentclass[11pt,twoside]{article}
\usepackage{times}
\usepackage{CJK}
\usepackage{mathptmx}
\usepackage{amsmath,amssymb,mathrsfs,color,times,textcomp,sectsty}
\usepackage{enumerate}
\usepackage{cite}
\usepackage{mathrsfs,graphicx}
\usepackage{float}
\usepackage{stmaryrd}
\usepackage[margin=1in]{geometry}  
\usepackage{graphicx}
\usepackage{slashed}            
\usepackage{amsfonts}              
\usepackage{amsthm}
\usepackage{breqn}
\usepackage{verbatim}              
\usepackage{fancyhdr}
\usepackage{color}
\usepackage[pagewise]{lineno}
\usepackage{amssymb}
\usepackage{authblk}
\begin{document}

\footskip=0pt
\footnotesep=2pt

\allowdisplaybreaks

\newtheorem{claim}{Claim}[section]

\theoremstyle{definition}
\newtheorem{thm}{Theorem}[section]
\newtheorem*{thmm}{Theorem}
\newtheorem{mydef}{Definition}[section]
\newtheorem{lem}[thm]{Lemma}
\newtheorem{prop}[thm]{Proposition}
\newtheorem{remark}[thm]{Remark}
\newtheorem{rem}{Remark}[section]
\newtheorem*{propp}{Proposition}
\newtheorem{cor}[thm]{Corollary}
\newtheorem{conj}[thm]{Conjecture}
\def\F{\mathcal{F}}
\def\la{\lambda}
\def\tR{\tilde{R}}
\def\ds{\displaystyle}
\def\ve{\varepsilon}
\def\g{\gamma}
\def\ls{\lesssim}
\newcommand{\bd}[1]{\mathbf{#1}}  
\newcommand{\R}{\mathbb{R}}      
\newcommand{\ZZ}{\mathbb{Z}}      
\newcommand{\Om}{\Omega}
\newcommand{\bv}{\mathbf{v}}
\newcommand{\bn}{\mathbf{n}}
\newcommand{\bU}{\mathbf{U}}
\newcommand{\q}{\quad}
\newcommand{\p}{\partial}
\newcommand{\n}{\nabla}
\newcommand{\f}{\frac}

\numberwithin{equation}{section}

\title{On the blowup mechanism of smooth solutions to 1D quasilinear strictly hyperbolic systems
with large initial data\footnote{Li Jun (lijun@nju.edu.cn) is supported by NSFC (No.11871030). 
Xu Gang (gxumath@outlook.com, gxu@njnu.edu.cn) and Yin Huicheng (huicheng@nju.edu.cn, 05407@njnu.edu.cn) are 
supported by NSFC (No.11731007, No.11971237).}}

\author[1]{Li Jun}
\author[2]{Xu Gang}
\author[1,2]{Yin Huicheng}
\affil[1]{Department of Mathematics, Nanjing University, Nanjing 210093, China}
\affil[2]{School of Mathematical Sciences and Institute of Mathematical Sciences, Nanjing Normal University, Nanjing 210023, China}

\date{}
\maketitle
\centerline{}

\date{}
\maketitle
\thispagestyle{empty}
\begin{abstract} For the first order 1D $n\times n$ quasilinear
strictly hyperbolic system $\p_tu+F(u)\p_xu=0$ with $u(x, 0)=\ve u_0(x)$, where $\ve>0$ is small,
$u_0(x)\not\equiv 0$ and $u_0(x)\in C_0^2(\Bbb R)$,
when at least one eigenvalue of $F(u)$ is genuinely nonlinear, it is well-known that on the finite blowup time $T_{\ve}$,
the derivatives $\p_{t,x}u$ blow up while the solution $u$  keeps to be small.
For the 1D scalar equation or $2\times 2$ strictly hyperbolic system (corresponding to $n=1, 2$),
if the smooth solution $u$ blows up in finite time, then
the blowup mechanism  can be well understood
(\emph{i.e.}, only the blowup of $\p_{t,x}u$ happens). In the present paper,
for the $n\times n$ ($n\ge 3$) strictly hyperbolic system with a class of large initial data,
we are concerned with the  blowup mechanism of smooth solution $u$ on the finite blowup time and
the detailed singularity behaviours
of  $\p_{t,x}u$ near the blowup point. Our results are based on the efficient  decomposition of $u$ along the different
characteristic directions, the suitable introduction of the modulated coordinates and the global weighted energy estimates.

\end{abstract}

\vskip 0.2cm

\vskip 0.2 true cm


\setcounter{tocdepth}{1}
\tableofcontents

\section{Introduction}\label{i}

In the paper, we are concerned with the blowup mechanism  of smooth solutions to the
following Cauchy problem of 1D $n\times n$ quasilinear strictly hyperbolic system:
\begin{subequations}\label{i-1}\begin{align}
&\partial_t u+F(u)\partial_x u=0,\label{i-1a}\\
&u(x, 0)=u_0(x),\label{i-1b}
\end{align}
\end{subequations}
where $t\ge 0$, $x\in\Bbb R$, $u=(u_1, \cdots, u_n)^{\top}$, the $n\times n$ real matrix $F(u)$ is smooth on its argument $u$,
and $u_0(x)\in C^2(\Bbb R)$. The strict hyperbolicity of system \eqref{i-1a} means that $F(u)$ has $n$ distinct real eigenvalues
\begin{equation}\label{i-2}
\lambda_1(u)<\cdots<\lambda_n(u),
\end{equation}
meanwhile the corresponding right eigenvectors are denoted by $\g_1(u), \cdots, \g_n(u)$ respectively. One calls
system \eqref{i-1a} to be genuinely nonlinear with respect to some eigenvalue $\lambda_{i_0}(u)$ ($1\le i_0\le n$) when
\begin{equation}\label{i-3}
\nabla_u\lambda_{i_0}(u)\cdot \gamma_{i_0}(u)\neq 0.
\end{equation}
Otherwise, \eqref{i-1a} is called to be linearly degenerate with respect to the eigenvalue $\lambda_{i_0}(u)$ when
\begin{equation}\label{y-1}
\nabla_u\lambda_{i_0}(u)\cdot \gamma_{i_0}(u)\equiv 0.
\end{equation}
Our purpose of the paper is to discuss the blowup mechanism of smooth solutions to problem \eqref{i-1}
for a class of  large smooth initial data $u_0(x)$  provided that system \eqref{i-1a} is genuinely nonlinear with respect to some eigenvalue $\lambda_{i_0}(u)$ for $i_0\in \{1, \cdots, n\}$.

\subsection{Reviews and problems}
For the 1D scalar equation
\begin{equation}\label{Y-2}\begin{cases}
\partial_t v+f(v)\partial_x v=0,\\
v(x, 0)=v_0(x),
\end{cases}
\end{equation}
where $v_0(x)\not\equiv 0$, $v_0(x)\in C_0^1(\mathbb{R})$, $f(v)$ is a $C^1$ smooth function
and  $f'(v)\not= 0$ for $v\in \text{supp} v_0(x)$.
Set $g(x)=f(v_0(x))$, then by the characteristics method, it is easy to know that
the $C^1$ solution $v$ will blow up
on the finite positive time $T^*=-\frac{1}{\min g'(x)}$ due to $\min\limits_{x\in\mathbb{R}}g'(x)<0$. Meanwhile,
$v\in C(\mathbb{R}\times [0, T^*])$ and $\lim\limits_{t\nearrow T^*}\|\p_{t,x}v(\cdot, t)\|_{C(\Bbb R)}=\infty$ hold.
This illustrates that the blowup of solution $v$ to problem \eqref{Y-2} corresponds  to the
geometric blowup by the terminology in \cite{A0}.

For the 1D $2\times 2$ strictly hyperbolic system
\begin{equation}\label{Y-3}\begin{cases}
\partial_t v+B(v)\partial_x v=0,\\
v(x, 0)=v_0(x),
\end{cases}
\end{equation}
where $v_0(x)\not\equiv 0$, $v_0(x)\in C_0^1(\Bbb R)$, $B(v)\in C^1$ is a $2\times 2$ matrix which admits two
distinct real eigenvalues $\la_1(v)$ and $\la_2(v)$, by introducing
two Riemann invariants $w_1=w_1(v)$ and  $w_2=w_2(v)$, then \eqref{Y-3} can be decoupled into the following
$2\times 2$ strictly hyperbolic system of $w=(w_1,w_2)$:
\begin{equation}\label{Y-4}\begin{cases}
\partial_t w_1+\la_1(w)\partial_x w_1=0,\\
\partial_t w_2+\la_2(w)\partial_x w_2=0,\\
w(x, 0)=w_0(x).
\end{cases}
\end{equation}
When the system in \eqref{Y-3} is genuinely nonlinear with respect to at least one eigenvalue $\la_i(v)$ ($i=1,2$),
then by \eqref{Y-4} and \cite{PDL1}, one knows that the smooth solution $v$ will blow up at the maximal
finite existence time $T^*$, meanwhile $\|v\|_{L^{\infty}(\mathbb{R}\times [0, T^*])}$ is bounded and
$\lim\limits_{t\nearrow T^*}\|\p_{t,x}v(\cdot, t)\|_{C(\Bbb R)}=\infty$ holds.
This implies that the blowup of solution $v$ to \eqref{Y-3} also corresponds  to the geometric blowup.

For the small data solution problem of 1D $n\times n$ quasilinear strictly hyperbolic system
\begin{equation}\label{Y-1}\begin{cases}
\partial_t v+B(v)\partial_x v=0,\\
v(x, 0)=\ve v_0(x),
\end{cases}
\end{equation}
where $\ve>0$ is small, $v_0(x)\not\equiv 0$, $v_0(x)\in C_0^2(\Bbb R)$ and $B(v)\in C^2$ is a $n\times n$ matrix,
when the system in \eqref{Y-1} is genuinely nonlinear with respect to at least
one eigenvalue of $B(v)$, it follows from the results in  \cite{LH} and \cite{FJ} that the lifespan $T_{\ve}$ of smooth solution $v$ to \eqref{Y-1}
satisfies
$$\lim\limits_{\ve\to 0^{+}}\ve T_{\ve}=\tau_0>0.$$
Moreover, $\|v\|_{C(\mathbb{R}\times [0, T_\ve])}\le C\ve$ and  $\ds\lim_{t\to T_{\ve}-}\|\p_{t,x}v(\cdot, t)\|_{C(\Bbb R)}=\infty$ hold.
This means that the blowup of solution $v$ to \eqref{Y-1} corresponds  to the  geometric blowup.

Compared with the results on problem \eqref{Y-2} and problem \eqref{Y-3}, two natural problems arise for the system \eqref{i-1a}
with $n\geq 3$: when at least one eigenvalue of $F(u)$  is genuinely nonlinear,

{\bf Q1.} Can we find
a class of large initial data \eqref{i-1b} such that the blowup of solution $u$ corresponds  to the  geometric blowup
as in the small data solution problem \eqref{Y-1}?

{\bf Q2.} Can we find
another class of large initial data \eqref{i-1b} such that the solution $u$ itself blows up in finite time?

In the present paper, we focus on the investigation of {\bf Q1}.

\subsection{Statement of main results}
By Proposition \ref{lemA-1} in Section \ref{II},
\eqref{i-1} can be equivalently changed into the following problem
\begin{subequations}\label{i-6}\begin{align}
&\partial_t w+A(w)\partial_x w=0,\label{i-6a}\\
&w(x, -\varepsilon)=w_0(x),\label{i-6b}
\end{align}
\end{subequations}
where $w(x, t)=(w_1, \cdots, w_n)^{\top}$, $t\geq -\varepsilon$, and $\varepsilon>0$ is a
small constant (for the convenience of expression,
here the initial temporal variable is shifted from $t=0$ to $t=-\varepsilon$).
In addition, the $n$ distinct real eigenvalues of smooth function matrix $A(w)=\left(a_{ij}(w)\right)_{n\times n}$
are denoted by $\mu_1(w), \cdots, \mu_n(w)$. Based on the reduction in Proposition \ref{lemA-1} and the strictly
hyperbolic condition \eqref{i-2}, there hold
\begin{subequations}\label{i-7}\begin{align}
&\mu_1(w)<\cdots<\mu_{i_0-1}(w)<\mu_n(w)<\mu_{i_0}(w)<\cdots<\mu_{n-1}(w),\label{i-7a}\\
&a_{in}(w)=0\ (1\leq i\leq n-1),\label{i-7b}\\
& a_{nn}(w)=\mu_n(w)=\mu_n(0)+\partial_{w_n}\mu_n(0)w_n+\sum\limits_{i=1}^{n-1}\partial_{w_i}\mu_n(0)w_i+O(|w|^2),\label{i-7c}\\
&A(0)=\text{diag}\{\mu_1(0), \cdots, \mu_n(0)\},\label{i-7d}
\end{align}
\end{subequations}
where $\partial_{w_n}\mu_n(0)\neq 0$. This means that the
system \eqref{i-6a} is genuinely nonlinear with respect to the eigenvalue $\mu_n(w)$. Let $\ell_i(w)$ and $\gamma_i(w)$ be the left and right eigenvectors of the matrix $A(w)$ corresponding to the eigenvalue $\mu_i(w)$  ($1\le i\le n$), respectively.
Together with \eqref{i-7}, without loss of generality, one can assume
\begin{subequations}\label{i-71}\begin{align}
&\ell_i(w)\cdot\gamma_j(w)=\delta_{i}^{j}\ (1\leq i, j\leq n),\label{i-71a}\\
& \gamma_n(w)={\bf e}_n, \ \gamma_i(0)={\bf e}_i,\ \|\gamma_i(w)\|=1\ (1\leq i\leq n-1),\label{i-71b}\\
&\ell_i(0)={\bf e}_i^{\top}\ (1\leq i\leq n),\label{i-71c}
\end{align}
\end{subequations}
where $\|\gamma_i(w)\|=\sqrt{\ds\sum_{k=1}^n(\gamma_{i}^k)^2(w)}$ with $\gamma_i(w)=(\gamma_{i}^1(w), ..., \gamma_{i}^n(w))^{\top}$.

To study the blowup mechanism of smooth solution to problem \eqref{i-6} with a class of large initial data $w_0(x)=(w_{10}, \cdots, w_{n0})(x)$, motivated by \cite{TSV1}-\cite{TSV3}, we choose $w_0(x)$ as follows:

At first, let $w_{n0}(x)$ satisfy the following generic nondegenerate condition at $x=0$:
\begin{equation}\label{i-8}
w_{n0}(0)=\kappa_0\varepsilon^{\frac{1}{3}},\ w_{n0}'(0)=-\frac{1}{\varepsilon}=\min\limits_{x\in\mathbb{R}}w_{n0}'(x),\ w_{n0}''(0)=0,\ w_{n0}'''(0)=\frac{6}{\varepsilon^4},
\end{equation}
where $\kappa_0$ is a suitable constant.

Secondly, in order to derive $L^{\infty}$ estimates for the lower order derivatives of $w$ and track the development of possible singularity, we require such assumptions
of $w_{0}(x)$:
\begin{subequations}\label{i-9}\begin{align}
&|w_{n0}(x)-w_{n0}(0)|\leq 2\varepsilon^{\frac{1}{2}-\frac{1}{30}},\label{i-9a}\\
&|\hat w(x)|\leq \varepsilon^{\frac{3}{2}}\eta^{\frac{1}{6}}(\varepsilon^{-\frac{3}{2}}x),\ \ |\hat w'(x)|\leq \eta^{-\frac{1}{3}}(\varepsilon^{-\frac{3}{2}}x)\quad \text{for $|x|\leq\mathcal{L}\varepsilon^{\frac{3}{2}}$},\label{i-9b}\\
&|\hat w^{(4)}(x)|\leq \varepsilon^{\frac{1}{9}-\frac{9}{2}}\quad\text{for $|x|\leq \varepsilon^{\frac{3}{2}}$},\label{i-9e}\\
&\varepsilon|w_{n0}'(x)|\leq 2\eta^{-\frac{1}{3}}(\varepsilon^{-\frac{3}{2}}x)\quad\text{for $\mathcal{L}\varepsilon^{\frac{3}{2}}\leq|x|\leq 2\mathcal{L}\varepsilon^{\frac{3}{2}}$},\label{i-9c}\\
&\varepsilon|w_{n0}'(x)|\leq \eta^{-1}(\varepsilon^{-\frac{3}{2}}x)\quad\text{for $|x|\geq 2\mathcal{L}\varepsilon^{\frac{3}{2}}$},\label{i-9d}
\end{align}
\end{subequations}
and for $1\leq j\leq n-1$,
\begin{equation}\label{i-10}
|w_{j0}(x)|\leq\varepsilon,\ |w_{j0}'(x)|\leq \eta^{-\frac{1}{3}}(\varepsilon^{-\frac{3}{2}}x),\ |w_{j0}''(x)|\leq \varepsilon^{-\frac{11}{6}}\eta^{-\frac{1}{3}}(\varepsilon^{-\frac{3}{2}}x),
\end{equation}
where $\mathcal{L}=\varepsilon^{-\frac{1}{10}}, \eta(x)=1+x^2$
and $\hat{w}(x)=w_{n0}(x)-\kappa_0\varepsilon^{\frac{1}{3}}-\overline{w}(x)$ with $\overline{w}(x)=\varepsilon^{\frac{1}{2}}\overline{W}(\varepsilon^{-\frac{3}{2}}x)$ and
$\overline{W}(y)=(-\frac{y}{2}+(\frac{1}{27}+\frac{y^2}{4})^{\frac{1}{2}})^{\frac{1}{3}}
-(\frac{y}{2}+(\frac{1}{27}+\frac{y^2}{4})^{\frac{1}{2}})^{\frac{1}{3}}$.

Thirdly, in order to derive the $L^{2}-$energy estimates for the $\mu_0-$order derivatives of $w$,
we demand that
\begin{equation}\label{i-11}
\sum\limits_{j=1}^{n-1}\|\partial_x^{\mu_0}w_{j0}(\cdot)\|_{L^2}+\varepsilon\|\partial_x^{\mu_0}w_{n0}(\cdot)\|_{L^2}\lesssim \varepsilon^{\frac{3}{2}(1-\mu_0)},
\footnote{Hereafter, $A\lesssim B$ means that there exists a generic positive constant $C$ such that $A\leq CB$.}
\end{equation}
where $\mu_0\ge 6$ is a suitably given constant.

Our main results are stated as:
\begin{thm}\label{thmi-1} {\it Under the conditions \eqref{i-7}, and without loss of generality,
$\mu_n(0)=0$ and $\partial_{w_n}\mu_n(0)=1$ are assumed, then there exists a positive constant $\varepsilon_0$
such that when $0<\varepsilon<\varepsilon_0$ and $w_0(x)$ satisfies \eqref{i-8}-\eqref{i-11},
the problem \eqref{i-6} admits a unique local smooth solution $w$, which will firstly
blow up at the point $(x^*, T^*)$.
Moreover,
\begin{enumerate}[$(1)$]
\item $x^{*}=O(\varepsilon^2),\ T^*=O(\varepsilon^{\frac{4}{3}}).$
\item $w$ lies in the following spaces:\begin{equation}\label{i-12}\begin{cases}
w\in C([-\varepsilon, T^*), H^{\mu_0}(\mathbb{R}))\cap C^1([-\varepsilon, T^*), H^{\mu_0-1}(\mathbb{R})),\\[2mm]
w_i\in C^1([-\varepsilon, T^*]\times \mathbb{R})\ (1\leq i\leq n-1),\  w_n\in L^{\infty}([0, T^*], C^{\frac{1}{3}}(\mathbb{R})).
\end{cases}
\end{equation}

\item There exist two smooth functions $\xi(t)$ and $\tau(t)$ such that
\begin{equation}\label{i-13}\begin{cases}
\lim\limits_{t\nearrow T^*}(\xi(t), \tau(t))=(x^*, T^*),\ \lim\limits_{t\nearrow T^*}\partial_x w_n(\xi(t), t)=-\infty,\\[1mm]
-2<(T^*-t)\partial_x w_n(\xi(t), t)<-\frac{1}{2}\quad\text{for $-\varepsilon\leq t<T^*$},\\[1mm]
\text{$|\tau(t)-T^*|\lesssim \varepsilon^{\frac{1}{3}}(T^*-t)$ and $|\xi(t)-x^*|\lesssim \varepsilon (T^*-t)$
for $-\varepsilon\leq t<T^*$}.
\end{cases}
\end{equation}

\end{enumerate}}

\end{thm}

\begin{rem}\label{remi-H1} {\it By Theorem \ref{thmi-1}, we know that the solution $w\in C(\Bbb R\times [0, T^*])$ of \eqref{i-6a}
blows up at the point $(x^*, T^*)$, i.e., $\lim\limits_{t\nearrow T^*}\|\p_{t,x}w(\cdot, t)\|_{C(\Bbb R)}=\infty$.
This corresponds to the geometric blowup for the problem \eqref{i-6}.}
\end{rem}

\begin{rem}\label{rem1-1} {\it The assumptions of $\mu_n(0)=0$ and $\partial_{w_n}\mu_n(0)=1$ in Theorem \ref{thmi-1}
can be realized by the translation $(t, x)\mapsto (t, x+\mu_n(0)t)$ and then the spatial scaling  $x\mapsto \partial_{w_n}\mu_n(0) x$.}
\end{rem}

\begin{rem}\label{rem1-2} {\it We now give some comments on \eqref{i-6}-\eqref{i-7}.
It is not difficult to find that there are  a great number of $w_0(x)$ to fulfill the constrains \eqref{i-8}-\eqref{i-11}. In addition,
it follows from Proposition \ref{lemA-1} that  the unknown $w$
admits the good components $(w_1, \cdots, w_{n-1})$ and the bad component $w_n$.  The conditions \eqref{i-9b}-\eqref{i-9c} imply that the bad component $w_n$  mainly tracks the possible singularity and it can be thought as a suitable perturbation of the
singular function $\overline{W}$. On the other hand, in order to control the detailed behaviors of $w_n$
near the possible blowup point, we posed the suitable perturbation for the fourth order derivatives of $\hat w$ in \eqref{i-9e} when $|x|\leq \varepsilon^{\frac{3}{2}}$. The conditions \eqref{i-9a} and \eqref{i-9d} are posed to control the behavior of $w_n$
 away from the blowup position. In addition, to avoid the influence of the initial data $w_0(x)$ at infinity,
 we naturally pose the appropriate decaying condition \eqref{i-10}-\eqref{i-11} for large $|x|$.}
\end{rem}

\begin{rem}\label{remi-3} {\it In \cite{TSV1}-\cite{TSV3}, through  introducing
suitable modulated coordinates and taking the constructive proofs, the authors  systematically study the shock formation
of multidimensional compressible Euler equations with a class of smooth initial data.  Motivated by these papers,
we study the geometric blowup mechanism of problem \eqref{i-1}, whose nonlinear structure is more general than
the 1D compressible Euler equations. Thanks to the new reformulation  in the equivalent problem \eqref{i-6} as well as \eqref{i-7},
we can establish some  suitable  exponential-growth  controls on the bounds of the characteristics corresponding to
$\mu_i(w)\ (1\leq i\leq n-1)$  (see Lemma \ref{lem7-1} and Lemma \ref{lem7-2} below) such that
the problem \eqref{i-6} can be mainly dominated by the approximate Burgers equation of $w_n$.}
\end{rem}

\begin{rem}\label{remi-H2}
{\it When \eqref{i-6a} admits the structure of conservation laws, there are some interesting works on the shock construction through the first-in-time blowup point $(x^*, T^*)$ for $t\ge T^*$.
For instances,  under various nondegenerate conditions with finite orders or infinitely degenerate conditions
for the initial data, the shock construction from the blowup point is completed for the 1D scalar equation
$\p_tu+\p_x(f(u))=0$ in  \cite{YL1}; under the generic nondegenerate condition of initial data,
the shock surface from the blowup curve has been constructed for the multidimensional scalar equation
$\p_tu+\p_1(f_1(u))+\cdot\cdot\cdot+\p_n(f_n(u))=0$ in \cite{YL2}; under the generic nondegenerate conditions
of the initial data, for the 1-D $2\times 2$ $p-$ system of
polytropic gases, the authors in \cite{K1},\cite{LB} and \cite{CD} obtain the formation and  construction
of the shock wave starting from the blowup point under some variant conditions; for the 1-D $3\times 3$
strictly hyperbolic conservation laws with the small initial data or the 3-D full compressible Euler equations with
symmetric structure and small perturbed initial data,
the authors in \cite{CXY}, \cite{YHC} and \cite{CD2} also get the formation and  construction
of the resulting shock waves, respectively. In the near future, we hope that the shock formation can
be constructed from the blowup point $(x^*, T^*)$ in Theorem \ref{thmi-1} when \eqref{i-6a} has
the  structure of conservation laws.}
\end{rem}

\begin{rem}\label{remi-2} {\it In the recent years, the formation of shock waves have made much progress
 for the multi-dimensional Euler equations and the quasilinear wave equations under various restrictions
 on the related initial data. One can see
 the remarkable articles \cite{TSV1}-\cite{TSV4}, \cite{CD1}, \cite{CD3}-\cite{CD4}, \cite{LS},
\cite{MY} and \cite{SJ}.}
\end{rem}

\subsection{Comments on the proof of Theorem \ref{thmi-1}}
Let us give comments on the proof of Theorem \ref{thmi-1}.
Motivated by \cite{TSV3}, by introducing the modulated coordinate which is smooth before the singularity formation,
we can convert the finite time singularity formation of \eqref{i-6} into the global well-posedness of smooth
solutions to the resulting new system of $W=(W_1, \cdot\cdot\cdot, W_n)^{\top}$ (see \eqref{ii-4}-\eqref{ii-5}).
To achieve this aim, we take the following strategies:

{\bf $\bullet$} Due to the important form of \eqref{i-6a} with \eqref{i-7}, we divide  $W$ as the $n-1$ good
components $(W_1, \cdots, W_{n-1})^{\top}$
and the bad unknown $w_n$. Inspired by \cite{TSV2}, we continue to decompose $w_n$ into another bad part $W_0$ and a good part $\kappa(t)$
(see \eqref{ii-3}).
The $L^{\infty}$ estimates for the lower order derivatives of $W_0$ are carried out in two different domains  $\{(y, s): |y|\leq \mathcal{L}\varepsilon^{\frac{1}{4}}e^{\frac{s}{4}}\}$ and $\{(y, s): |y|\geq \mathcal{L}\varepsilon^{\frac{1}{4}}e^{\frac{s}{4}}\}$ with $\mathcal{L}=\varepsilon^{-\frac{1}{10}}$. In the interior domain $\{(y, s): |y|\leq \mathcal{L}\varepsilon^{\frac{1}{4}}e^{\frac{s}{4}}\}$,
$W_0$ is expected to have the similar behavior as $\overline{W}(y)$, which is the steady solution of 1D Burgers type equation
$(\partial_s-\frac{1}{2})\overline{W}+\left(\frac{3}{2}y+\overline{W}\right)\partial_y\overline{W}=0$.
In the exterior domain $\{(y, s): |y|\geq \mathcal{L}\varepsilon^{\frac{1}{4}}e^{\frac{s}{4}}\}$, the treatment of $W_0$
is rather delicate since the temporal and spatial decay estimates of good components $(W_1, \cdots, W_{n-1})^{\top}$ are
required to be established simultaneously.

{\bf $\bullet$}  Due to the partially decoupling form of \eqref{i-6a}, in order to
prove Theorem \ref{thmi-1},  we need to  establish the $L^{\infty}$
estimates of the lower order derivatives  and the $L^2$ estimates of the
highest order derivatives for  $(W_1, \cdots, W_{n})^{\top}$. To get the related $L^{\infty}$ estimates,
by utilizing the characteristics method and delicate analysis, at first, we derive
the basic  exponential controls on  the bounds of the characteristics corresponding to
$\mu_i(w)\ (1\leq i\leq n-1)$. Subsequently,  the spatial decay rate $(1+y^2)^{-\frac{1}{3}}$ of $\partial_y w_0$
and further the temporal decay of $\partial_y W_0$ are obtained. From these, the $L^{\infty}$ estimates of $(W_1, \cdots, W_{n})^{\top}$
are achieved. On the other hand, we observe that the coefficients in the equations of $w$
admit the key $O(e^{\frac{s}{2}})$ scale because of the strict hyperbolicity of \eqref{i-6a} (see \eqref{v-35}).
This will lead to the expected $L^2$ estimates on the
highest order derivatives of  $(W_1, \cdots, W_{n})^{\top}$. Here, we specially point out that the $L^{\infty}$
estimates of each related quantity depend on the information of the higher order derivatives of $W$
since  the related \eqref{i-6a} only admits the partial decoupling form.
This is the main reason to apply the $L^2$ estimates for dealing with the highest order derivatives of $W$.

When these are done, the proof of Theorem \ref{thmi-1} can be  completed successfully.  It is hoped that
our analysis methods in the paper will be adopted to study the singularity formation problem
for the general multi-dimensional symmetric hyperbolic systems with some classes of large initial data,
which is a generalization of the results in \cite{TSV1}-\cite{TSV4} for the  multi-dimensional compressible
Euler equations.

The rest of the paper is arranged as follows: In Section \ref{II}, we reduce the problem \eqref{i-1} into
the equivalent partially decoupling problem \eqref{i-6} via Proposition \ref{lemA-1}.  In Section \ref{ii},
under the modulated coordinate, the problem \eqref{i-6} and the choice of initial data $w_0(x)$
are reformulated. Moreover, as the heuristics of the formation of the expected singularity,
the rigorous derivation on the resulting Burgers-type equation is also given in this section.
The bootstrapping assumptions and their closure of the arguments are
arranged in Section \ref{iv}-Section \ref{v} respectively: The descriptions of bootstrapping assumptions
on $w$ and the modulated coordinate are made in Section \ref{iv};
the $L^{\infty}$ estimates for the bad unknown $W_n$ and the good components $(W_1,\cdot\cdot\cdot,
W_{n-1})^{\top}$ are taken in Section \ref{vi}-Section \ref{vii} respectively. In addition,
the closure of bootstrapping assumptions for the modulation variables is completed in Section \ref{viii};
the related energy estimates for the higher order derivatives of $W$ are derived in Section \ref{v}.
In Section \ref{V}, we establish the main results in Theorem \ref{thmii-1} and further Theorem \ref{thmi-1}.
Finally, a useful interpolation inequality and its application for deriving some delicate estimates
are given in Appendix \ref{A}.

\section{Reduction}\label{II}

In the section, our main aim is to reduce \eqref{i-1a} to a partially decoupling form \eqref{i-6a}
such that the resulting new unknown functions $w=(w_1, \cdot\cdot\cdot, w_n)^{\top}$
will admit $n-1$ good components and only one bad component. The good  component
and the bad  component mean that their regularities are in  $C^1$ and  in  $C^{1/3}$
up to the blowup time, respectively.

\begin{prop}\label{lemA-1} {\it Under assumptions \eqref{i-2}-\eqref{i-3}, there exists a constant $\delta_0>0$
such that when $|u|<\delta_0$, the system \eqref{i-1a} can be equivalently reduced  into
\begin{equation}\label{A-1}
\partial_{t}w+A(w)\partial_{x} w=0,
\end{equation}
where the smooth mapping $u\mapsto w=w(u)$ is invertible and $w(0)=0$.
In addition, the inverse mapping of $w(u)$ is denoted as $u=u(w)$, and the $n\times n$ matrix
\begin{equation*}
A(w)=\left(\frac{\partial w}{\partial u}\right)F(u(w))\left(\frac{\partial w}{\partial u}\right)^{-1}:=\left(a_{ij}(w)\right)_{n\times n}
\end{equation*} satisfies
\begin{enumerate}[$(1)$]
\item $A(w)$ has $n$ distinct eigenvalues $\left\{\mu_i(w)\right\}_{i=1}^{n}$ with
\begin{equation*}
\mu_i(w)=\lambda_i(u(w))\ (1\leq i<i_0);\ \mu_i(w)=\lambda_{i+1}(u(w))\ (i_0\leq i<n);\ \mu_n(w)=\lambda_{i_0}(u(w)).
\end{equation*}
\item $a_{in}(w)=0\ (i\neq n),\ a_{nn}(w)=\mu_{n}(w)$ and $\partial_{w_{n}}\mu_{n}(w)\neq 0$.
\item $A(0)=\text{diag}\{\mu_1(0), \cdots, \mu_n(0)\}$.
\end{enumerate}}

\end{prop}

\begin{proof} At first, we claim that when  $|u|\leq \delta_0$ for some constant $\delta_0>0$, there exist $(n-1)$
linearly independent Riemann invariants $\alpha_i(u)\ (i\neq i_0)$ corresponding to
$\lambda_{i_0}(u)$ such that
\begin{equation}\label{B-2}
\nabla_u \alpha_i(u)\cdot \gamma_{i_0}(u)=0\ (i\neq i_0,\ |u|<\delta_0).
\end{equation}
Indeed, let $\{\zeta_i\}_{(i\neq i_0)}$ be $(n-1)$ linearly independent column vectors orthogonal to $\gamma_{i_0}(0)$ and set $\alpha_i(u)=\zeta_i^{\top}\cdot u+\bar{\alpha}_i(u)$, then it follows from \eqref{B-2} that the
unknowns $\{\bar{\alpha}_i(u)\}_{i\neq i_0}$ should satisfy
\begin{equation}\label{B-3}
\nabla_u \bar{\alpha}_i(u)\cdot\gamma_{i_0}(u)=-\zeta_i^{\top}\cdot(\gamma_{i_0}(u)-\gamma_{i_0}(0)),\ \bar{\alpha}_i(0)=0\ (i\neq i_0).
\end{equation}
It is not difficult to find that there exists a constant $\delta_0>0$ such that \eqref{B-3} is uniquely solved when $|u|<\delta_0$ and $|\bar{\alpha}_i(u)|\lesssim |u|^2$. Hence \eqref{B-2} is obtained.

By $\alpha_i(0)=0\ (i\neq i_0)$, we define a mapping $u\mapsto v=v(u)=(v_1, \cdots, v_n)^{\top}(u)$ with $v(0)=0$ as
\begin{equation}\label{B-4}
v_i=\alpha_i(u)\ (1\leq i< i_0);\quad  v_i=\alpha_{i+1}(u)\ (i_0\leq i<n);\quad v_{n}=\gamma_{i_0}^{\top}(0)\cdot u.
\end{equation}
Note that $\{\zeta_i\}_{i\neq i_0}$ are $(n-1)$ linearly independent column vectors
which are orthogonal to $\gamma_{i_0}(0)$. Then the transformation $u\mapsto v=v(u)$ is reversible for $|u|<\delta_0$
since its Jacobian matrix $J_{v}(u)$ satisfies $J_{v}(0)=\left(\frac{\partial v}{\partial u}\right)|_{u=0}
=\left(\zeta_1,\cdots, \zeta_{i_0-1}, \zeta_{i_0+1}, \cdots, \zeta_n, \gamma_{i_0}(0)\right)^{\top}$ and $J_{v}(0)$
is non-singular. We now denote the inverse mapping of $v=v(u)$ as $u=u(v)$.

By \eqref{B-2} and \eqref{B-4}, the system \eqref{i-1a} is equivalently converted into
\begin{equation}\label{B-5}
\partial_t v+G(v)\partial_x v=0,
\end{equation}
where $G(v)=J_{v}(u)F(u(v))J_{v}^{-1}(u):=\left(g_{ij}(v)\right)_{n\times n}$,
$G(v)$ has $n$ distinct eigenvalues $\{\lambda_i(u(v))\}_{i=1}^{n}$ and the corresponding right
eigenvectors are $\{J_{v}(u)\gamma_i(u(v))\}_{i=1}^{n}$. In addition, \eqref{B-2} shows
\begin{equation}\label{A-6}
J_{v}(u)\gamma_{i_0}(u(v))=\gamma_{i_0}^{\top}(0)\cdot\gamma_{i_0}(u(v)){\bf e}_{n}\neq {\bf 0}.
\end{equation}
This implies that
\begin{equation}\label{A-7}
g_{i n}(v)=0\ (i\neq n),\quad g_{n n }(v)=\lambda_{i_0}(u(v)).
\end{equation}
It follows from \eqref{B-5} and \eqref{A-7} that the $(n-1)$ order square matrix
\begin{equation}\label{A-9}
G_{n-1}(v)=\left(g_{ij}(v)\right)_{(n-1)\times (n-1)}
\end{equation}
has $(n-1)$ eigenvalues
\begin{equation*}\label{A-10}
\lambda_1(u(v))<\cdots<\lambda_{i_0-1}(u(v))<\lambda_{i_0+1}(u(v))<\cdots<\lambda_n(u(v)).
\end{equation*}
Then there exists a unique $(n-1)$ order invertible constant square matrix $B_{n-1}=\left(b_{ij}\right)_{(n-1)\times (n-1)}$
such that
\begin{equation}\label{B-10}
\bar{G}_{n-1}(v)=B_{n-1}G_{n-1}(v)B_{n-1}^{-1}:=\left(\bar{g}_{ij}(v)\right)_{(n-1)\times (n-1)},
\end{equation}
where
\begin{equation*}
\bar{G}_{n-1}(0)=\text{diag}\{\lambda_1(0),\cdots, \lambda_{i_0-1}(0), \lambda_{i_0+1}(0), \cdots, \lambda_n(0)\}.
\end{equation*}
Furthermore, it is derived from \eqref{i-3} and \eqref{A-6}-\eqref{A-7} that
\begin{equation}\label{A-8}
\nabla_{v}\lambda_{i_0}(u(v))J_{v}(u)\gamma_{i_0}(u(v))=\nabla_{u}\lambda_{i_0}(u)\cdot\gamma_{i_0}(u)\neq 0.
\end{equation}
Denote the invertible transformation $v\mapsto w=w(v):=(w_1, \cdots, w_n)^{\top}(v)$ as
\begin{equation}\label{A-11}
(w_1, \cdots, w_{n-1})^{\top}=B_{n-1}(v_1, \cdots, v_{n-1})^{\top},\quad w_n=v_n+\sum\limits_{j=1}^{n-1}b_{nj}w_j,
\end{equation}
where the constants $\{b_{nj}\}_{j=1}^{n-1}$ will be determined later. Set the inverse mapping of $w=w(v)$ as $v=v(w)$ with $v(0)=0$.
By \eqref{B-5}, \eqref{A-9}-\eqref{A-11} and a direct computation, we arrive at
\begin{equation}\label{A-12}
\partial_t w+A(w)\partial_x w=0,
\end{equation}
where $A(w)=\left(a_{ij}(w)\right)_{n\times n}$ satisfies
\begin{equation}\label{A-13}\begin{cases}
a_{ij}(w)=\bar{g}_{ij}(v(w))\ (1\leq i, j\leq n-1),\\[2mm]
a_{jn}(w)=0\ (1\leq j\leq n-1), \\[2mm]
a_{nn}(w)=g_{nn}(v(w))
\end{cases}
\end{equation}
and
\begin{equation}\label{A-14}\begin{aligned}
&(a_{n 1}, \cdots, a_{n n-1})(w)\\
=&(g_{n1}, \cdots, g_{n n-1})(v(w))B_{n-1}
+(b_{n1 }, \cdots, b_{n n-1})(\bar{G}_{n-1}-g_{n n}I_{n-1})(v(w)).
\end{aligned}
\end{equation}
Since the $(n-1)$ order square matrix
\begin{equation*}
\bar{G}_{n-1}(0)-g_{n n}(0)I_{n-1}=\text{diag}\{\lambda_1(0)-\lambda_{i_0}(0), \cdots \lambda_{i_0-1}(0)-\lambda_{i_0}(0), \lambda_{i_0+1}(0)-\lambda_{i_0}(0), \cdots, \lambda_{n}(0)-\lambda_{i_0}(0)\}
\end{equation*}
is invertible due to \eqref{i-2} and \eqref{A-7}, it is derived from \eqref{A-14} that there exists unique $\{b_{n j}\}_{1\leq j\leq n-1}$
such that
\begin{equation}\label{A-15}
(a_{n 1}, \cdots, a_{n, n-1})(0)=(0, \cdots, 0).
\end{equation}

For the constant invertible  square matrix $B=\left(b_{ij}\right)_{n\times n}$ with $b_{i n}=\delta_{i}^{n}\ (1\leq i\leq n)$ and
\begin{equation}\label{A-16}
w=Bz,
\end{equation}
one has from \eqref{A-11}-\eqref{A-14} that
\begin{equation}\label{A-17}
A(w)=BG(B^{-1}w)B^{-1}:=(a_{ij}(w))_{n\times n},
\end{equation}
where
\begin{equation*}
A(0)=BG(0)B^{-1}=\text{diag}\{\lambda_1(0), \cdots, \lambda_{i_0-1}(0), \lambda_{i_0+1}(0), \cdots, \lambda_n(0), \lambda_{i_0}(0)\}.
\end{equation*}

The expected invertible mapping $u\mapsto w=w(u)$ is just the composition of two mappings $v=v(u)$ and $w=w(v)$ defined by \eqref{B-4} and \eqref{A-16} respectively. Its inverse mapping is denoted as $u=u(w)$. It is easy to know that the matrix $A(w)$ has $n$
distant eigenvalues $\{\lambda_i(w(u))\}_{i=1}^n$ and the corresponding right eigenvectors are $\{J_w(u)\gamma_i(u(w))\}_{i=1}^n$.
In addition, it is derived from \eqref{A-6}, \eqref{A-8} and \eqref{A-11} that
\begin{equation}\label{A-18}
J_w(u)\gamma_{i_0}(u(w))=B J_v(u)\gamma_{i_0}(u(v))=\gamma_{i_0}^{\top}(0)\cdot\gamma_{i_0}(u(v)){\bf e}_n\neq {\bf 0}
\end{equation}
and
\begin{equation}\label{A-19}\begin{aligned}
&\nabla_w\lambda_{i_0}(u(w))\cdot J_w(u)\gamma_{i_0}(u(w))\\
=&\nabla_v\lambda_{i_0}(u(v))\cdot J_v(w)\cdot J_w(v)\cdot J_v(u)\gamma_{i_0}(u(v))\\
=&\nabla_v\lambda_{i_0}(u(v))\cdot J_v(u)\gamma_{i_0}(u(v))\neq 0.
\end{aligned}
\end{equation}
Then the properties (1)-(3) of $A(w)$ come from \eqref{A-9}-\eqref{B-10}, \eqref{A-13}, \eqref{A-15}
and \eqref{A-18}-\eqref{A-19}.
\end{proof}

\begin{rem}\label{YH-1} {\it We point out that Proposition \ref{lemA-1} is a generalization of Lemma 2.1
in \cite{CXY}, where \eqref{i-1} with small initial data and $n=3$ is simplified analogously.}
\end{rem}

\section{Reformulation under the modulated coordinates}\label{ii}

Motivated by \cite{TSV2}, to show the geometric blowup mechanism in Theorem \ref{thmi-1}, we introduce
three modulation variables $\tau(t),\ \xi(t)$ and $\kappa(t)$ as
\begin{equation}\label{ii-1}\begin{cases}
\tau(t):\quad \text{tracking the exact blowup time},\\
\xi(t): \quad \text{tracking the location of the blowup point},\\
\kappa(t):\quad \text{fixing the speed of the singularity development}.
\end{cases}
\end{equation}
Set the modulated coordinate $(y, s)$ as follows
\begin{equation}\label{ii-2}
s=s(t)=-\log(\tau(t)-t),\quad y=(x-\xi(t))(\tau(t)-t)^{-\frac{3}{2}}=(x-\xi(t))e^{\frac{3s}{2}}.
\end{equation}
In addition,  the new unknowns $W=(W_1, \cdots, W_n)^{\top}$ and $W_0$ are defined as
\begin{equation}\label{ii-3}
w_i(x, t)=W_i(y, s) (i\neq n),\quad w_{n}(x, t)=W_n(y, s)=e^{-\frac{s}{2}}W_{0}(y, s)+\kappa(t),
\end{equation}
where $\kappa(t)=w_n(\xi(t),t)$.

By \eqref{ii-1}-\eqref{ii-3},   when $t<\tau(t)$, the system \eqref{i-6a} can be equivalently rewritten as
\begin{equation}\label{ii-4}
\partial_s W+\left(\frac{3}{2}y-e^{\frac{s}{2}}\beta_{\tau}\dot{\xi}(t)\right)\partial_y W
+e^{\frac{s}{2}}\beta_{\tau}A(w)\partial_y W=0,
\end{equation}
where $\beta_{\tau}(t)=\frac{1}{1-\dot{\tau}(t)}$, $\dot{\xi}(t)=\xi'(t)$ and $\dot{\tau}(t)={\tau}'(t)$.

In addition, $W_0$ is determined by
\begin{equation}\label{ii-5}
(\partial_s-\frac{1}{2})W_{0}+\left(\frac{3}{2}y+e^{\frac{s}{2}}\beta_{\tau}(\mu_{n}(w)
-\dot{\xi}(t))\right)\partial_y W_{0}+\sum\limits_{j\neq n}e^{s}\beta_{\tau}a_{nj}(w)\partial_y W_j
=-e^{-\frac{s}{2}}\beta_{\tau}\dot{\kappa}(t),
\end{equation}
where $\dot{\kappa}(t)=\kappa'(t)$.

\subsection{Global steady solution for 1D Burgers equation}\label{ii-a}

The simplest case of problem \eqref{i-6} is (i.e., $n=1$ and $a_n(w)=w_n$)
\begin{equation}\label{A-1}\begin{cases}
\partial_t\overline{\omega}+\overline{\omega}\partial_x\overline{\omega}=0,\ x\in\mathbb{R}, \ t>-\varepsilon,\\
\overline{\omega}(x, -\varepsilon)=\overline{\omega}_0(x),\ x\in\mathbb{R},
\end{cases}
\end{equation}
where $\overline{\omega}_0(x)\in C^{\infty}(\mathbb{R})$ and $\overline{\omega}_0'(x)\leq 0$.
It is assumed that $\overline{\omega}_0(x)$ satisfies the generic nondegenerate condition at $x=0$:
\begin{equation}\label{A-2}
\overline{\omega}_0(0)=0,\ \overline{\omega}_0'(0)=-\frac{1}{\varepsilon}=\min\limits_{x\in\mathbb{R}}\overline{\omega}_0'(x),\ \overline{\omega}_0''(0)=0,\ \overline{\omega}_0'''(0)=\frac{6}{\varepsilon^4}.
\end{equation}
By the  characteristics method, it is easy to know that under the assumption \eqref{A-2},
the smooth solution $\overline{\omega}(x, t)$ of problem \eqref{A-1} will blowup at the first-in-time
singularity point $(0, 0)$ and the related characteristics starting from
the point $(0, -\varepsilon)$ is $\{(x, t): x=0, \ -\varepsilon<t<0\}$.
From this and the procedures in \eqref{ii-1}-\eqref{ii-5}, we define
\begin{equation}\label{A-3}
\tau_0(t)=0,\ \xi_0(t)=0,\ \kappa_0(t)=0,
\end{equation}
and
\begin{equation}\label{A-4}
s=s(t)=-\log(\tau_0(t)-t),\ y=(x-\xi_0(t))(\tau_0(t)-t)^{-\frac{3}{2}}=(x-\xi_0(t))e^{\frac{3s}{2}}
\end{equation}
and
\begin{equation}\label{A-5}
\overline{\omega}=e^{-\frac{s}{2}}\overline{W}(y, s)+\kappa_0(t).
\end{equation}

Then it follows from \eqref{A-1} and \eqref{A-3}-\eqref{A-5} that $\overline{W}(y, s)$ satisfies
\begin{equation}\label{ii-6}
(\partial_s-\frac{1}{2})\overline{W}+\left(\frac{3}{2}y+\overline{W}\right)\partial_y\overline{W}=0.
\end{equation}

In Appendix A.1 of \cite{TSV2}, it is proved that equation \eqref{ii-6} has a group of steady
smooth solutions $\overline{W}=\overline{W}(y)$ satisfying such a generic nondegenerate condition
\begin{equation}\label{ii-7}
\overline{W}(0)=0,\ \overline{W}'(0)=\min\limits_{y\in\mathbb{R}}\overline{W}'(y)<0,\ \overline{W}''(0)=0,\ \overline{W}'''(0)>0.
\end{equation}
According to the initial data \eqref{A-2} and the transformation \eqref{A-3}-\eqref{A-4},
the solution $\overline{W}(y)$ of \eqref{ii-6} is introduced in \cite{TSV2}
\begin{equation}\label{ii-8}
\overline{W}(y)=\left(-\frac{y}{2}+\left(\frac{1}{27}+\frac{y^2}{4}\right)^{\frac{1}{2}}\right)^{\frac{1}{3}}
-\left(\frac{y}{2}+\left(\frac{1}{27}+\frac{y^2}{4}\right)^{\frac{1}{2}}\right)^{\frac{1}{3}}.
\end{equation}

 In addition, it is easy to obtain
 \begin{subequations}\label{ii-9}\begin{align}
 &\overline{W}(0)=0,\ \ \ \overline{W}'(0)=-1,\ \  \ \overline{W}''(0)=0,\ \ \ \overline{W}'''(0)=6,\label{ii-9a}\\
& \|\eta^{-\frac{1}{6}}\overline{W}\|_{L^{\infty}}\leq 1,\ -1\leq\eta^{\frac{1}{3}}\overline{W}'\leq -\frac{1}{6},\ \|\eta^{\frac{5}{6}}\overline{W}''\|_{L^{\infty}}\leq 2,\label{ii-9b}\\
& \|\eta^{\frac{5}{6}}\overline{W}^{(\mu)}\|_{L^{\infty}}\lesssim_{\mu}1\quad\text{for $\mu\geq 3$}.\label{ii-9c}
 \end{align}
 \end{subequations}

 \subsection{Evolution for the modulation variables}\label{ii-b}
With the expectation $\lim\limits_{s\nearrow +\infty}W_{0}(y, s)=\overline{W}(y)$
and by the properties of $\overline{W}(y)$, we pose
\begin{equation}\label{ii-10}
W_0(0, s)=0,\quad \partial_y W_0(0, s)=-1,\quad\partial_y^2 W_0(0, s)=0,\quad \partial_y^3 W_0(0, s)=6.
\end{equation}

Next, we derive the equations of the modulation variables in \eqref{ii-1}. For any nonnegative integer $\mu$, acting $\partial^{\mu}=\partial_y^{\mu}$ on both sides of \eqref{ii-4} and \eqref{ii-5} yields
\begin{subequations}\label{ii-11}\begin{align}
&(\partial_s+\frac{3}{2}\mu)\partial^{\mu}W+(\frac{3}{2}y-e^{\frac{s}{2}}\beta_{\tau}\dot{\xi}(t))\partial_y\partial^{\mu}W
+e^{\frac{s}{2}}\beta_{\tau}A(w)\partial_y\partial^{\mu}W=F_{\mu},\label{ii-11a}\\[2mm]
&(\partial_s+\frac{3\mu-1}{2})\partial^{\mu}W_0+\left(\frac{3}{2}y+e^{\frac{s}{2}}\beta_{\tau}(\mu_{n}(w)
-\dot{\xi}(t))\right)\partial_y\partial^{\mu}W_0=F_{\mu}^0,\label{ii-11b}
\end{align}
\end{subequations}
where
\begin{equation*}\begin{cases}
F_{\mu}=-e^{\frac{s}{2}}\beta_{\tau}\sum\limits_{1\leq\beta\leq\mu}C_{\mu}^{\beta}\partial^{\beta}A(w)\partial_y\partial^{\mu-\beta}W,\\[2mm]
F_{\mu}^0=-e^{\frac{s}{2}}\beta_{\tau}\sum\limits_{1\leq\beta\leq\mu}C_{\mu}^{\beta}\partial^{\beta}\mu_{n}(w)\partial_y\partial^{\mu-\beta}W_0
+e^{s}\beta_{\tau}\sum\limits_{j\neq n}\partial^{\mu}\left(a_{nj}(w)\partial_yW_j\right)
-e^{-\frac{s}{2}}\beta_{\tau}\dot{\kappa}(t)\delta_{\mu}^0.\\
\end{cases}
\end{equation*}

Due to \eqref{ii-10}, it is derived from \eqref{ii-11b} that for $\mu=0, 1, 2$,
the modulation variables satisfy the following ordinary differential system:
\begin{subequations}\label{ii-12}\begin{align}
\dot{\kappa}(t)&=e^{s}\left(\mu_{n}(w^{0})-\dot{\xi}(t)\right)-e^{\frac{3s}{2}}\sum\limits_{j\neq n}a_{nj}(w^{0})(\partial_y W_j)^0,\label{ii-12a}\\[2mm]
\dot{\tau}(t)&=1-\partial_{w_n}\mu_{n}(w^{0})+e^{\frac{s}{2}}\sum\limits_{j\neq n}\partial_{w_j}\mu_{n}(w^{0})(\partial_y W_j)^0\nonumber\\[2mm]
&\quad -e^{s}\sum\limits_{j\neq n}\left(\partial_y(a_{nj}(w)\partial_y W_j)\right)^0,\label{ii-12b}\\[2mm]
\dot{\xi}(t)&=\mu_{n}(w^{0})-\frac{1}{6}(\partial_y^2  \mu_n(w))^{0}+\frac{1}{6}\sum\limits_{j\neq n}e^{\frac{s}{2}}\left(\partial_y^{2}(a_{nj}(w)\partial_y W_j)\right)^0,\label{ii-12c}
\end{align}
\end{subequations}
where the notation $v^{0}$ represents $v(0, s)$ for the function $v(y, s)$.

\subsection{The equation of $\mathcal{W}=W_{0}-\overline{W}$}\label{ii-c}

Set
\begin{equation}\label{ii-14}
\mathcal{W}=W_{0}-\overline{W}.
\end{equation}
It follows from \eqref{ii-5} and \eqref{ii-6} that for any nonnegative integer $\mu$,
\begin{equation}\label{ii-15}
\left(\partial_{s}+\left(\frac{3}{2}y+e^{\frac{s}{2}}\beta_{\tau}(\mu_{n}(w)
-\dot{\xi}(t)\right)\partial_y\right)\partial^{\mu}\mathcal{W}+\mathcal{D}_{\mu}\partial^{\mu}\mathcal{W}=\mathcal{F}_{\mu},
\end{equation}
where
\begin{equation*}
\mathcal{D}_{\mu}=\frac{3\mu-1}{2}+\beta_{\tau}\overline{W}'+e^{\frac{s}{2}}\beta_{\tau}\mu \partial_y \mu_n(w)\\
\end{equation*}
and
\begin{equation*}\begin{aligned}
\mathcal{F}_{\mu}=&-\sum\limits_{1\leq\beta\leq\mu}C_{\mu}^{\beta}\beta_{\tau}\partial_y^{1+\beta}\overline{W}\partial_y^{\mu-\beta}\mathcal{W}
-\sum\limits_{2\leq\beta\leq \mu}C_{\mu}^{\beta}e^{\frac{s}{2}}\beta_{\tau}\partial_y^{\beta}\mu_n(w)\partial_y^{\mu-\beta+1}\mathcal{W}
-e^{-\frac{s}{2}}\beta_{\tau}\dot{\kappa}(t)\delta_{\mu}^0\\
&-\sum\limits_{j\neq n}e^{s}\beta_{\tau}\partial_y^{\mu}\left(a_{nj}(w)\partial_y W_j\right)-\beta_{\tau}e^{\frac{s}{2}}\partial_y^{\mu}\left(\overline{W}' (\mu_{n}(w)-e^{-\frac{s}{2}}W_0-\dot{\xi}(t))\right)\\
&+(1-\beta_{\tau})\partial_y^{\mu}(\overline{W}' \overline{W}).
\end{aligned}
\end{equation*}

In addition, for the purpose to establish the weighted estimates of $\mathcal{W}$, it is derived from \eqref{ii-15}
that for any real number $\nu$,
\begin{equation}\label{ii-16}
\left(\partial_s+\left(\frac{3}{2}y+e^{\frac{s}{2}}\beta_{\tau}(\mu_{n}(w)
-\dot{\xi}(t))\right)\partial_y\right)\left[\eta^{\nu}\partial^{\mu}\mathcal{W}\right]+\mathcal{D}_{\mu, \nu}\left[\eta^{\nu}\partial^{\mu}\mathcal{W}\right]=\eta^{\nu}\mathcal{F}_{\mu},
\end{equation}
where
\begin{equation*}
\mathcal{D}_{\mu, \nu}=\mathcal{D}_{\mu}-2\nu y\eta^{-1}\left(\frac{3}{2}y+e^{\frac{s}{2}}\beta_{\tau}(\mu_{n}(w)-\dot{\xi}(t))\right).
\end{equation*}

\subsection{The decomposition on the derivatives of $W$}\label{ii-d}

To deal with the derivatives of $W$, we adopt the method of eigendecomposition in \cite{PDL1}.
Set
\begin{equation}\label{ii-17}
\partial_{y}^{\mu}W=\sum\limits_{m=1}^{n}W_{\mu}^{m}\gamma_{m}(w),
\end{equation}
where $\left\{\gamma_{m}(w)\right\}_{m=1}^{n}$ have been defined in \eqref{i-71}.

Acting each left eigenvector $\ell_m(w)$ on both sides of \eqref{ii-11a}
and substituting the expansion \eqref{ii-17} into \eqref{ii-11a} yield
\begin{equation}\label{ii-18}
\left(\partial_s+\left(\frac{3}{2}y+e^{\frac{s}{2}}\beta_{\tau}(\mu_{m}(w)-\dot{\xi}(t))\right)\partial_y +\frac{3}{2}\mu\right)W_{\mu}^{m}=\mathbb{F}_{\mu}^m,
\end{equation}
where
\begin{equation*}
\mathbb{F}_{\mu}^m=-\sum\limits_{j=1}^{n}W_{\mu}^{j}\ell_m(w)\cdot\left(\partial_s\gamma_j(w)
+((\frac{3}{2}y-e^{\frac{s}{2}}\beta_{\tau}\dot{\xi}(t))I_n+e^{\frac{s}{2}}\beta_{\tau}A(w))\partial_y\gamma_j(w)\right)
+\ell_m(w)\cdot F_{\mu}.
\end{equation*}

On the other hand, for any real number $\nu$, it is derived from \eqref{ii-18} that
\begin{equation}\label{ii-19}
\left(\partial_s+\left(\frac{3}{2}y+e^{\frac{s}{2}}\beta_{\tau}(\mu_m(w)
-\dot{\xi}(t))\right)\partial_y\right)\left[\eta^{\nu}W_{\mu}^{m}\right]
+\mathbb{D}_{\mu, \nu}^{m}\left[\eta^{\nu}W_{\mu}^{m}\right]=\eta^{\nu}\mathbb{F}_{\mu}^{m},
\end{equation}
where
\begin{equation*}
\mathbb{D}_{\mu, \nu}^{m}=\frac{3\mu}{2}-2\nu y\eta^{-1}\left(\frac{3}{2}y+e^{\frac{s}{2}}\beta_{\tau}(\mu_m(w)-\dot{\xi}(t)\right).
\end{equation*}

\subsection{Initial data and main results under the modulation coordinates}\label{ii-e}

Under the constrains \eqref{i-8}-\eqref{i-11} and the definition \eqref{ii-3},
the initial data of $W(y, s)$ on $s=-\log\varepsilon$ can be determined accordingly.

Indeed, due to \eqref{i-8} and the definitions \eqref{ii-1} and \eqref{ii-3} (also see \eqref{ii-10}),
the initial data of the modulation variables $\tau(t), \xi(t)$ and $\kappa(t)$ on $t=-\varepsilon$ are
\begin{equation}\label{iii-14}
\tau(-\varepsilon)=0,\ \xi(-\varepsilon)=0,\ \kappa(-\varepsilon)=w_n(0, -\varepsilon)=\kappa_0\varepsilon^{\frac{1}{3}}.
\end{equation}

In addition, for the bad component $W_0(y, s)$, it is derived from \eqref{i-8}-\eqref{i-9} and \eqref{ii-2}-\eqref{ii-3}
that on $s=-\log\varepsilon$,
\begin{equation}\label{iii-6}
W_{0}(0, s)=\partial_{y}^{2}W_{0}(0, s)=0,\ \partial_{y}W_{0}(0, s)=\min\limits_{y\in\mathbb{R}}\partial_{y}W_{0}(y, s)=-1,\  \partial_{y}^{3}W_{0}(0, s)=6
\end{equation}
and
\begin{subequations}\label{iii-7}\begin{align}
&|W_{0}(y, -\log\varepsilon)|\leq 2\varepsilon^{-\frac{1}{30}},\label{iii-7a}\\
&\text{$|\mathcal{W}(y, -\log\varepsilon)|\leq \varepsilon\eta^{\frac{1}{6}}(y)$
and $|\partial_{y}\mathcal{W}(y, -\log\varepsilon)|\leq  \varepsilon\eta^{-\frac{1}{3}}(y)$ for $|y|\leq\mathcal{L}$},\label{iii-7b}\\
&|\partial_y^4 \mathcal{W}(y, -\log\varepsilon)|\leq \varepsilon^{\frac{1}{9}}\quad\text{for $|y|\leq 1$},\label{iii-7d}\\
&|\partial_y W_0(y, -\log\varepsilon)|\leq 2\eta^{-\frac{1}{3}}(y)\boldsymbol{1}_{\{\mathcal{L}\leq |y|\leq 2\mathcal{L}\}}+\eta^{-1}(y)\boldsymbol{1}_{\{|y|\geq 2\mathcal{L}\}}\quad\text{for $|y|\geq\mathcal{L}$}.\label{iii-7c}
\end{align}
\end{subequations}

For the good components of $W$, it follows from \eqref{i-10} and \eqref{ii-2}-\eqref{ii-3} that for $j\neq n$
\begin{subequations}\label{iii-8}\begin{align}
&|W_{j}(y, -\log\varepsilon)|\leq \varepsilon,\label{iii-8a}\\
&|\partial_y W_j(y, -\log\varepsilon)|\leq \varepsilon^{\frac{3}{2}-3\nu} \eta^{-\nu}(y)\ (\nu\in [0, \frac{1}{3}]),\label{iii-8b}\\
&|\partial_y^2 W_j(y, -\log\varepsilon)|\leq \varepsilon^{\frac{7}{6}}\eta^{-\frac{1}{3}}(y).\label{iii-8c}
\end{align}
\end{subequations}

Following \eqref{i-11}, the initial energy of $\partial_y^{\mu_0}W$ on $s=-\log\varepsilon$ satisfies
\begin{equation}\label{iii-9}
\sum\limits_{j=1}^{n-1}\|\partial_{y}^{\mu_0}W_{j}(\cdot, -\log\varepsilon)\|_{L^{2}(\mathbb{R})}
+\varepsilon^{\frac{3}{2}}\|\partial_{y}^{\mu_0}W_{0}(\cdot, -\log\varepsilon)\|_{L^{2}(\mathbb{R})}\lesssim \varepsilon^{\frac{3}{2}}.
\end{equation}

Under the preparations above, the new version of Theorem \ref{thmi-1} under the modulated coordinates can be stated as:
\begin{thm}\label{thmii-1} {\it Under the conditions in \eqref{i-7} and the notations in \eqref{ii-1}-\eqref{ii-3}, there exists a positive constant $\varepsilon_0$ such that when $0<\varepsilon<\varepsilon_0$, the system \eqref{ii-4}-\eqref{ii-5} and \eqref{ii-12} with the initial data satisfying \eqref{iii-14}-\eqref{iii-9} has a global-in-time solution $W$ and $\tau(t), \xi(t), \kappa(t)$, which satisfy
\begin{enumerate}[$(1)$]
\item $\lim\limits_{s\to +\infty}\xi(t)=x^*=O(\varepsilon^2), \lim\limits_{s\to +\infty}\tau(t)=T^*=O(\varepsilon^{\frac{4}{3}})$.

\item $|\dot{\tau}(t)|\lesssim \varepsilon^{\frac{1}{3}},\ |\dot{\xi}(t)|\lesssim \varepsilon,\ |\dot{\kappa}(t)|\lesssim 1$, and $|\tau(t)|\lesssim\varepsilon^{\frac{4}{3}},\ |\xi(t)|\lesssim \varepsilon^2,\ |\kappa(t)-\kappa_0\varepsilon^{\frac{1}{3}}|\lesssim\varepsilon$.

\item  $|W_0(y, s)|\lesssim\varepsilon^{\frac{1}{3}}e^{\frac{s}{2}}, |W_i(y, s)|\lesssim\varepsilon\ (1\leq i\leq n-1)$.

\item With respect to $W_0$,
\begin{equation*}\begin{cases}
\text{$|W_0(y, s)-\overline{W}(y)|\leq\varepsilon^{\frac{1}{11}}\eta^{\frac{1}{6}}(y)$
and $|\partial_y (W_0(y, s)-\overline{W}(y))|\leq \varepsilon^{\frac{1}{12}}\eta^{-\frac{1}{3}}(y)$ for $|y|\leq\mathcal{L}\varepsilon^{\frac{1}{4}}e^{\frac{s}{4}}$},\\
|\partial_y W_0(y, s)|\leq\frac{7}{6}\eta^{-\frac{1}{3}}(y)\quad\text{for $|y|\geq\mathcal{L}\varepsilon^{\frac{1}{4}}e^{\frac{s}{4}}$}.
\end{cases}
\end{equation*}

\item For $1\leq j\leq n-1$,
\begin{equation*}
|\partial_y W_j(y, s)|\lesssim e^{(3\nu-\frac{3}{2})s}\eta^{-\nu}(y)\ (\nu\in [0, \frac{1}{3}]), |\partial_y^2 W_j(y, s)|\lesssim e^{(\nu^+-\frac{7}{6})s}\eta^{-\frac{1}{3}}(y)\ (\nu^+>0).
\end{equation*}

\item For $\mu_0$ given in \eqref{iv-50},
\begin{equation*}
\|\partial_y^{\mu_0}W_j(\cdot, s)\|_{L^2(\mathbb{R})}\lesssim e^{-\frac{3}{2}s}\ (1\leq j\leq n-1),\ \|\partial_y^{\mu_0}W_n(\cdot, s)\|_{L^2(\mathbb{R})}\lesssim e^{-\frac{s}{2}}.
\end{equation*}
\end{enumerate}}
\end{thm}

\begin{rem}\label{remii-2} In Theorem \ref{thmii-1}, the spatial decay estimates in $(4)$ and $(5)$
come from the influences of the initial data $W(y, -\log\varepsilon)$ without compact support.
\end{rem}

\section{Bootstrap assumptions}\label{iv}

Since the local existence of \eqref{i-6} was known already (one can see \cite{MA} for instance),
we utilize the continuous induction to establish the global-in-time estimates in Theorem \ref{thmii-1}. According to the initial data in \eqref{iii-14}-\eqref{iii-9}, we first make the following induction assumptions. In what follows, $M>0$ is denoted as a
suitably large constant, which is independent of $\varepsilon$.

For the modulation variables in \eqref{ii-1}, suppose that
\begin{subequations}\label{iv-1}\begin{align}
&|\kappa(t)-\kappa_0\varepsilon^{\frac{1}{3}}|\leq M\varepsilon,\quad|\tau(t)|\leq M\varepsilon^{\frac{4}{3}},\quad\quad |\xi(t)|\leq M\varepsilon^2,\label{iv-1a}\\
&|\dot{\kappa}(t)|\leq M,\quad\quad |\dot{\tau}(t)|\leq M\varepsilon^{\frac{1}{3}},\quad\quad |\dot{\xi}(t)|\leq M\varepsilon.\label{iv-1b}
\end{align}
\end{subequations}

For the bad unknown $W_0$ and the related $\mathcal{W}$ in \eqref{ii-14}, the bootstrap assumptions are
\begin{subequations}\label{iv-2}\begin{align}
&\|W_0(\cdot, s)\|_{L^{\infty}}\leq M\varepsilon^{\frac{1}{3}}e^{\frac{s}{2}},\ |\partial_y^{\mu}W_0(y, s)|\leq M\ (1\leq\mu\leq 4),\label{iv-2a}\\
&|\partial_y W_0(y, s)|\leq \frac{7}{6}\eta^{-\frac{1}{3}}(y)\quad\text{for $|y|\geq \mathcal{L}\varepsilon^{\frac{1}{4}}e^{\frac{s}{4}}$},\label{iv-2b}\\
&\text{$|\mathcal{W}(y, s)|\leq \varepsilon^{\frac{1}{11}} \eta^{\frac{1}{6}}(y)$
and $|\partial_y\mathcal{W}(y, s)|\leq \varepsilon^{\frac{1}{12}}\eta^{-\frac{1}{3}}(y)$ for $|y|\leq\mathcal{L}\varepsilon^{\frac{1}{4}}e^{\frac{s}{4}}$}.\label{iv-2c}
\end{align}
\end{subequations}

For the good unknowns $W_i (1\leq i\leq n-1)$, it is assumed that
\begin{subequations}\label{iv-3}\begin{align}
&|W_i(y, s)|\leq M\varepsilon,\ |\partial_y^{\mu} W_i(y, s)|\leq Me^{-\frac{3}{2}s}\ (1\leq\mu\leq 4),\label{iv-3a}\\
&|\partial_y W_i(y, s)|\leq Me^{(3\nu-\frac{3}{2})s}\eta^{-\nu}(y)\ (\nu\in [0, \frac{1}{3}]),\label{iv-3b}\\
&|\partial_y^2 W_i(y, s)|\leq Me^{(\nu^+-\frac{7}{6})s}\eta^{-\frac{1}{3}}(y)\ (\nu^+>0).\label{iv-3c}
\end{align}
\end{subequations}

In addition, we  make the following auxiliary assumptions with $\ell=\frac{1}{M^4}$,
\begin{equation}\label{iv-4}
|\partial_y^{\mu}\mathcal{W}(y, s)|\leq \varepsilon^{\frac{1}{10}}\ell^{4-\mu},\ 0\leq \mu\leq 4,\ |y|\leq \ell.
\end{equation}

With respect to the energies of the higher order derivatives of $W$, by
fixing $\mu_0$ to be the minimum positive integer such that
\begin{equation}\label{iv-50}
\mu_0\geq 6,\ 3\mu_0-e^{\frac{s}{2}}\beta_{\tau}\partial_y\mu_m(w)\geq \frac{13}{2}\ (1\leq m\leq n-1),
\end{equation}
we assume
\begin{equation}\label{iv-6}
\sum\limits_{m=1}^{n-1}\|\partial_y^{\mu_0}W_m(\cdot, s)\|_{L^2(\mathbb{R})}\leq Me^{-\frac{3}{2}s},\quad \|\partial_y^{\mu_0}W_n(\cdot, s)\|_{L^2(\mathbb{R})}\leq Me^{-\frac{s}{2}}.
\end{equation}

\section{Bootstrap estimates on the bad component of $W$}\label{vi}

In the section, we close the bootstrap arguments on $W_0$  and $\mathcal{W}$.

\subsection{The analysis on the characteristics of \eqref{ii-5}}

We now study some properties of the characteristics of \eqref{ii-5}. For any point $(y_0, \zeta_0)$ with $\zeta_0\geq -\log\varepsilon$,
the characteristics  $y(\zeta)=y(\zeta; y_0, \zeta_0)$ of \eqref{ii-5}
starting from $(y_0, \zeta_0)$ is defined as
\begin{equation}\label{vi1-1}\begin{cases}
\dot{y}(\zeta)=\frac{3}{2}y(\zeta)+e^{\frac{s}{2}}\beta_{\tau}(\mu_n(w)-\dot{\xi}(t(\zeta)))(y(\zeta), \zeta),\ \zeta\geq \zeta_0,\\
y(\zeta_0)=y_0.
\end{cases}
\end{equation}

\begin{prop}\label{lem5-1} {\it Under the assumptions \eqref{i-7c} and \eqref{iv-1}-\eqref{iv-3},
when $|y_0|\geq \ell$, one has
\begin{equation}\label{vi1-2}
|y(\zeta)|\geq |y_0|e^{\frac{\zeta-\zeta_0}{2}}\ (\zeta\geq\zeta_0).
\end{equation}
In addition, if $(y(\zeta), \zeta)$  goes through some point $(y, s)$ with $|y|\leq\ell$ and $s\geq\zeta_0$, then
\begin{equation}\label{vi1-21}
|y(\zeta)|\leq \ell\ (\zeta_0\leq \zeta\leq s).
\end{equation}}
\end{prop}

To prove Proposition \ref{lem5-1} and for later uses, we first establish the following results
on $\mu_n(w)-\dot{\xi}(t)$.
\begin{lem}\label{lem5-0} {\it One has
\begin{equation}\label{vi1-6}
|\mu_n(w)-\dot{\xi}(t)|\leq e^{-\frac{s}{2}}(\frac{7}{6}+\varepsilon^{\frac{1}{20}})|y|+M^2e^{-s}
\end{equation}
and
\begin{equation}\label{vi1-61}
|\mu_n(w)-\dot{\xi}(t)-e^{-\frac{s}{2}}W_0|\leq \varepsilon^{\frac{1}{8}}e^{-\frac{s}{2}}|y|^{\frac{1}{2}}+M^2e^{-s}.
\end{equation}
}
\end{lem}

\begin{proof} Note that
\begin{equation}\label{vi1-3}
(\mu_n(w)-\dot{\xi}(t))(y, s)=(\mu_n(w)-\mu_n(w^0))(y, s)+(\mu_n(w^0)-\dot{\xi}(t))(s).
\end{equation}
Then it follows from \eqref{ii-3}, \eqref{ii-10}, \eqref{ii-12c}, \eqref{iv-2a} and \eqref{iv-3a} that
\begin{equation}\label{vi1-4}
|\mu_n(w^0)-\dot{\xi}(t)|(s)\leq M^2 e^{-s}.
\end{equation}

In addition, due to $\partial_{w_n}\mu_n(0)=1$, it is derived from \eqref{i-7c}, \eqref{ii-10}, \eqref{iv-2a}-\eqref{iv-2b}, \eqref{iv-3a} and \eqref{iv-3b} with $\nu=\frac{1}{4}$ that
\begin{equation}\label{vi1-5}\begin{aligned}
&|\mu_n(w)-\mu_n(w^0)-e^{-\frac{s}{2}}W_0|(y, s)\\
=&|\mu_n(w)-\mu_n(w^0)-e^{-\frac{s}{2}}\int_0^{y}\partial_y W_0(z, s)dz|(y, s)\\
\leq &\sum\limits_{i=1}^{n}|\int_0^1\partial_{w_i}\mu_n(\beta w+(1-\beta)w^0) d\beta-\delta_n^i|\cdot |\int_0^{y}\partial_y W_i(z, s) dz|\\
\leq &M^2\varepsilon^{\frac{1}{4}}e^{-\frac{s}{2}}\min\{|y|^{\frac{1}{2}}, |y|\}.\end{aligned}
\end{equation}
By \eqref{iv-2b}-\eqref{iv-2c}, \eqref{ii-9b} and \eqref{ii-10}, we arrive at
\begin{equation}\label{vi1-51}
|W_0(y, s)|=|\int_0^y\partial_y W_0(z, s)dz|\leq (\frac{7}{6}+\varepsilon^{\frac{1}{14}})|y|.
\end{equation}
Substituting \eqref{vi1-4}-\eqref{vi1-51} into \eqref{vi1-3} yields \eqref{vi1-6}-\eqref{vi1-61} and
then the proof of Lemma \ref{lem5-0} is completed.\end{proof}

We now start the proof of Proposition \ref{lem5-1}.
\begin{proof} Since \eqref{vi1-21} can be easily derived from \eqref{vi1-2}, it suffices to prove \eqref{vi1-2}.
Due to $|\beta_{\tau}|\leq 2$ by \eqref{ii-4} and \eqref{iv-1b}, then it follows from \eqref{vi1-1} and \eqref{vi1-6} that
\begin{equation}\label{vi1-7}
\frac{d}{d\zeta}y^2(\zeta)\geq \frac{7}{12}y^2(\zeta)-2M^2 e^{-\zeta}|y(\zeta)|\geq \frac{13}{24}y^2(\zeta)-M^2 e^{-\zeta}.
\end{equation}
We derive from \eqref{vi1-7} and the assumption $|y_0|\geq\ell$ that
\begin{equation*}
e^{-\frac{13}{24}\zeta}y^2(\zeta)\geq e^{-\frac{13}{24}\zeta_0}y_0^2-M^2 e^{-(1+\frac{13}{24})\zeta_0}\geq \frac{1}{4}e^{-\frac{13}{24}\zeta_0}\ell^2\ (\zeta\geq \zeta_0)
\end{equation*}
and
\begin{equation*}
|y(\zeta)|\geq \frac{1}{2}\ell\quad\text{when $\zeta\geq\zeta_0$}.
\end{equation*}
Together with \eqref{vi1-7}, this yields
\begin{equation*}
\frac{d}{d\zeta}y^2(\zeta)\geq \frac{1}{2}y^2(\zeta).
\end{equation*}
Then \eqref{vi1-2} is obtained and the proof of Proposition \ref{lem5-1} is completed.\end{proof}

\begin{rem}\label{lem5-2} {\it For each point $(y, s)\in\mathbb{R}\times [-\log\varepsilon, +\infty)$, one can define the
following backward characteristics $y=y(\zeta)$ starting from $(y_0, \zeta_0)=(y_0(y, s), \zeta_0)$
\begin{equation}\label{vi1-8}\begin{cases}
\dot{y}(\zeta)=\frac{3}{2}y(\zeta)+e^{\frac{s}{2}}\beta_{\tau}(\mu_n(w)-\dot{\xi}(t))(y(\zeta), \zeta),\ \zeta_0\leq \zeta\leq s,\\
y(\zeta_0)=y_0.
\end{cases}
\end{equation}
According to Proposition \ref{lem5-1}, $(y(\zeta), \zeta)$ can be clarified into one of the following cases:
\begin{enumerate}[{\bf Case} $1$.]
\item Set $\zeta_0=-\log\varepsilon$ and there are no additional constrains on $y(\zeta)$.

\item When $|y|\leq \ell$, set $\zeta_0=-\log\varepsilon$,
then $|y(\zeta)|\leq \ell$ holds for $-\log\varepsilon\leq\zeta\leq s$ due to \eqref{vi1-21}.

\item When $|y|\geq \ell$, set $(y_0, \zeta_0)$ with $|y_0|>\ell$ and $\zeta_0=-\log\varepsilon$ or $|y_0|=\ell$ and $\zeta_0\geq-\log\varepsilon$, then $|y(\zeta)|\geq |y_0| e^{\frac{\zeta-\zeta_0}{2}}$ holds
    for $\zeta_0\leq\zeta\leq s$ due to \eqref{vi1-2}.

\end{enumerate}}

\end{rem}

\subsection{Bootstrap estimate on $W_0$}\label{vi-a}

For each point $(y, s)\in\mathbb{R}\times [-\log\varepsilon, +\infty)$, it follows from
the  {\bf Case} 1 in Remark \ref{lem5-2} and  \eqref{ii-11b} with $\mu=0$ that
\begin{equation}\label{vi2-1}
(\frac{d}{d\zeta}-\frac{1}{2})W_0(y(\zeta), \zeta)=F_0^0(y(\zeta), \zeta).
\end{equation}
By \eqref{iv-1b}, \eqref{iv-2a} and \eqref{iv-3a}, we have
\begin{equation}\label{vi2-2}\begin{aligned}
|F_0^0(y(\zeta), \zeta)|\leq& e^{-\frac{\zeta}{2}}(|\beta_{\tau}||\dot{\kappa}(t)|)(\zeta)
+\beta_{\tau}e^{\zeta}\sum\limits_{j=1}^{n-1}(|a_{nj}(w)||\partial_y W_j|)(y(\zeta), \zeta)\\
\leq& 2Me^{-\frac{\zeta}{2}}+M^2 e^{-\frac{\zeta}{2}}\leq 2M^2 e^{-\frac{\zeta}{2}}.
\end{aligned}
\end{equation}
Combining \eqref{vi2-1} with \eqref{vi2-2} yields
\begin{equation}\label{vi2-3}\begin{aligned}
e^{-\frac{s}{2}}|W_0(y, s)|&=|e^{\frac{1}{2}\log\varepsilon}W_0(y_0, -\log\varepsilon)
+\int_{-\log\varepsilon}^{s}e^{-\frac{\zeta}{2}}F_0^0(y(\zeta), \zeta)d\zeta|\\
&\leq \sqrt{\varepsilon}\|W_0(\cdot, -\log\varepsilon)\|_{L^{\infty}}+2M^2\varepsilon.
\end{aligned}
\end{equation}
Then it is derived from \eqref{vi2-3} and \eqref{iii-7a} that
\begin{equation}\label{vi2-4}
\|W_0(\cdot, s)\|_{L^{\infty}}\leq \varepsilon^{\frac{1}{3}}e^{\frac{s}{2}}.
\end{equation}

\subsection{Bootstrap estimate on $\mathcal{W}$ when $|y|\leq\ell$}\label{vi-b}

For each $(y, s)\in \mathbb{R}\times [-\log\varepsilon, +\infty)$, by
{\bf Case} 1 in Remark \ref{lem5-2}, it follows from \eqref{ii-15} that
\begin{equation}\label{vi3-1}\begin{aligned}
\partial^{\mu}\mathcal{W}(y, s)=&\partial^{\mu}\mathcal{W}(y_0(y, s), -\log\varepsilon)\exp\left({-\int_{-\log\varepsilon}^{s}\mathcal{D}_{\mu}(y(\alpha), \alpha)d\alpha}\right)\\
&+\int_{-\log\varepsilon}^{s}\mathcal{F}_{\mu}(y(\zeta), \zeta)\exp\left(-{\int_{\zeta}^{s}\mathcal{D}_{\mu}(y(\alpha), \alpha)d\alpha}\right)d\zeta.
\end{aligned}
\end{equation}
We next estimate $\partial_y^4\mathcal{W}(y, s)$ when $|y|\leq\ell$. In this situation,
$|y(\alpha)|\leq\ell$ holds for $-\log\varepsilon\leq\alpha\leq s$ by \eqref{vi1-21} in Proposition \ref{lem5-1}.
For $\mu=4$ in \eqref{ii-15}, one has from \eqref{ii-3} and \eqref{ii-14} that
\begin{equation}\label{vi3-11}\begin{aligned}
&\mathcal{D}_4(y(\alpha), \alpha)\\
=&\frac{11}{2}+\beta_{\tau}\overline{W}'(y(\alpha))+\sum\limits_{m=1}^{n}4e^{\frac{s}{2}}\beta_{\tau}\partial_{w_m}\mu_{n}(w)\partial_y W_m\\
=&\frac{11}{2}+\beta_{\tau}\overline{W}'(y(\alpha))+4\beta_{\tau}\partial_y W_0(y(\alpha), \alpha)+\sum\limits_{m=1}^{n}4e^{\frac{s}{2}}\beta_{\tau}\left(\partial_{w_m}\mu_n(w)-\delta_n^m\right)\partial_y W_m\\
=&\frac{11}{2}+5\beta_{\tau}\overline{W}'(y(\alpha))+4\beta_{\tau}\partial_y\mathcal{W}(y(\alpha), \alpha)+\sum\limits_{m=1}^{n}4e^{\frac{s}{2}}\beta_{\tau}\left(\partial_{w_m}\mu_n(w)-\delta_n^m\right)\partial_y W_m.
\end{aligned}
\end{equation}
Due to $\partial_{w_n}\mu_n(0)=1$ and $|y(\alpha)|\leq\ell$ for $-\log\varepsilon\leq\alpha\leq s$,
then by \eqref{vi3-11}, \eqref{i-7}, \eqref{ii-3}, \eqref{ii-9b}, \eqref{iv-1}, \eqref{iv-2} and \eqref{iv-3a},
we arrive at
\begin{equation}\label{vi3-2}
\mathcal{D}_4\geq \frac{11}{2}-5-M\varepsilon^{\frac{1}{12}}\geq\frac{1}{3}.
\end{equation}

It follows from \eqref{vi3-1} and \eqref{vi3-2} that
\begin{equation}\label{vi3-3}\begin{aligned}
|\partial_y^4\mathcal{W}(y, s)|&\leq|\partial_y^4\mathcal{W}(y_0(y, s), -\log\varepsilon)|+3\|\mathcal{F}_{4}(y(\zeta), \zeta)\chi_{[-\log\varepsilon, s]}(\zeta)\|_{L^{\infty}}.
\end{aligned}
\end{equation}

When $|y|\leq\ell$, it is derived from \eqref{vi1-21} in Proposition \ref{lem5-1} and \eqref{iii-7d} that $|y_0|=|y_0(y, s)|\leq \ell<1$ and
\begin{equation}\label{vi3-4}
|\partial_y^4\mathcal{W}(y_0(y, s), -\log\varepsilon)|\leq \varepsilon^{\frac{1}{9}}.
\end{equation}

For $\mathcal{F}_4(y(\zeta), \zeta)$ in \eqref{vi3-3}, it follows from \eqref{ii-15}, \eqref{i-7}, \eqref{ii-9}, \eqref{ii-12c}, \eqref{iv-1}, \eqref{iv-2a}-\eqref{iv-2b} and \eqref{iv-3a}-\eqref{iv-3b} that
\begin{equation}\label{vi3-5}\begin{aligned}
&|\mathcal{F}_4(y(\zeta), \zeta)|\\
\leq& \sum\limits_{\nu=0}^{3}M|\partial_y^{\nu}\mathcal{W}(y(\zeta), \zeta)|+M^2\varepsilon^{\frac{1}{2}}+Me^{\frac{s}{2}}\sum\limits_{\nu=0}^{4}|\partial_y^{\nu}(\mu_{n}(w)
-e^{-\frac{s}{2}}W_0-\dot{\xi}(t))|(y(\zeta), \zeta)
\end{aligned}
\end{equation}
and
\begin{equation}\label{vi3-6}\begin{aligned}
&\sum\limits_{\nu=0}^{4}|\partial_y^{\nu}(\mu_{n}(w)-e^{-\frac{s}{2}}W_0-\dot{\xi}(t))|(y(\zeta), \zeta)\\
=&\sum\limits_{\nu=0}^{4}|\partial_y^{\nu}\left(\sum\limits_{m=1}^{n}\int_0^1(\partial_{w_m}\mu_n(\beta w+(1-\beta)w^0)d\beta-\delta_n^m)\cdot\int_0^{y}\partial_y W_m(z, \zeta)dz\right)|(y(\zeta), \zeta)\\
\leq& M^2 e^{-\frac{3s}{2}}+M^2\varepsilon^{\frac{1}{3}}e^{-\frac{s}{2}}(1+|y(\zeta)|).
\end{aligned}
\end{equation}
By the definition of $\ell$ in \eqref{iv-4} and $|y(\zeta)|\leq\ell $ shown in \eqref{vi1-21} when $|y|\leq\ell$,
then it is derived from \eqref{vi3-5}-\eqref{vi3-6} and \eqref{iv-4} that
\begin{equation*}
|\mathcal{F}_4(y(\zeta), \zeta)|\leq\varepsilon^{\frac{1}{10}}\ell^{\frac{3}{5}}.
\end{equation*}
Combining this with \eqref{vi3-3}-\eqref{vi3-4} yields
\begin{equation}\label{vi-9}
|\partial_y^4\mathcal{W}(y, s)|\leq 2\varepsilon^{\frac{1}{10}}\ell^{\frac{1}{2}}\leq\frac{1}{2}\varepsilon^{\frac{1}{10}}\quad
\text{for $|y|\leq\ell$}.
\end{equation}
In addition, by \eqref{ii-9a} and \eqref{ii-10}, $\partial_y^{\mu}\mathcal{W}(0, s)=0$ for $\mu=1, 2, 3$,
we have that for $|y|\leq\ell$,
\begin{equation}\label{vi-10}\begin{aligned}
|\partial_y^{\nu}\mathcal{W}(y, s)|=&|\int_{0}^{y}\partial_y^{\nu+1}\mathcal{W}(z, s)dz|\\
\leq& 2\varepsilon^{\frac{1}{10}}\ell^{4-\nu+\frac{1}{2}}\leq \frac{1}{2}\varepsilon^{\frac{1}{10}}\ell^{4-\nu}\quad (0\leq\nu\leq 3).
\end{aligned}
\end{equation}

\subsection{Bootstrap estimates on $\mathcal{W}$ when $|y|\leq\mathcal{L}\varepsilon^{\frac{1}{4}}e^{\frac{s}{4}}$}\label{vi-c}

In the region $\{(y, s): |y|\leq \ell\}$, it is derived from \eqref{vi-10} that
\begin{equation}\label{vi-11}
|\eta^{-\frac{1}{6}}(y)\mathcal{W}(y, s)|\leq 2\varepsilon^{\frac{1}{10}}\ell^{4+\frac{1}{2}}\quad\text{for $|y|\leq\ell$}.
\end{equation}
In the region $\{(y, s): \ell<|y|\leq \mathcal{L}\varepsilon^{\frac{1}{4}}e^{\frac{s}{4}}\}$, by {\bf Case} 3 in Remark \ref{lem5-2}
and \eqref{ii-16}, one has
\begin{equation*}\label{vi4-1}
(\frac{d}{d\zeta}+\mathcal{D}_{\mu, \nu}(y(\zeta), \zeta))[\eta^{\nu}\partial^{\mu}\mathcal{W}](y(\zeta) \zeta)=[\eta^{\nu}\mathcal{F}_{\mu}](y(\zeta), \zeta)
\end{equation*}
and
\begin{equation}\label{vi4-2}\begin{aligned}
\left[\eta^{\nu}\partial^{\mu}\mathcal{W}\right](y, s)&=[\eta^{\nu}\partial^{\mu}\mathcal{W}](y_0, \zeta_0)\exp\left(-\int_{\zeta_0}^{s}\mathcal{D}_{\mu, \nu}(y(\alpha), \alpha)d\alpha\right)\\
&+\int_{\zeta_0}^{s}[\eta^{\nu}\mathcal{F}_{\mu}](y(\zeta), \zeta)\exp\left(-\int_{\zeta}^{s}\mathcal{D}_{\mu, \nu}(y(\alpha), \alpha)d\alpha\right)d\zeta.
\end{aligned}
\end{equation}
For $(\mu, \nu)=(0, -\frac{1}{6})$ in \eqref{ii-16}, it comes from \eqref{ii-12}, \eqref{iv-2a} and \eqref{iv-3a} that
\begin{equation}\label{vi4-3}
\mathcal{D}_{0, -\frac{1}{6}}=-\frac{1}{2(1+y^2)}+\beta_{\tau}\overline{W}'+\frac{\beta_{\tau}}{3}\frac{y}{1+y^2}e^{\frac{s}{2}}(\mu_{n}(w)-\dot{\xi}(t)).
\end{equation}
In addition, by \eqref{vi1-61}, \eqref{iv-2b} and \eqref{ii-10}, one has
\begin{equation}\label{vi4-4}
|\mu_n(w)-\dot{\xi}(t)|(y(\alpha), \alpha)\leq 4e^{-\frac{\alpha}{2}}|y(\alpha)|^{\frac{1}{2}}+M^2 e^{-\alpha}.
\end{equation}
Combining \eqref{vi4-3}-\eqref{vi4-4} with  \eqref{ii-9b} and \eqref{iv-1b} yields
\begin{equation*}
|\mathcal{D}_{0, -\frac{1}{6}}(y(\alpha), \alpha)|\leq 10\eta^{-\frac{1}{4}}(y(\alpha))+M^2 e^{-\frac{\alpha}{2}},
\end{equation*}
and then it follows from Proposition \ref{lem5-1} and {\bf Case 3} in Remark \ref{lem5-2} that
\begin{equation}\label{vi4-5}
\int_{\zeta_0}^{+\infty}|\mathcal{D}_{0, -\frac{1}{6}}|(y(\alpha), \alpha)d\alpha
\leq 2M^2\varepsilon^{\frac{1}{2}}+10\int_{\zeta_0}^{+\infty}\frac{1}{(1+\ell^2 e^{\frac{\alpha-\zeta_0}{2}})^{\frac{1}{4}}}d\alpha\leq 200 \ln\frac{1}{\ell}.
\end{equation}
From \eqref{vi4-2} with $(\mu, \nu)=(0, -\frac{1}{6})$ and \eqref{vi4-5}, we obtain
\begin{equation}\label{vi4-6}\begin{aligned}
&|\eta^{-\frac{1}{6}}\mathcal{W}|(y, s)\\
\leq& \frac{1}{\ell^{200}}\left(\eta^{-\frac{1}{6}}(y(\zeta_0))|\mathcal{W}|(y(\zeta_0), \zeta_0)+\int_{\zeta_0}^{s}\eta^{-\frac{1}{6}}(y(\zeta))|\mathcal{F}_0|(y(\zeta), \zeta)d\zeta\right)\quad\text{for $|y|\geq\ell$}.
\end{aligned}
\end{equation}
With \eqref{ii-9b}, \eqref{iii-6}, \eqref{iv-1}, \eqref{iv-3a} and \eqref{vi1-4}-\eqref{vi1-5} , $\mathcal{F}_0$ in \eqref{ii-15} satisfies
\begin{equation}\label{vi4-7}\begin{aligned}
&|\mathcal{F}_0|(y(\zeta), \zeta)\\
\leq& 2e^{\frac{s}{2}}\eta^{-\frac{1}{3}}(y(\zeta))(|\mu_n(w^0)-\dot{\xi}(t)|+|\mu_n(w)-\mu_{n}(w^0)-e^{-\frac{s}{2}}W_0|)\\
&+M^2 e^{-\frac{s}{2}}+2M\varepsilon\eta^{-\frac{1}{6}}(y(\zeta))\\
\leq& 2e^{\frac{s}{2}}\eta^{-\frac{1}{3}}(y(\zeta))(M^2\varepsilon^{\frac{1}{4}}e^{-\frac{\zeta}{2}}\eta^{\frac{1}{4}}(y(\zeta))
+M^2 e^{-\zeta})+M^2 e^{-\frac{\zeta}{2}}+2M\varepsilon\eta^{-\frac{1}{6}}(y(\zeta))\\
\leq& M^3 e^{-\frac{\zeta}{2}}+M^3\varepsilon^{\frac{1}{4}}\eta^{-\frac{1}{12}}(y(\zeta)).
\end{aligned}
\end{equation}
When $|y|\leq\mathcal{L}\varepsilon^{\frac{1}{4}}e^{\frac{s}{4}}$ and $\zeta_0=-\log\varepsilon$, one has $|y(\zeta_0)|\leq\mathcal{L}$ due to \eqref{vi1-2} in Proposition \ref{lem5-1}. Thus, combining \eqref{vi4-7} with \eqref{vi4-6}, \eqref{vi-11}, \eqref{iii-7b}
and {\bf Case 3} in Remark \ref{lem5-2} shows
\begin{equation}\label{vi-16}
|\eta^{-\frac{1}{6}}(y)\mathcal{W}|(y, s)\leq\frac{1}{2}\varepsilon^{\frac{1}{11}}\quad\text{for $|y|\leq \mathcal{L}\varepsilon^{\frac{1}{4}}e^{\frac{s}{4}}$}.
\end{equation}

\subsection{Bootstrap estimates on $\partial_y\mathcal{W}$ when $|y|\leq \mathcal{L}\varepsilon^{\frac{1}{4}}e^{\frac{s}{4}}$}\label{vi-d}

As in Subsection \ref{vi-c}, the $L^{\infty}$ estimate of $\eta^{\frac{1}{3}}\partial_y\mathcal{W}$ is still considered in
the cases of $|y|\leq \ell$ and $\ell<|y|\leq \mathcal{L}\varepsilon^{\frac{1}{4}}e^{\frac{s}{4}}$. In the region $\{(y, s): \ell<|y|\leq \mathcal{L}\varepsilon^{\frac{1}{4}}e^{\frac{s}{4}}\}$, by {\bf Case} 3 in Remark \ref{lem5-2},
it follows from \eqref{vi1-2} in Proposition \ref{lem5-1} that when $|y|\leq \mathcal{L}\varepsilon^{\frac{1}{4}}e^{\frac{s}{4}}$,
\begin{equation}\label{vi-170}
|y(\zeta)|\leq \mathcal{L}\varepsilon^{\frac{1}{4}}e^{\frac{\zeta}{4}},\ \zeta_0\leq \zeta\leq s.
\end{equation}
For $|y|\leq\ell$, by \eqref{vi-10}, one has that
\begin{equation}\label{vi-17}
\eta^{\frac{1}{3}}(y)|\partial_y\mathcal{W}(y, s)|\leq 2\varepsilon^{\frac{1}{10}}\ell^{3+\frac{1}{2}}\ (|y|\leq\ell).
\end{equation}
Next, we estimate $\eta^{\frac{1}{3}}\partial_y\mathcal{W}$ when $\ell\leq |y|\leq \mathcal{L}\varepsilon^{\frac{1}{4}}e^{\frac{s}{4}}$.
For $(\mu, \nu)=(1, \frac{1}{3})$ in \eqref{ii-16}, we have
\begin{equation*}\begin{aligned}
\mathcal{D}_{1, \frac{1}{3}}=&\frac{1}{1+y^2}+\beta_{\tau}\overline{W}'+\beta_{\tau}e^{\frac{s}{2}}\partial_y \mu_{n}(w)-\frac{2y}{3(1+y^2)}\beta_{\tau}e^{\frac{s}{2}}(\mu_{n}(w)-\dot{\xi}(t))\\
=&\frac{1}{1+y^2}+\beta_{\tau}\overline{W}'+\beta_{\tau}\partial_{w_n}\mu_n(w)\partial_y W_0+\beta_{\tau}e^{\frac{s}{2}}\sum\limits_{j\neq n}\partial_{w_j}\mu_n(w)\partial_y W_j\\
&-\frac{2y}{3(1+y^2)}\beta_{\tau}e^{\frac{s}{2}}(\mu_n(w)-\dot{\xi}(t)).
\end{aligned}
\end{equation*}
Combining this with \eqref{i-7c}, \eqref{ii-9b}, \eqref{iv-1b}, \eqref{iv-2a}-\eqref{iv-2b}, \eqref{iv-3a}
and \eqref{vi4-4}-\eqref{vi4-5} yields
\begin{equation*}
|\mathcal{D}_{1, \frac{1}{3}}(y(\alpha), \alpha)|\leq 10\eta^{-\frac{1}{4}}(y(\alpha))+M^2 e^{-\frac{\alpha}{2}}
\end{equation*}
and
\begin{equation}\label{vi-18}
\int_{\zeta_0}^{+\infty}|\mathcal{D}_{1, \frac{1}{3}}|(y(\alpha), \alpha)d\alpha\leq 200\ln\frac{1}{\ell}.
\end{equation}
It is derived from \eqref{vi4-2} and \eqref{vi-18} that for $|y|\geq\ell$,
\begin{equation}\label{vi-19}\begin{aligned}
&|\eta^{\frac{1}{3}}(y)\partial_y\mathcal{W}|(y, s)\\
\leq& \frac{1}{\ell^{200}}\left(\eta^{\frac{1}{3}}(y(\zeta_0))|\partial_y\mathcal{W}|(y(\zeta_0), \zeta_0)+\int_{\zeta_0}^{s}\eta^{\frac{1}{3}}(y(\zeta))|\mathcal{F}_1|(y(\zeta), \zeta)d\zeta\right).
\end{aligned}
\end{equation}
For $\mathcal{F}_1$ in \eqref{vi-19}, one has from  \eqref{ii-15}, \eqref{i-7d} and \eqref{iv-1}, \eqref{iv-2a}, \eqref{iv-3a} that
\begin{equation}\label{vi-20}\begin{aligned}
|\mathcal{F}_1|&\leq 2|\overline{W}''\mathcal{W}|+Me^{s}|W|\sum\limits_{j=1}^{n-1}|\partial_y^2 W_j|+Me^{s}\sum\limits_{j=1}^{n-1}|\partial_y W_j|^2+Me^{\frac{s}{2}}\sum\limits_{j=1}^{n-1}|\partial_y W_0 \partial_y W_j|\\
&+2e^{\frac{s}{2}}|\overline{W}''||\mu_{n}(w)-e^{-\frac{s}{2}}W_0-\dot{\xi}(t)|+Me^{\frac{s}{2}}|\overline{W}'|\sum\limits_{j=1}^{n-1}|\partial_y W_j|\\
&+2|\overline{W}'||\partial_{w_n}\mu_{n}(w)-1||\partial_y W_0|+2M\varepsilon(|\overline{W}'|^2+|\overline{W}''\overline{W}|)
=\sum\limits_{k=1}^{8}I_k.
\end{aligned}
\end{equation}
For $I_1$, it follows from \eqref{ii-9b}, \eqref{vi-16} and \eqref{vi-170} that
\begin{equation}\label{vi-21}
\eta^{\frac{1}{3}}(y(\zeta))I_1(y(\zeta), \zeta)\leq 2\varepsilon^{\frac{1}{11}}\eta^{-\frac{1}{3}}(y(\zeta)).
\end{equation}
For $I_2$, we have from \eqref{ii-3}, \eqref{iv-1a}, \eqref{vi2-4}, \eqref{iv-3a} and \eqref{iv-3c} with $\nu^+=\frac{1}{24}$ that
\begin{equation}\label{vi-22}
\eta^{\frac{1}{3}}(y(\zeta))I_2(y(\zeta), \zeta)\leq M^4\varepsilon^{\frac{1}{3}}e^{-\frac{1}{8}\zeta}.
\end{equation}
For $I_3$, \eqref{iv-3b} with $\nu=\frac{1}{6}$ shows
\begin{equation}\label{vi-23}
\eta^{\frac{1}{3}}(y(\zeta))I_3(y(\zeta), \zeta)\leq M^3 e^{-\zeta}.
\end{equation}
In the similar way, due to $\partial_{w_n}\mu_n(0)=1$, it is derived from \eqref{i-7}, \eqref{iv-2a}, \eqref{iv-3a}
and \eqref{ii-9b} that
\begin{equation}\label{vi-24}
\eta^{\frac{1}{3}}(y(\zeta))(I_4+I_6+I_7+I_8)(y(\zeta), \zeta)\leq M^3 e^{-\zeta}
+M^2\varepsilon^{\frac{1}{3}}\eta^{-\frac{1}{3}}(y(\zeta)).
\end{equation}
In addition, it follows from \eqref{ii-9b} and \eqref{vi1-61} that
\begin{equation}\label{vi-25}
\eta^{\frac{1}{3}}(y(\zeta))I_5(y(\zeta), \zeta)\leq M^2\varepsilon^{\frac{1}{8}}\eta^{-\frac{1}{2}}(y(\zeta)).
\end{equation}
Substituting \eqref{vi-21}-\eqref{vi-25} into \eqref{vi-20} yields
\begin{equation}\label{vi-26}
\eta^{\frac{1}{3}}(y(\zeta))|\mathcal{F}_1|(y(\zeta), \zeta)\leq 4\varepsilon^{\frac{1}{11}}\eta^{-\frac{1}{3}}(y(\zeta))+\varepsilon^{\frac{1}{11}}e^{-\frac{1}{8}\zeta}.
\end{equation}
Analogously to obtain \eqref{vi-16}, combining \eqref{vi-26} with \eqref{vi-19}, \eqref{vi-17} and \eqref{iii-7b} derives
\begin{equation}\label{vi-27}
\eta^{\frac{1}{3}}(y)|\partial_y\mathcal{W}(y, s)|\leq\frac{1}{2}\varepsilon^{\frac{1}{12}}\quad \text{for $|y|\leq\mathcal{L}\varepsilon^{\frac{1}{4}}e^{\frac{s}{4}}$}.
\end{equation}

\subsection{More delicate estimates for $W_0$}\label{vi-e}

In the Subsection, we mainly estimate the weighted $L^{\infty}$ norms of $\eta^{-\frac{1}{6}}W_0$ and $\eta^{\frac{1}{3}}\partial_y W_0$
in the whole spatial space.
Since the proof procedures are very similar to the processes in Subsection \ref{vi-c} and Subsection \ref{vi-d},
we just give the sketch of the related verifications. For $\mu\in\mathbb{N}_0$ and $\nu\in\mathbb{R}$, it is
derived from \eqref{ii-5} that
\begin{equation}\label{vi-28}
\left(\partial_s+(\frac{3}{2}y+\beta_{\tau}e^{\frac{s}{2}}(\mu_n(w)-\dot{\xi}(t)))\partial_y\right)[\eta^{\nu}\partial^{\mu}W_0]+\overline{D}_{\mu, \nu}[\eta^{\nu}\partial^{\mu}W_0]=\eta^{\nu}\overline{F}_{\mu}^{0},
\end{equation}
where
\begin{equation*}
\overline{D}_{\mu, \nu}=\frac{3\mu-1}{2}+\mu\beta_{\tau} e^{\frac{s}{2}}\partial_y\mu_n(w)+\beta_{\tau}\partial_y W_0\boldsymbol{1}_{\mu\geq 2}-\frac{2\nu y}{1+y^2}(\frac{3}{2}y+\beta_{\tau}e^{\frac{s}{2}}(\mu_{n}(w)-\dot{\xi}(t)))
\end{equation*}
and
\begin{equation*}\begin{aligned}
\overline{F}_{\mu}^{0}&=-\sum\limits_{2\leq\beta\leq\mu-1}
C_{\mu}^{\beta}\beta_{\tau}e^{\frac{s}{2}}\partial_y^{\beta}\mu_n(w)\partial_y^{\mu-\beta+1}W_0
-\beta_{\tau}e^{\frac{s}{2}}\partial_y^{\mu}(\mu_n(w)-w_n)\partial_y W_0\boldsymbol{1}_{\mu\geq 2}\\
&-\sum\limits_{j=1}^{n-1}\beta_{\tau}e^{s}\partial_y^{\mu}(a_{nj}(w)\partial_y W_j)-\beta_{\tau}e^{-\frac{s}{2}}\dot{\kappa}(t)\delta_{\mu}^0.
\end{aligned}
\end{equation*}
When $|y|\leq\mathcal{L}\varepsilon^{\frac{1}{4}}e^{\frac{s}{4}}$,  it follows from \eqref{ii-9b}, \eqref{vi-16}
and \eqref{vi-27} that for $|y|\leq\mathcal{L}\varepsilon^{\frac{1}{4}}e^{\frac{s}{4}}$,
\begin{equation}\label{vi-29}
\eta^{-\frac{1}{6}}(y)|W_0(y, s)|\leq 1+\frac{1}{2}\varepsilon^{\frac{1}{11}},\ \eta^{\frac{1}{3}}(y)|\partial_y W_0(y, s)|\leq 1+\frac{1}{2}\varepsilon^{\frac{1}{12}}.
\end{equation}
When $|y|\geq \mathcal{L}\varepsilon^{\frac{1}{4}}e^{\frac{s}{4}}$, the backward characteristics $y=y(\zeta)$ is
defined by \eqref{vi1-8} with $(y_0, \zeta_0)$ satisfying $|y_0|\geq \mathcal{L}, \zeta_0=-\log\varepsilon$ or $|y_0|=\mathcal{L}\varepsilon^{\frac{1}{4}}e^{\frac{\zeta_0}{4}}, \zeta_0>-\log\varepsilon$. In this case, we have $|y(\zeta)|\geq\mathcal{L}e^{\frac{\zeta-\zeta_0}{4}}$ for $\zeta\geq\zeta_0$ due to \eqref{vi1-2}.
Moreover, it is derived from \eqref{vi-28} that
\begin{equation}\label{vi6-0}\begin{aligned}
\left[\eta^{\nu}\partial^{\mu}W_0\right](y, s)&=[\eta^{\nu}\partial^{\mu}W_0](y_0, \zeta_0)\exp \left(-\int_{\zeta_0}^{s}\overline{D}_{\mu, \nu}(y(\alpha), \alpha)d\alpha\right)\\
&+\int_{\zeta_0}^{s}[\eta^{\nu}\overline{F}_{\mu}^0](y(\zeta), \zeta)\exp\left(-\int_{\zeta}^{s}\overline{D}_{\mu, \nu}(y(\alpha), \alpha)d\alpha\right)d\zeta.
\end{aligned}
\end{equation}
In addition, by \eqref{vi4-4} and \eqref{iv-2a}-\eqref{iv-2b}, \eqref{iv-3a}, we have
\begin{equation}\label{vi6-1}\begin{aligned}
&|\overline{D}_{0, -\frac{1}{6}}(y(\alpha), \alpha)|\\
=&\left|-\frac{1}{2(1+y(\alpha)^2)}+\frac{y(\alpha)}{3(1+y(\alpha)^2)}e^{\frac{\alpha}{2}}\beta_{\tau}(\mu_n(w)-\dot{\xi}(t))\right|\\
\leq&\frac{1}{2}\eta^{-1}(y(\alpha))+e^{\frac{\alpha}{2}}\eta^{-\frac{1}{2}}(y(\alpha))|\mu_n(w)-\dot{\xi}(t)|\\
\leq&5\eta^{-\frac{1}{4}}(y(\alpha))+M^2 e^{-\frac{\alpha}{2}}
\end{aligned}
\end{equation}
and
\begin{equation}\label{vi6-2}\begin{aligned}
&|\overline{D}_{1, \frac{1}{3}}(y(\alpha), \alpha)|\\
=&\left|\frac{1}{1+y(\alpha)^2}+e^{\frac{\alpha}{2}}\beta_{\tau}\partial_y\mu_n(w)
-\frac{2y(\alpha)}{3(1+y(\alpha)^2)}e^{\frac{\alpha}{2}}\beta_{\tau}(\mu_n(w)-\dot{\xi}(t))\right|\\
\leq&\eta^{-1}(y(\alpha))+2e^{\frac{\alpha}{2}}\eta^{-\frac{1}{2}}(y(\alpha))|\mu_n(w)
-\dot{\xi}(t)|+2e^{\frac{\alpha}{2}}\sum\limits_{j=1}^{n}|\partial_{w_j}\mu_n(w)||\partial_y W_j|\\
\leq& 20\eta^{-\frac{1}{4}}(y(\alpha))+M^3 e^{-\frac{\alpha}{2}}.
\end{aligned}
\end{equation}
Similarly to \eqref{vi4-5}, for $|y(\zeta)|\geq\mathcal{L}e^{\frac{\zeta-\zeta_0}{4}}$,
one has from \eqref{vi6-1}-\eqref{vi6-2} that
\begin{equation}\label{vi-31}\begin{aligned}
&\int_{\zeta_0}^{s}|\overline{D}_{0, -\frac{1}{6}}|(y(\alpha), \alpha)d\alpha+\int_{\zeta_0}^{s}|\overline{D}_{1, \frac{1}{3}}|(y(\alpha), \alpha)d\alpha\\
\leq& 4M^3\varepsilon^{\frac{1}{2}}+\int_{\zeta_0}^{s}\frac{25}{(1+\mathcal{L}^2 e^{\frac{\zeta-\zeta_0}{2}})^{\frac{1}{4}}}d\zeta\\
\leq & 4M^3\varepsilon^{\frac{1}{2}}+600\ln(1+\mathcal{L}^{-1})\leq \varepsilon^{\frac{1}{20}}.
\end{aligned}
\end{equation}
Based on \eqref{vi-31},  together with \eqref{vi6-0} this yields
\begin{equation}\label{vi-32}
\eta^{-\frac{1}{6}}(y)|W_0(y, s)|\leq (1+2\varepsilon^{\frac{1}{20}})\left(\eta^{-\frac{1}{6}}(y_0)|W_0(y_0, \zeta_0)|+\int_{\zeta_0}^{s}\eta^{-\frac{1}{6}}(y(\zeta))|\overline{F}_0^{0}|(y(\zeta), \zeta)d\zeta\right)
\end{equation}
and
\begin{equation}\label{vi-33}
\eta^{\frac{1}{3}}(y)|\partial_y W_0(y, s)|\leq (1+2\varepsilon^{\frac{1}{20}})\left(\eta^{\frac{1}{3}}(y_0)|\partial_y W_0(y_0, \zeta_0)|+\int_{\zeta_0}^{s}\eta^{\frac{1}{3}}(y(\zeta))|\overline{F}_1^{0}|(y(\zeta), \zeta)d\zeta\right).
\end{equation}
By \eqref{i-7}, \eqref{iv-1}-\eqref{iv-3}, \eqref{iv-3b} with $\nu=\frac{1}{3}$
and  \eqref{iv-3c} with $\nu^+=\frac{1}{24}$, then $\overline{F}_0^{0}$ and $\overline{F}_1^{0}$ in \eqref{vi-28} satisfy
\begin{equation}\label{vi-34}
|\overline{F}_0^{0}(y(\zeta), \zeta)|\leq 4M^2 e^{-\frac{\zeta}{2}},\ |\overline{F}_1^{0}(y(\zeta), \zeta)|\leq 2M^2e^{-\frac{\zeta}{9}}\eta^{-\frac{1}{3}}(y(\zeta)).
\end{equation}
Therefore, we derive from \eqref{iii-7b}, \eqref{iii-7c}, \eqref{vi-29} and \eqref{vi-32}-\eqref{vi-34} that
\begin{equation}\label{vi-35}
\eta^{-\frac{1}{6}}(y)|W_0(y, s)|\leq 1+\varepsilon^{\frac{1}{21}},\ \eta^{\frac{1}{3}}(y)|\partial_y W_0(y, s)|\leq 1
+\varepsilon^{\frac{1}{21}}.
\end{equation}
For the estimates of $\partial_y^{\mu}W_0$, together with Lemma \ref{lemA-3} for $u=\partial_y^{\mu}W_0$ and $d=1, p, q=\infty, r=2, j=\mu-1, m=\mu_0-1$, it comes from \eqref{ii-3}, \eqref{iv-6} and \eqref{vi-35} that for $2\leq \mu\leq 4$ and $\alpha
=\frac{\mu-1}{\mu_0-\frac{1}{2}}\in (0, \frac{6}{11}]$,
\begin{equation}\label{vi-36}
\|\partial_y^{\mu}W_0(\cdot, s)\|_{L^{\infty}}\leq M^{\frac{1}{20}}\|\partial_y^{\mu_0}W_0(\cdot, s)\|_{L^2}^{\alpha}\|\partial_y W_0(\cdot, s)\|_{L^{\infty}}^{1-\alpha}
\leq 2^{1-\alpha}M^{\frac{1}{20}+\alpha}\leq M^{\frac{3}{5}}.
\end{equation}

\section{Bootstrap  estimates on good components of $W$}\label{vii}

In the section, we will apply the characteristics method  to establish a series of estimates of $W_m\ (m\neq n)$.

\subsection{Framework for the characteristics method}\label{vii-a}

For $1\leq m\leq  n-1$ and any point $(y_0, \zeta_0)\in\mathbb{R}\times [-\log\varepsilon, +\infty)$, we consider the
following forward characteristics  $y(\zeta):=y(\zeta; y_0, \zeta_0)$ of  \eqref{ii-18} which
starts from $(y_0, \zeta_0)$:
\begin{equation}\label{vii1-0}\begin{cases}
\dot{y}(\zeta)=\frac{3}{2}y(\zeta)+e^{\frac{\zeta}{2}}\beta_{\tau}(\mu_m(w)-\dot{\xi}(t))(y(\zeta), \zeta),\\
y(\zeta_0)=y_0.
\end{cases}
\end{equation}
This yields that for $s\geq \zeta_0$,
\begin{equation}\label{vii-1}
y(\zeta)e^{-\frac{3}{2}\zeta}=y_0 e^{-\frac{3}{2}\zeta_0}+\int_{\zeta_0}^{\zeta}e^{-\alpha}\left(\beta_{\tau}(\cdot)(\mu_m(w)-\dot{\xi}(\cdot))\right)(y(\alpha), \alpha)d\alpha:=G_m(\zeta; y_0, \zeta_0).
\end{equation}

Next we discuss the positions of $y(\zeta; y_0, \zeta_0)$ for the different cases of $(y_0, \zeta_0)$.

\begin{lem}\label{lem7-1} {\it For each $i_0\leq m\leq n-1$, $a_m:=\mu_n(0)>0$  is due to \eqref{i-7a} and the assumption $\mu_n(0)=0$ in Remark \ref{rem1-1}. Then for any point $(y_0, \zeta_0)\in\mathbb{R}\times [-\log\varepsilon, +\infty)$, $y(\zeta):=y(\zeta; y_0, \zeta_0)$
 can be classified into the following six cases:
\begin{enumerate}[{\bf Case} $1$.]
\item When $y_0<-4a_m e^{\frac{\zeta_0}{2}}$,\ $y(\zeta)\leq -a_m e^{\frac{3}{2}\zeta-\zeta_0}<0$ holds for $\zeta\geq\zeta_0.$

\item When $ -\frac{a_m}{4}e^{\frac{\zeta_0}{2}}\leq y_0\leq 0$, there exists a number $\zeta^*\geq \zeta_0$ such that $G_m(\zeta^*; y_0, \zeta_0)=0$ and
\begin{equation*}\begin{cases}
-\frac{3}{2}a_m (e^{-\zeta}-e^{-\zeta^*}) e^{\frac{3}{2}\zeta}\leq y(\zeta)\leq -\frac{a_m}{2}(e^{-\zeta}-e^{-\zeta^*})e^{\frac{3}{2}\zeta}\leq 0
\quad\text{for $\zeta_0\leq \zeta\leq \zeta^*$},\\[2mm]
y(\zeta)\geq \frac{a_m}{2}(e^{-\zeta^*}-e^{-\zeta})e^{\frac{3}{2}\zeta}\geq 0\quad\text{for $\zeta\geq \zeta^*$}.
\end{cases}
\end{equation*}
\item When $y_0\geq 0$, one has $y(\zeta)\geq \frac{a_m}{2}(e^{-\zeta_0}-e^{-\zeta})e^{\frac{3}{2}\zeta}\geq 0$
\quad \text{for $\zeta\geq\zeta_0.$}

\item When $(y_0, \zeta_0)\in D^+$ and $(y(\zeta), \zeta)$ lies in the domain $D^+$, there holds
\begin{equation*}
-4a_m e^{\frac{\zeta}{2}}\leq y(\zeta)\leq -\frac{a_m}{4}e^{\frac{\zeta}{2}}<0\quad\text{for $\zeta\geq \zeta_0$},
\end{equation*}
where $D^{+}=\{(y, \zeta): -4a_m e^{\frac{\zeta}{2}}\leq y\leq -\frac{a_m}{4}e^{\frac{\zeta}{2}},\ \zeta\geq -\log\varepsilon\}$.

\item When $(y_0, \zeta_0)\in D^+$ and the characteristics $y=y(\zeta)$ goes through $\partial D^+$ at some point $(\hat y, \hat \zeta)$
with $\hat y=-4a_m e^{\frac{\hat\zeta}{2}}$, we have
\begin{equation*}\begin{cases}
(y(\zeta), \zeta)\in D^+\quad\text{for $\zeta_0\leq \zeta\leq\hat\zeta$},\\
y(\zeta)\leq -a_m e^{\frac{3}{2}\zeta-\hat\zeta}<0\quad\text{for $\zeta\geq\hat\zeta$.}
\end{cases}
\end{equation*}

\item When $(y_0, \zeta_0)\in D^+$ and the characteristics $y=y(\zeta)$ goes through $\partial D^+$ at
some point $(\hat y, \hat \zeta)$ with $\hat y=-\frac{a_m}{2} e^{\frac{\hat\zeta}{2}}$,
there exists $\tilde{\zeta}>\hat\zeta$ such that $G_m(\tilde{\zeta}; y_0, \zeta_0)=0$ and $y=y(\zeta)$ can be divided into
the three parts as:
\begin{equation*}\begin{cases}
(y(\zeta), \zeta)\in D^{+}\quad\text{for $\zeta_0\leq \zeta\leq \hat \zeta$},\\
-2a_m (e^{-\zeta}-e^{-\tilde{\zeta}})e^{\frac{3}{2}\zeta}\leq y(\zeta)\leq -\frac{a_m}{2}(e^{-\zeta}-e^{-\tilde{\zeta}})
e^{\frac{3}{2}\zeta}\leq 0\quad\text{for $\hat \zeta\leq \zeta\leq \tilde{\zeta}$},\\
y(\zeta)\geq\frac{a_m}{2}(e^{-\tilde{\zeta}}-e^{-\zeta})e^{\frac{3}{2}\zeta}\geq 0\quad\text{for $\zeta\geq \tilde{\zeta}$}.
\end{cases}
\end{equation*}
\end{enumerate}
}
\end{lem}

\begin{proof} Since $a_m>0$ for $i_0\leq m\leq n-1$, by \eqref{ii-3}, \eqref{iv-1}, \eqref{iv-2a} and \eqref{iv-3a}, we
then have
\begin{equation}\label{vii1-1}
|\beta_{\tau}(\mu_m(w)-\dot{\xi}(t))-a_m|\leq M^3\varepsilon^{\frac{1}{3}}\leq \frac{a_m}{2}.
\end{equation}

When $y_0<-4a_m e^{\frac{\zeta_0}{2}}$, it is derived from \eqref{vii-1}  and \eqref{vii1-1} that
\begin{equation}\label{vii1-2}
y(\zeta)e^{-\frac{3}{2}\zeta}\leq -4a_m e^{-\zeta_0}+\frac{3}{2}a_m (e^{-\zeta_0}-e^{-\zeta})\leq -a_m e^{-\zeta_0}.
\end{equation}
This shows {\bf Case} 1.

When $-\frac{a_m}{4}e^{\frac{\zeta_0}{2}}\leq y_0\leq 0$, it follows from \eqref{vii1-1}
that $G_m(\zeta; y_0, \zeta_0)$ in \eqref{vii-1} satisfies
\begin{equation}\label{vii1-3}\begin{cases}
G_m(\zeta_0; y_0, \zeta_0)=y_0 e^{-\frac{3}{2}\zeta_0}\leq 0,\\
G_m(+\infty; y_0, \zeta_0)\geq y_0 e^{-\frac{3}{2}\zeta_0}+\frac{a_m}{2}e^{-\zeta_0}\geq
-\frac{a_m}{4}e^{-\zeta_0}+\frac{a_m}{2}e^{-\zeta_0}>0.
\end{cases}
\end{equation}
Since $G_m(\zeta; y_0, \zeta_0)$ is a continuous function with respect to the variable $\zeta$, \eqref{vii1-3}
shows that there exists $\zeta^*\geq \zeta_0$ such that $G_m(\zeta^*; y_0, \zeta_0)=0$. In this situation,
we derive from \eqref{vii-1} that
\begin{equation}\label{vii1-4}
y(\zeta)e^{-\frac{3}{2}\zeta}=G_m(\zeta; y_0, \zeta_0)-G_m(\zeta^*; y_0, \zeta_0)=\int_{\zeta^*}^{\zeta}e^{-\alpha}\left(\beta_{\tau}(\cdot)(\mu_m(w)-\dot{\xi}(\cdot))\right)(y(\alpha), \alpha)d\alpha.
\end{equation}
Therefore, {\bf Case} 2 is obtained from \eqref{vii1-1}, \eqref{vii1-4} and $a_m>0$ for $i_0\leq m\leq n-1$.

Based on the results established in {\bf Case} 1-{\bf Case} 2, {\bf Case} 3-{\bf Case} 6 in Lemma \ref{lem7-1}
can be carried out in the same way due to the formula \eqref{vii-1} and the definition of $D^{+}$,
here we omit the details.\end{proof}

\begin{lem}\label{lem7-2} {\it For each $1\leq m\leq i_0-1$, $a_m:=\mu_m(0)<0$ is due to \eqref{i-7a} and the assumption $\mu_n(0)=0$ in Remarks \ref{rem1-1}. Then for any point $(y_0, \zeta_0)\in\mathbb{R}\times [-\log\varepsilon, +\infty)$, $y(\zeta)=y(\zeta; y_0, \zeta_0)$
 can be classified into the following six cases:
\begin{enumerate}[{\bf Case} $1$.]
\item When $y_0>-4a_m e^{\frac{\zeta_0}{2}}$, $y(\zeta)\geq -a_m e^{\frac{3}{2}\zeta-\zeta_0}>0.$

\item When $ 0\leq y_0\leq -\frac{a_m}{4}e^{\frac{\zeta_0}{2}}$, there exists a number $\zeta^{*}\geq \zeta_0$
such that $G_m(\zeta^*; y_0, \zeta_0)=0$ and
\begin{equation*}\begin{cases}
0\leq -\frac{a_m}{2}(e^{-\zeta}-e^{-\zeta^*})e^{\frac{3}{2}\zeta}\leq y(\zeta)\leq  -\frac{3}{2}a_m (e^{-\zeta}-e^{-\zeta^*}) e^{\frac{3}{2}\zeta} \quad\text{for $\zeta_0\leq \zeta\leq \zeta^{*}$},\\[2mm]
y(\zeta)\leq \frac{a_m}{2}(e^{-\zeta^*}-e^{-\zeta})\leq 0\quad\text{for $\zeta\geq \zeta^{*}$}.
\end{cases}
\end{equation*}

\item When $y_0\leq 0$, $y(\zeta)\leq \frac{a_m}{2}(e^{-\zeta_0}-e^{-\zeta})e^{\frac{3}{2}\zeta}\leq 0$
\quad\text{for $\zeta\geq \zeta_0.$}

\item When $(y_0, \zeta_0)\in D^-$ and the characteristics $(y(\zeta), \zeta)$ lies in $D^- $, one has
\begin{equation*}
0<-\frac{a_m}{4} e^{\frac{\zeta}{2}}\leq y(\zeta)\leq -4a_m e^{\frac{\zeta}{2}}\quad\text{for $\zeta\geq \zeta_0$},
\end{equation*}
where $D^{-}=\{(y, \zeta):  -\frac{a_m}{4}e^{\frac{\zeta}{2}}\leq y\leq -4a_m e^{\frac{\zeta}{2}},\ \zeta\geq -\log\varepsilon\}$.
\item When $(y_0, \zeta_0)\in D^-$ and the characteristics $(y(\zeta), \zeta)$ goes through $\partial D^-$ at some point $(\hat y, \hat \zeta)$ with $\hat y=-4 a_m e^{\frac{\hat\zeta}{2}}$, we have
\begin{equation*}\begin{cases}
(y(\zeta), \zeta)\in D^{-}\quad\text{for $\zeta_0\leq \zeta\leq \hat \zeta$},\\
y(\zeta)\geq -a_m e^{\frac{3}{2}\zeta-\hat \zeta}>0\quad\text{for $\zeta\geq \hat \zeta$}.
\end{cases}
\end{equation*}

\item When $(y_0, \zeta_0)\in D^-$ and the characteristics $(y(\zeta), \zeta)$ goes through $\partial D^-$
at some point $(\hat y, \hat \zeta)$ with $\hat y=-\frac{a_m}{2} e^{\frac{\hat\zeta}{2}}$, there exists $\tilde{\zeta}>\hat\zeta$
such that $y=y(\zeta)$ can be divided into the three parts as:
\begin{equation*}\begin{cases}
(y(\zeta), \zeta)\in D^{-}\quad\text{for $\zeta_0\leq \zeta\leq \hat \zeta$},\\
0\leq -\frac{a_m}{2}  (e^{-\zeta}-e^{-\tilde{\zeta}})e^{\frac{3}{2}\zeta}\leq y(\zeta)\leq -2a_m (e^{-\zeta}-e^{-\tilde{\zeta}})e^{\frac{3}{2}\zeta}\quad\text{for $\hat \zeta\leq \zeta\leq \tilde{\zeta}$},\\
y(\zeta)\leq\frac{a_m}{2}(e^{-\tilde{\zeta}}-e^{-\zeta})e^{\frac{3}{2}\zeta}\leq 0\quad\text{for $\zeta\geq \tilde{\zeta}$}.
\end{cases}
\end{equation*}
\end{enumerate}
}
\end{lem}
\begin{proof} Since the proof of Lemma \ref{lem7-2} is just the same as in Lemma \ref{lem7-1}, we omit the details here.
\end{proof}

Based on Lemma \ref{lem7-1} and Lemma \ref{lem7-2}, we now establish the following results.
\begin{lem}\label{lem7-3} {\it For $1\leq m\leq n-1$ and each forward characteristics $y(\zeta):=y(\zeta; y_0, \zeta_0)$
defined  by \eqref{vii1-0},  when the function $\mathcal{D}(z, \zeta)$ satisfies that for some positive constant $c_0$,
\begin{equation}\label{vii-2}
|\mathcal{D}(z, \zeta)|\leq c_0\eta^{-\frac{\kappa}{2}}(z)\ (0<\kappa<1),
\end{equation}
then
\begin{equation}\label{vii-3}
\int_{\zeta_0}^{+\infty}|\mathcal{D}(y(\zeta), \zeta)|d\zeta\leq \frac{16 c_0}{\kappa (1-\kappa)|a_m|^{\kappa}}e^{-\frac{\kappa}{2}\zeta_0}.
\end{equation}}
\end{lem}
\begin{proof} We only consider the {\bf Case} 2 in Lemma \ref{lem7-1}. The estimate \eqref{vii-3} for other {\bf Cases} in Lemma \ref{lem7-1} and Lemma \ref{lem7-2} can be done analogously. In the present situation,  we choose $\zeta_{1}, \zeta_{2}\in [-\log\varepsilon, +\infty)$ such that
 \begin{equation*}
 e^{-\zeta_{1}}=\min\{2e^{-\zeta^{*}}, e^{-\zeta_0}\},\  e^{-\zeta_{2}}=\frac{1}{2}e^{-\zeta^{*}}.
 \end{equation*}
This implies $\zeta_0\leq \zeta_1\leq \zeta^*<\zeta_2$. Then it is derived from \eqref{vii-2} and {\bf Case} 2
in Lemma \ref{lem7-1} that
 \begin{equation*}\label{vii-4}\begin{aligned}
 &\int_{\zeta_{0}}^{+\infty}|\mathcal{D}(y(\zeta), \zeta)|d\zeta\\
 \leq& \frac{2 c_0}{|a_m|^{\kappa}}\left(\int_{\zeta_{0}}^{\zeta_{1}}\frac{e^{-\frac{3}{2}\kappa \zeta}}{|e^{-\zeta}-e^{-\zeta^{*}}|^{\kappa}}d\zeta+\int_{\zeta_{1}}^{\zeta_{2}}\frac{e^{-\frac{3}{2}\kappa \zeta}}{|e^{-\zeta}-e^{-\zeta^{*}}|^{\kappa}}d\zeta+\int_{\zeta_{2}}^{+\infty}\frac{e^{-\frac{3}{2}\kappa \zeta}}{|e^{-\zeta}-e^{-\zeta^{*}}|^{\kappa}}ds\right)\\
 \leq & \frac{2 c_0}{|a_m|^{\kappa}}\left(2\int_{\zeta_{0}}^{\zeta_{1}}e^{-\frac{\kappa}{2}\zeta}d\zeta+2 e^{\kappa \zeta^{*}}\int_{\zeta_{2}}^{+\infty}e^{-\frac{3}{2}\kappa \zeta}d\zeta+4 e^{-\frac{\kappa}{2}\zeta^{*}}\int_{\frac{1}{2}}^{2}|1-t|^{-\kappa}dt\right)\\
\leq& \frac{16 c_0}{\kappa (1-\kappa)|a_m|^{\kappa}}e^{-\frac{\kappa}{2}\zeta_{0}}.
 \end{aligned}
 \end{equation*}
Therefore, the estimate \eqref{vii-3} holds for the case $i_0\leq m\leq n-1$ in \eqref{vii-1} and {\bf Case} 2
in Lemma \ref{lem7-1}.\end{proof}

\subsection{Auxiliary analysis}

As in \cite{PDL1}, we will apply the decomposition \eqref{ii-17} and the reduced system \eqref{ii-18}-\eqref{ii-19}
to establish the related estimates for the good components of $W$. To this end, we first show the
relation between $\partial_y^{\mu}W$ and $W_{\mu}^m$ ($1\leq m\leq n$).

Due to \eqref{i-71}, \eqref{ii-3}, \eqref{iv-1a}, \eqref{iv-2a} and \eqref{iv-3a}, one has $\ell_m^n(w)=0$
and $\|\ell_m(w)\|_{L^{\infty}}\leq 2$ for $1\leq m \leq n-1$. Combining this with \eqref{ii-17} yields
\begin{equation}\label{vii-401}\begin{cases}
|W_{\mu}^m|=|\ell_m(w)\cdot \partial_y^{\mu}W|\leq 2\sum\limits_{j=1}^{n-1}|\partial_y^{\mu}W_j|\ (1\leq m\leq n-1),\\
|W_{\mu}^n|=|\ell_n(w)\cdot \partial_y^{\mu}W|\leq 2\sum\limits_{j=1}^{n}|\partial_y^{\mu}W_j|.
\end{cases}
\end{equation}

In addition, by $\|\gamma_j(w)\|=1$ for $1\leq j\leq n$ and $\gamma_k^n(w)=0$ for $1\leq k\leq n-1$,
then it follows from \eqref{ii-17} that
\begin{equation}\label{vii-402}\begin{cases}
|\partial_y^{\mu}W_j|=|\sum\limits_{m=1}^{n}W_{\mu}^m\gamma_m^j(w)|\leq \sum\limits_{k=1}^{n-1}|W_{\mu}^k|\ (1\leq j\leq n-1),\\
|\partial_y^{\mu}W_n|=|\sum\limits_{m=1}^{n}W_{\mu}^m\gamma_m^n(w)|\leq |W_{\mu}^n|.
\end{cases}
\end{equation}

\subsection{Bootstrap estimates on $W_j\ (j\neq n)$}\label{vii-b}

First, by \eqref{i-7}, \eqref{ii-9b}, \eqref{iv-1a}, \eqref{iv-2} and \eqref{iv-3a}-\eqref{iv-3b} with $\nu=\frac{1}{3}$,
it is derived from \eqref{ii-18} and \eqref{ii-3} that
\begin{equation}\label{vii-5}\begin{aligned}
&\partial_s\gamma_{j}(w)+((\frac{3}{2}y-e^{\frac{s}{2}}\beta_{\tau}\dot{\xi}(t))I_{n}
+e^{\frac{s}{2}}\beta_{\tau}A(w))\partial_y\gamma_{j}(w)\\
=&\frac{\partial\gamma_j(w)}{\partial w}\left(\partial_s W+(\frac{3}{2}y-e^{\frac{s}{2}}\beta_{\tau}\dot{\xi}(t))\partial_y W\right)+e^{\frac{s}{2}}\beta_{\tau}A(w)\frac{\partial\gamma_j(w)}{\partial w}\partial_y W\\
=&e^{\frac{s}{2}}\beta_{\tau}[A(w), \frac{\partial\gamma_j(w)}{\partial w}]\partial_y W
\end{aligned}
\end{equation}
and
\begin{equation}\label{vii3-1}
\left|e^{\frac{s}{2}}\beta_{\tau}[A(w), \frac{\partial\gamma_j(w)}{\partial w}]\partial_y W\right|\leq M^2 \eta^{-\frac{1}{3}}(y)(1-\delta_{j}^n),
\end{equation}
where $[A, B]=AB-BA$ for two $n\times n$ matrices $A$ and $B$.

For any $(y, s)\in \mathbb{R}\times [-\log\varepsilon, +\infty)$, the backward characteristics $y(\zeta):=y(\zeta; y, s)$
of \eqref{ii-18} which starts from $(y_0(y, s), -\log\varepsilon)$, is defined as
\begin{equation}\label{vii-6}\begin{cases}
\dot{y}(\zeta)=\frac{3}{2}y(\zeta)+e^{\frac{\zeta}{2}}\beta_{\tau}(\cdot)(\mu_m(w)-\dot{\xi}(\cdot))(y(\zeta), \zeta),\\
y(-\log\varepsilon)=y_0(y, s).
\end{cases}
\end{equation}
Then it is derived from \eqref{ii-18} with $\mu=0$ and \eqref{vii-6} for $1\leq m\leq n-1$ that
\begin{equation}\label{vii-8}
W_{0}^m(y, s)=W_{0}^m(y_0(y, s), -\log\varepsilon)+\int_{-\log\varepsilon}^{s}\mathbb{F}_{0}^{m}(y(\zeta), \zeta)d\zeta,
\end{equation}
where $\mathbb{F}_{0}^m$ is given in \eqref{ii-18}, and it comes from $F_0=0$ in \eqref{ii-11a} and \eqref{vii-5}-\eqref{vii3-1} that
\begin{equation}\label{vii-9}
|\mathbb{F}_0^m(y(\zeta), \zeta)|\leq M^3 \eta^{-\frac{1}{3}}(y(\zeta))\sum\limits_{j=1}^{n-1}\|W_{0}^j\|_{L^{\infty}}.
\end{equation}
In addition, by Lemma \ref{lem7-3}  with $\kappa=\frac{2}{3}$ and \eqref{vii-8}-\eqref{vii-9},  we arrive at
\begin{equation}\label{vii-10}\begin{aligned}
|W_{0}^m(y, s)|&\leq \|W_{0}^m(\cdot, -\log\varepsilon)\|_{L^{\infty}}
+M^4 \varepsilon^{\frac{1}{3}}\sum\limits_{j=1}^{n-1}\|W_{0}^j\|_{L^{\infty}}\\
&\leq \|W_{0}^m(\cdot, -\log\varepsilon)\|_{L^{\infty}}
+\varepsilon^{\frac{1}{4}}\sum\limits_{j=1}^{n-1}\|W_{0}^j\|_{L^{\infty}}.
\end{aligned}
\end{equation}
On the other hand, due to the arbitrariness of $(y, s)$, summing $m$ in both sides of \eqref{vii-10} from $1$ to $n-1$
yields
\begin{equation}\label{vii-11}
\sum\limits_{m=1}^{n-1}\|W_{0}^m\|_{L^{\infty}}\leq 2\sum\limits_{m=1}^{n-1}\|W_{0}^m(\cdot, -\log\varepsilon)\|_{L^{\infty}}.
\end{equation}
Then it follows from \eqref{vii-11}, \eqref{vii-401}-\eqref{vii-402} and \eqref{iii-8a} that
\begin{equation}\label{vii-12}
|W_j(y, s)|\leq 4n\sum\limits_{k=1}^{n-1}\|W_k(\cdot,-\log\varepsilon)\|_{L^{\infty}}\leq 4n^2\varepsilon\ (1\leq j\leq n-1).
\end{equation}

\subsection{Bootstrap estimates on $\partial_y W_j\ (j\neq n)$}\label{vii-c}

With the definition \eqref{vii-6},  it is derived from \eqref{ii-18} for $\mu=1$ and $1\leq m\leq n-1$ that
\begin{equation}\label{vii-13}
e^{\frac{3}{2}s}W_{1}^m(y, s)=\varepsilon^{-\frac{3}{2}}W_{1}^m(y_0(y, s), -\log\varepsilon)+\int_{-\log\varepsilon}^{s}e^{\frac{3}{2}\zeta}\mathbb{F}_1^m(y(\zeta), \zeta)d\zeta.
\end{equation}

Since $a_{in}(w)=0$ and $\ell_i^n(w)=0$ for $1\leq i\leq n-1$ (see \eqref{i-7b} and \eqref{i-71}),
it follows from \eqref{ii-18}, \eqref{vii-402}-\eqref{vii-5},  \eqref{iv-2}-\eqref{iv-3} and \eqref{ii-11a} with $\mu=1$ that
\begin{equation}\label{vii-14}
|\mathbb{F}_1^m(y(\zeta), \zeta)|\leq M^3 \eta^{-\frac{1}{3}}(y(\zeta))\sum\limits_{j=1}^{n-1}|W_{1}^j|(y(\zeta), \zeta).
\end{equation}

As in \eqref{vii-11}, combining \eqref{vii-13}-\eqref{vii-14} with Lemma \ref{lem7-3} yields that for $1\leq m\leq n-1$,
\begin{equation}\label{vii-15}\begin{aligned}
e^{\frac{3}{2}s}|W_{1}^m|(y, s)&\leq \varepsilon^{-\frac{3}{2}}\|W_{1}^m(\cdot, -\log\varepsilon)\|_{L^{\infty}}+M^4\varepsilon^{\frac{1}{3}}\sum\limits_{j=1}^{n-1}\|e^{\frac{3}{2}\varsigma}W_{1}^j(z, \varsigma)\|_{L^{\infty}_{z, \varsigma}}\\
&\leq \varepsilon^{-\frac{3}{2}}\|W_{1}^m(\cdot, -\log\varepsilon)\|_{L^{\infty}}+\varepsilon^{\frac{1}{4}}\sum\limits_{j=1}^{n-1}\|e^{\frac{3}{2}\varsigma}W_{1}^j(z, \varsigma)\|_{L^{\infty}_{z, \varsigma}}.
\end{aligned}
\end{equation}
Similarly to \eqref{vii-11}-\eqref{vii-12},  it is derived from \eqref{vii-15} and \eqref{iii-8b} with $\nu=0$ that
\begin{equation}\label{vii-16}
|\partial_y W_j(y, s)|\leq 4n^2 e^{-\frac{3}{2}s}\ (1\leq j\leq n-1).
\end{equation}
In addition, as in \eqref{vi-36}, by \eqref{iv-6} and \eqref{vii-16}, we have that
for $2\leq\mu\leq 4$ and $\alpha=\frac{\mu-1}{\mu_0-\frac{1}{2}}\in (0, \frac{6}{11}]$,
\begin{equation}\begin{aligned}\label{vii-17}
\|\partial_y^{\mu}W_k(\cdot, s)\|_{L^{\infty}}&\leq M^{\frac{1}{20}}\|\partial_y^{\mu_0}W_k(\cdot, s)\|_{L^2}^{\alpha}\|\partial_y W_k(\cdot, s)\|_{L^{\infty}}^{1-\alpha}\\
&\leq (4n^2)^{1-\alpha} M^{\frac{1}{20}+\alpha} e^{-\frac{3}{2}s}\leq M^{\frac{3}{5}}e^{-\frac{3}{2}s}\ (1\leq k\leq n-1).
\end{aligned}
\end{equation}

\subsection{Weighted bootstrap estimates of the good components}\label{vii-d}

For $a=\max\limits_{1\leq i\leq n-1}\{5|\mu_m(0)|+1\}$, set the domain $D^0$ as
\begin{equation}\label{vii-18}
D^0=\{(y, s): |y|<4ae^{\frac{s}{2}}, -\log\varepsilon\leq s<+\infty\}.
\end{equation}
Then $D^{\pm}\subset D^0$ for the domains $D^{\pm}$ defined in Lemma \ref{lem7-1} and Lemma \ref{lem7-2}.

When $(y, s)\in D^0$, it is derived from \eqref{vii-16}-\eqref{vii-17} that
\begin{equation}\label{vii-19}\begin{cases}
\sum\limits_{j=1}^{n-1}|\partial_y W_j(y, s)|\leq 4n^2 e^{-\frac{3}{2}s}\leq M^{\frac{1}{5}} e^{(\nu-\frac{3}{2})s}\eta^{-\nu}(y)\ (0\leq \nu\leq \frac{1}{3}]),\\[4mm]
\sum\limits_{j=1}^{n-1}|\partial_y^2 W_j(y, s)|\leq M^{\frac{2}{5}} e^{-\frac{3}{2}s}\leq M^{\frac{3}{5}}e^{-\frac{7}{6}s}\eta^{-\frac{1}{3}}(y).
\end{cases}
\end{equation}

Next we derive the weighted estimates on the good components of $W$ when $(y, s)\notin D^0$.
In this situation, as in \eqref{vii-6}, for each $1\leq m\leq n-1$, the backward characteristics
$y(\zeta):=y(\zeta; y, s)$ of \eqref{ii-19} which starts from $(y_0(y, s), \zeta_0)\notin D^0$ is defined as
\begin{equation}\label{vii-20}\begin{cases}
\dot{y}(\zeta)=\frac{3}{2}y(\zeta)+e^{\frac{\zeta}{2}}\beta_{\tau}(\mu_m(w)-\dot{\xi}(t))(y(\zeta), \zeta),\ \zeta_0\leq\zeta\leq s,\\
y(\zeta_0)=y_0(y, s),
\end{cases}
\end{equation}
where either $\zeta_0=-\log\varepsilon$ or $(y_0(y, s),\zeta_0)\in\partial D^0$.

Note that $y(\zeta)$ in \eqref{vii-20} has the following expression
\begin{equation}\label{vii-21}
y(\zeta)e^{-\frac{3}{2}\zeta}=y_0(y, s)e^{-\frac{3}{2}\zeta_0}+\int_{\zeta_0}^{\zeta}e^{-\alpha}\beta_{\tau}(\mu_m(w)-\dot{\xi}(t))(y(\alpha), \alpha)d\alpha.
\end{equation}
It follows from \eqref{vii1-1} and \eqref{vii-21} with $|y_0(y, s)|\geq 4a e^{\frac{\zeta_0}{2}}$ that
for $\zeta_0\leq \zeta\leq s$,
\begin{equation}\label{vii-22}
|y(\zeta)|\geq e^{\frac{3}{2}\zeta}(4ae^{-\zeta_0}-2|a_m|(e^{-\zeta_0}-e^{-\zeta}))\geq e^{\frac{3}{2}\zeta-\zeta_0}.
\end{equation}

Based on the definition of $y(\zeta)=y(\zeta; y, s)$ in \eqref{vii-20}, we derive from \eqref{ii-19} that
\begin{equation}\label{vii5-1}\begin{aligned}
\left[\eta^{\nu}W_{\mu}^m\right](y, s)&=[\eta^{\nu} W_{\mu}^m](y_0, \zeta_0)\exp(-\int_{\zeta_0}^{s}\mathbb{D}_{\mu,\nu}^m(y(\alpha), \alpha)d\alpha)\\
&+\int_{\zeta_0}^{s}[\eta^{\nu}\mathbb{F}_{\mu}^{m}(y(\zeta), \zeta)
\exp(-\int_{\zeta}^{s}\mathbb{D}_{\mu,\nu}^m(y(\alpha), \alpha)d\alpha))d\zeta
\end{aligned}
\end{equation}
and
\begin{equation}\label{vii-23}\begin{aligned}
\mathbb{D}_{\mu, \nu}^m(y(\zeta), \zeta)&=\frac{3\mu}{2}-3\nu+\frac{3\nu}{1+y^2}-\frac{2\nu y}{1+y^2}e^{\frac{s}{2}}\beta_{\tau}(\mu_m(w)-\dot{\xi}(t))\bigl|_{(y, s)=(y(\zeta), \zeta)}\\
&:=\frac{3\mu}{2}-3\nu+\mathbb{D}_{\nu}^{m}(y(\zeta), \zeta).
\end{aligned}
\end{equation}
It is derived from \eqref{vii-22}, \eqref{vii-23} and \eqref{vii1-1} that
\begin{equation}\label{vii-25}
\int_{\zeta_0}^{+\infty}|\mathbb{D}_{\nu}^{m}|(y(\zeta), \zeta)d\zeta\leq 6|\nu|a\int_{\zeta_0}^{+\infty}
e^{\zeta_0-\zeta}d\zeta\leq 6|\nu| a.
\end{equation}
Due to \eqref{vii5-1} and \eqref{vii-25},  when $\frac{3\mu}{2}-3\nu>0$, we obtain
\begin{equation}\label{vii-26}\begin{aligned}
&e^{(\frac{3\mu}{2}-3\nu)s}\eta^{\nu}(y)|W_{\mu}^{m}(y, s)|\leq e^{6|\nu|a}\biggl( e^{(\frac{3\mu}{2}-3\nu)\zeta_0}\eta^{\nu}(y_0(y, s))|W_{\mu}^{m}(y_0(y, s), \zeta_0)|\\
&\qquad\qquad\qquad+\int_{\zeta_0}^{s}e^{(\frac{3\mu}{2}-3\nu)\zeta}\eta^{\nu}(y(\zeta))|\mathbb{F}_{\mu}^m|(y(\zeta), \zeta)d\zeta\biggr)
\end{aligned}
\end{equation}
and
\begin{equation}\label{vii-27}\begin{aligned}
&e^{(\frac{7}{6}-\nu^+)s}\eta^{\frac{1}{3}}(y)|W_{2}^m(y, s)|\leq e^{2a}\biggl(e^{(\frac{7}{6}-\nu^+)\zeta_0}|\eta^{\frac{1}{3}}(y_0(y, s))|W_{2}^m(y_0(y, s), \zeta_0)|\\
&\qquad\qquad\qquad+\int_{\zeta_0}^{s}e^{(\frac{7}{6}-\nu^+)\zeta}\eta^{\frac{1}{3}}(y(\zeta))|\mathbb{F}_2^m|(y(\zeta), \zeta)dl\biggr)\ (0\leq\nu^+<\frac{7}{6}).
\end{aligned}
\end{equation}
Here we point out that the factor $\frac{7}{6}-\nu^+$ appeared in \eqref{vii-27} for $0\leq \nu^+<\frac{7}{6}$
is due to \eqref{vii-19}.

With respect to $\mathbb{F}_2^m$, similarly to the argument for \eqref{vii-14}, it
is derived from \eqref{ii-11a}, \eqref{ii-18}, \eqref{vii-401}, \eqref{vii-16} and \eqref{iv-2}-\eqref{iv-3} that
\begin{equation}\label{vii-28}
|\mathbb{F}_2^m(y(\zeta), \zeta)|\leq M^3 (e^{-\zeta}+\eta^{-\frac{1}{3}}(y(\zeta)))\sum\limits_{j=1}^{n-1}(|W_2^j(y(\zeta), \zeta)|
+|W_1^j(y(\zeta), \zeta)|).
\end{equation}

For $\mu=1$, $0\leq \nu\leq \frac{1}{3}$ and $1\leq m\leq n-1$ in \eqref{vii-26}, we obtain from \eqref{iii-8b}, \eqref{vii-402}, \eqref{vii-14}, \eqref{vii-19}, \eqref{vii-20} and Lemma \ref{lem7-3} that
\begin{equation*}\label{vii-29}\begin{aligned}
e^{(\frac{3}{2}-3\nu)s}\eta^{\nu}(y)|W_1^m(y, s)|&\leq e^{2a}\left(M^{\frac{2}{5}}+M^3(\varepsilon+\varepsilon^{\frac{1}{3}})\sum\limits_{j=1}^{n-1}\|W_1^j(z, \tau)\|_{L^{\infty}(\{|z|\geq 4a e^{\frac{\tau}{2}}\})}\right)\\
&\leq M^{\frac{2}{5}}e^{2a}+\varepsilon^{\frac{1}{4}}\sum\limits_{j=1}^{n-1}\|W_1^j(z, \tau)\|_{L^{\infty}(\{|z|\geq 4a e^{\frac{\tau}{2}}\})}.
\end{aligned}
\end{equation*}
Combining this with \eqref{vii-402} shows that for $1\leq j\leq n-1$,
\begin{equation}\label{vii5-2}
|\partial_y W_j(y, s)|\leq \sum\limits_{m-1}^{n-1}|W_1^m(y, s)|\leq  M^{\frac{3}{5}}e^{(3\nu-\frac{3}{2})s}\eta^{-\nu}(y)\ (0\leq \nu\leq \frac{1}{3}).
\end{equation}

For the estimates of $\partial_y^2 W_m$ with $1\leq m\leq n-1$, it follows from \eqref{iii-8c}, \eqref{vii-401}, \eqref{vii-19}, \eqref{vii-27}-\eqref{vii5-2} and Lemma \ref{lem7-3} that
\begin{equation*}\label{vii-30}\begin{aligned}
&e^{(\frac{7}{6}-\nu^+)s}\eta^{\frac{1}{3}}(y)|W_2^m(y, s)|\leq 2e^{2a}M^{\frac{3}{5}}\varepsilon^{\nu^+}+M^{5}(\varepsilon+\varepsilon^{\frac{1}{3}})\\
&\qquad\qquad+M^4(\varepsilon+\varepsilon^{\frac{1}{3}})\sum\limits_{j=1}^{n-1}\|e^{(\frac{7}{6}-\nu^+)\tau}\eta^{\frac{1}{3}}(z)W_2^j(z, \tau)\|_{L^{\infty}(\{|z|\geq 4ae^{\frac{\tau}{2}}\})}\  (0<\nu^+<\frac{7}{6}).
\end{aligned}
\end{equation*}
Together with \eqref{vii-402}, this yields that for $1\leq j\leq n-1$,
\begin{equation}\label{vii-30}
|\partial_y^2 W_j(y, s)|\leq \sum\limits_{m=1}^{n-1}|W_2^m(y, s)|\leq M^{\frac{4}{5}}e^{(\nu^+-\frac{7}{6})s}\eta^{-\frac{1}{3}}(y)\ (0<\nu^+<\frac{7}{6}).
\end{equation}

\section{Bootstrap estimates on the modulation variables}\label{viii}

For $\dot{\kappa}(t)$ and $\dot{\tau}(t)$, it follows from \eqref{i-7}, \eqref{ii-3}, \eqref{ii-12a}, \eqref{ii-12c}, \eqref{iii-6} and \eqref{iv-1}-\eqref{iv-3} that
\begin{equation}\label{viii-2}\begin{aligned}
|\dot{\kappa}(t)|&\leq e^{s}|\mu_n(w^0)-\dot{\xi}(t)|+e^{\frac{3s}{2}}\sum\limits_{j\neq n}|a_{nj}(w^0)||(\partial_y W_j)^0|\\
&\leq e^{s}|(\partial_y^2 \mu_n(w))^0|+e^{\frac{3s}{2}}\sum\limits_{k=0, 2}\sum\limits_{j\neq n}|(\partial_y^k(a_{nj}(w)\partial_y W_j))^0|\\
&\leq |(\partial_{w_n w_n}(\mu_n(w)))^0|+M\varepsilon^{\frac{1}{3}}\leq M^{\frac{1}{4}}\leq M^{\frac{1}{2}}
\end{aligned}
\end{equation}
and
\begin{equation}\label{viii-3}\begin{aligned}
|\dot{\tau}(t)|&\leq |1-(\partial_{w_n}\mu_n(w))^0|+M^2 e^{-\frac{s}{2}}\leq M^2\varepsilon^{\frac{1}{2}}\leq \varepsilon^{\frac{1}{3}}\leq 2\varepsilon^{\frac{1}{3}}.
\end{aligned}
\end{equation}
By the coordinate transformation \eqref{ii-2}, one has $t=t(s)$ and
\begin{equation}\label{viii-4}
\frac{d}{ds}(\kappa, \tau, \xi)(t(s))=(\dot{\kappa}(t), \dot{\tau}(t), \dot{\xi}(t))\frac{e^{-s}}{\beta_{\tau}}.
\end{equation}
Combining \eqref{viii-4} with \eqref{viii-2}-\eqref{viii-3} and \eqref{iii-14} shows
\begin{equation}\label{viii-5}
|\kappa(t)-\kappa_0\varepsilon^{\frac{1}{3}}|=|\int_{-\log\varepsilon}^{s}(\dot{\kappa}(t)\frac{e^{-s}}{\beta_{\tau}})ds|\leq |\kappa_0\varepsilon|+2M^{\frac{1}{4}}\varepsilon\leq M^{\frac{1}{2}} \varepsilon
\end{equation}
and
\begin{equation}\label{viii-51}
|\tau(t)|=|\tau(-\varepsilon)+\int_{-\log\varepsilon}^{s}(\dot{\tau}(t)\frac{e^{-s}}{\beta_{\tau}})ds|\leq 2\varepsilon^{\frac{4}{3}}.
\end{equation}

With respect to $\xi(t)$, by \eqref{i-7}, \eqref{ii-3},  \eqref{iii-6}, \eqref{vii-12}, \eqref{vii-16}-\eqref{vii-17} and \eqref{viii-3}-\eqref{viii-5}, then $\xi(t)$ in \eqref{ii-12c} satisfies
\begin{equation}\label{viii-1}
|\dot{\xi}(t)|\leq |\mu_{n}(w^0)|+\frac{1}{6}|(\partial^2 \mu_n(w))^0|+\frac{1}{6}\sum\limits_{j\neq n}e^{\frac{s}{2}}|(\partial^2(a_{nj}(w)\partial_y W_j ))^0|\leq M^{\frac{3}{4}}\varepsilon\leq 2M^{\frac{3}{4}}\varepsilon
\end{equation}
and
\begin{equation}\label{viii-11}
|\xi(t)|=\left|\xi(-\varepsilon)+\int_{-\log\varepsilon}^{s}(\dot{\xi}(t)\frac{e^{-s}}{\beta_{\tau}})ds\right|\leq 2M^{\frac{3}{4}}\varepsilon^2.
\end{equation}

\section{Weighted energy estimates}\label{v}

In the section, we establish the spatial $L^2-$energy estimates of $\partial_y^{\mu_0}W$ when $\mu_0$  satisfies \eqref{iv-50}.

\begin{thm}\label{thm8-1} {\it For $\mu_0$ satisfying \eqref{iv-50}, under the assumptions \eqref{iv-1}-\eqref{iv-3} and \eqref{iv-6}, one has
\begin{equation}\label{v-1}
\sum\limits_{m=1}^{n-1}\|\partial_y^{\mu_0}W_m(\cdot, s)\|_{L^2(\mathbb{R})}\leq M^{\frac{1}{2}} e^{-\frac{3}{2}s},\ \|\partial_y^{\mu_0}W_{n}(\cdot, s)\|_{L^2(\mathbb{R})}\leq M^{\frac{1}{2}} e^{-\frac{s}{2}}.
\end{equation}}
\end{thm}

Based on the expansion \eqref{ii-17} and the assumption \eqref{iv-6}, we have from \eqref{vii-401}
and \eqref{v-1} that
\begin{equation}\label{v-2}
\sum\limits_{m=1}^{n-1}\|W_{\mu_0}^m(\cdot, s)\|_{L^2(\mathbb{R})}\leq 2nM e^{-\frac{3}{2}s},\ \|W_{\mu_0}^n(\cdot, s)\|_{L^2(\mathbb{R})}\leq 2M e^{-\frac{s}{2}}.
\end{equation}

\subsection{Framework for energy estimates}
To prove Theorem \ref{thm8-1}, we first establish the following framework for energy estimates:

\begin{lem}\label{lem8-1} With $\mu_0$ satisfying \eqref{iv-50}, under the assumption \eqref{v-2} (or see \eqref{iv-6}), one has
\begin{enumerate}[(1)]
\item When $1\leq m\leq n-1$, for any Lipschitz continuous function $q_m(y)$ and with the notation
\begin{equation}\label{v-3}
Q_m(y, s)=-(\frac{3}{2}y+e^{\frac{s}{2}}\beta_{\tau}(\mu_m(w)-\dot{\xi}(t)))q_m'(y),
\end{equation}
then
\begin{equation}\label{v-4}\begin{aligned}
&\frac{d}{ds}\int_{\mathbb{R}}e^{q_m(y)}|W_{\mu_0}^m|^2(y, s)dy+\int_{\mathbb{R}}(3\mu_0-\frac{5}{2}-e^{\frac{s}{2}}\beta_{\tau}\partial_y\mu_m(w)+Q_{m})e^{q_m(y)}|W_{\mu_0}^m|^2(y, s)dy\\
\leq& \int_{\mathbb{R}}e^{q_m(y)}|\mathbb{F}_{\mu_0}^m|^2(y, s)dy.
\end{aligned}
\end{equation}
\item  For $W_{\mu_0}^n$,
\begin{equation}\label{v-5}\frac{d}{ds}\int_{\mathbb{R}}|W_{\mu_0}^n|^2(y, s)dy+\int_{\mathbb{R}}(2\mu_0-\frac{5}{2}-e^{\frac{s}{2}}\beta_{\tau}\partial_y\mu_n(w))|W_{\mu_0}^n|^2(y, s)dy\leq \frac{1}{\mu_0+1}\int_{\mathbb{R}}|\mathbb{F}_{\mu_0}^n|^2(y, s)dy.
\end{equation}
\end{enumerate}
\end{lem}

\begin{proof}  For $1\leq m\leq n-1$, multiplying both sides of \eqref{ii-18} with $\mu=\mu_0$
by $2e^{q_m(y)}W_{\mu_0}^m$ and integrating on $\mathbb{R}$ yield
\begin{equation}\label{v-6}\begin{aligned}
&\frac{d}{ds}\int_{\mathbb{R}}e^{q_m(y)}|W_{\mu_0}^m|^2(y, s) dy+\int_{\mathbb{R}}(3\mu_0-\frac{3}{2}-e^{\frac{s}{2}}\beta_{\tau}\partial_y\mu_m(w)+Q_m)e^{q_m(y)}|W_{\mu_0}^m|^2(y, s) dy\\
=&2\int_{\mathbb{R}}e^{q_m(y)}(\mathbb{F}_{\mu_0}^m\cdot W_{\mu_0}^m)(y, s) dy\\
\leq &\int_{\mathbb{R}}e^{q_m(y)}|W_{\mu_0}^m|^2(y, s) dy+\int_{\mathbb{R}}e^{q_m(y)}|\mathbb{F}_{\mu_0}^m|^2(y, s) dy.
\end{aligned}
\end{equation}

Then \eqref{v-4} comes from \eqref{v-3} and \eqref{v-6}. The estimate \eqref{v-5}
 can be obtained by a standard energy estimate associated with the equation \eqref{ii-18} for $m=n$ and $\mu=\mu_0$.
\end{proof}

Next we analyze the structure of $\mathbb{F}_{\mu_0}^m$.
\begin{lem}\label{lem8-2} For $\mu_0$ satisfying \eqref{iv-50} and $1\leq m\leq n-1$,
then $\mathbb{F}_{\mu_0}^m$ in \eqref{ii-18} satisfies
\begin{equation}\label{v-7}\begin{aligned}
\int_{\mathbb{R}}|\mathbb{F}_{\mu_0}^m|^2 (y, s) dy&
\leq 4nM^{-\frac{1}{16}}\sum\limits_{j=1}^{n-1}\int_{\mathbb{R}}|W_{\mu_0}^j|^2(y, s)dy+2nM^{\frac{1}{16}}\sum\limits_{j=1}^{n-1}\int_{\mathbb{R}}|\eta^{-\frac{1}{3}}(y)W_{\mu_0}^j|^2(y, s)dy\\
&+2n M^{\frac{1}{4}}e^{-2s}\int_{\mathbb{R}}|W_{\mu_0}^n|^2(y, s)dy+e^{-(3+\frac{1}{4})s}.
\end{aligned}
\end{equation}
\end{lem}

\begin{proof} First, due to \eqref{ii-3}, \eqref{vi2-4}, \eqref{vi-35}, \eqref{vii-12}, \eqref{vii-16} and \eqref{viii-5}, the estimate in \eqref{vii3-1} can be improved as
\begin{equation}\label{v-72}
|e^{\frac{s}{2}}\beta_{\tau}[A(w), \frac{\partial\gamma_j(w)}{\partial w}]\partial_y W|\leq M^{\frac{1}{64}}\eta^{-\frac{1}{3}}(y)(1-\delta_j^n).
\end{equation}
Thus, we can obtain from \eqref{ii-18}, \eqref{vii-5} and \eqref{v-72} that
\begin{equation}\label{v-71}\begin{aligned}
\int_{\mathbb{R}}|\mathbb{F}_{\mu_0}^m|^2(y, s)dy&\leq M^{\frac{1}{16}}e^{-s}\sum\limits_{j=1}^{n-1}\int_{\mathbb{R}}|W_{\mu_0}^j|^2(y, s)dy+M^{\frac{1}{16}}\sum\limits_{j=1}^{n-1}\int_{\mathbb{R}}|\eta^{-\frac{1}{3}}(y)W_{\mu_0}^j|^2(y, s)dy\\
&+2\int_{\mathbb{R}}|\ell_m\cdot F_{\mu_0}|^2(y, s)dy.
\end{aligned}
\end{equation}
Due to $\ell_k^n(w)=0$ and $a_{kn}(w)=0$ ($1\leq k\leq n-1$) by \eqref{i-7}-\eqref{i-71}, for $F_{\mu_0}$ in \eqref{ii-11a}, one
then has
\begin{equation}\label{v-8}\begin{aligned}
|\ell_m\cdot F_{\mu_0}|
\leq& \sum\limits_{1\leq \beta\leq \mu_0} 2C_{\mu_0}^{\beta}e^{\frac{s}{2}}|\ell_m\cdot \partial_y^{\beta}A(w)\partial_y^{\mu_0+1-\beta} W|\\
\leq&\sum\limits_{1\leq\beta\leq \mu_0}\sum\limits_{1\leq q\leq \mu_0-\beta+1}2 C_{\mu_0}^{\beta}\|\partial_w^q A(w)\|_{L^{\infty}}e^{\frac{s}{2}}\sum\limits_{k=1}^{n-1}I_{\beta q k}\\
\leq&\left(M^{\frac{1}{16}}e^{s}\sum\limits_{1\leq\beta\leq\mu_0}\sum\limits_{1\leq q\leq \mu_0-\beta+1}\sum\limits_{1\leq k\leq n-1}I_{\beta q k}^2\right)^{1/2},
\end{aligned}
\end{equation}
where  the last inequality comes from \eqref{ii-3}, \eqref{iv-1}-\eqref{iv-3}, and $I_{\beta q k}$ satisfies
\begin{equation}\label{v-9}
I_{\beta q k}=\sum\limits_{\gamma_1+\cdots+\gamma_q=\mu_0-\beta+1,\ \gamma_j\geq 1\ (1\leq j\leq q)}|\partial_{y}^{\gamma_1} W|\cdots|\partial_y^{\gamma_q}W|\cdot|\partial_y^{\beta}W_k|.
\end{equation}
Note that the estimates for ${I_{\beta q k}}'s$ in \eqref{v-9} are taken in Lemma \ref{lemA-4} for $1\leq k\leq n-1$.
Substituting the three type estimates in \eqref{A2-3} into \eqref{v-8} shows
\begin{equation}\label{v-16}\begin{aligned}
&\int_{\mathbb{R}}|\ell_m\cdot F_{\mu_0}|^2(y, s)dy\\
\leq &2nM^{\frac{1}{16}}\sum\limits_{j=1}^{n-1}\int_{\mathbb{R}}|\eta^{-\frac{1}{3}}(y)W_{\mu_0}^j|^2(y, s)dy+2nM^{-\frac{1}{16}}\sum\limits_{j=1}^{n-1}\int_{\mathbb{R}}|W_{\mu_0}^j|^2(y, s)dy\\
&+4nM^{\frac{3}{8}}e^{-2s}\int_{\mathbb{R}^n}|W_{\mu_0}^n|^2(y, s)dy+M^{3\mu_0+2}e^{-(3+\frac{1}{2})s}\ (1\leq m\leq n-1).
\end{aligned}
\end{equation}
Therefore, \eqref{v-7} follows from \eqref{v-71}-\eqref{v-8} and \eqref{v-16}, and then
the proof of Lemma \ref{lem8-2} is finished.\end{proof}

\begin{lem}\label{lem8-3} {\it Let $\mu_0$ satisfy  \eqref{iv-50}, then for $\mathbb{F}_{\mu_0}^n$ in \eqref{ii-18}, one has
\begin{equation}\label{v-17}
\int_{\mathbb{R}}|\mathbb{F}_{\mu_0}^n|^2(y, s)dy\leq \frac{102}{100}(\mu_0+1)^2\int_{\mathbb{R}}|W_{\mu_0}^n|^2 (y, s)dy+M^{\frac{1}{2}}e^{-s}.
\end{equation}}
\end{lem}

\begin{proof} Note that under the assumptions \eqref{i-7}-\eqref{i-71}, $\ell_n(w)$ has the decomposition
\begin{equation}\label{v-18}
\ell_n(w)={\bf e}_n^{\top}+\sum\limits_{j=1}^{n-1}c_j(w)\ell_j(w),
\end{equation}
where $c_j(w)=-{\bf e}_n^{\top}\cdot\gamma_j(w)\ (1\leq j\leq n-1)$.

Combining \eqref{v-18} with $\gamma_n(w)={\bf e}_n$ in \eqref{i-71}, \eqref{ii-11}, \eqref{ii-18} shows
\begin{equation}\label{v-19}\begin{aligned}
|\mathbb{F}_{\mu_0}^n|=&|\ell_n\cdot F_{\mu_0}|
\leq \sum\limits_{j=1}^{n-1}|c_j(w)\ell_j\cdot F_{\mu_0}|+|{\bf e}_n^{\top}\cdot F_{\mu_0}|\\
\leq& \sum\limits_{j=1}^{n-1}|c_j(w)\ell_j\cdot F_{\mu_0}|+e^{\frac{s}{2}}\beta_{\tau}\sum\limits_{1\leq\beta\leq\mu_0}C_{\mu_0}^{\beta}|\partial_y^{\beta}\mu_{n}(w)
\partial_y^{\mu_0-\beta+1}W_n|\\
&+e^{\frac{s}{2}}\beta_{\tau}\sum\limits_{j=1}^{n-1}\sum\limits_{1\leq\beta\leq\mu_0}|\partial_y^{\beta}(a_{nj}(w))
\partial_y^{\mu_0-\beta+1}W_j|:=\sum\limits_{i=1}^{5}J_i,
\end{aligned}
\end{equation}
where
\begin{equation*}\begin{cases}
J_1=e^{\frac{s}{2}}\beta_{\tau}(\mu_0+1)|\partial_y W_n\partial_y^{\mu_0}W_n|,\\
J_2=e^{\frac{s}{2}}\beta_{\tau}\sum\limits_{2\leq\beta\leq\mu_0-1}C_{\mu_0}^{\beta}|\partial_y^{\beta}W_n\partial_y^{\mu_0-\beta}W_n|,\\
J_3=\sum\limits_{j=1}^{n-1}|c_j(w)\ell_j\cdot F_{\mu_0}|,\\
J_4=e^{\frac{s}{2}}\beta_{\tau}\sum\limits_{1\leq\beta\leq\mu_0}C_{\mu_0}^{\beta}|\partial_y^{\beta}(\mu_{n}(w)-W_n)\partial_y^{\mu_0-\beta+1}W_n|,\\
J_5=e^{\frac{s}{2}}\beta_{\tau}\sum\limits_{j=1}^{n-1}\sum\limits_{1\leq\beta\leq\mu_0}|\partial_y^{\beta}(a_{nj}(w))\partial_y^{\mu_0-\beta+1}W_j|.
\end{cases}
\end{equation*}
It is derived from \eqref{v-19} that
\begin{equation}\label{v-20}
\int_{\mathbb{R}}|\mathbb{F}_{\mu_0}^n|^2(y, s)dy\leq \frac{101}{100}\int_{\mathbb{R}}J_1^2(y, s)dy+M^{\frac{1}{32}}\sum\limits_{k=2}^{5}\int_{\mathbb{R}}J_k^2(y, s)dy.
\end{equation}
In addition, it follows from \eqref{ii-3}, \eqref{vi-35} and \eqref{viii-3} that
\begin{equation}\label{v-21}
\int_{\mathbb{R}}J_1^2(y, s)dy\leq (1+\varepsilon^{\frac{1}{40}})(\mu_0+1)^2\int_{\mathbb{R}}|\partial_y^{\mu_0}W_{n}|^2 (y, s)dy.
\end{equation}

For $J_2$, by H\"older inequality, \eqref{ii-3} and \eqref{vi-35}, we have
\begin{equation}\label{v-22}\begin{aligned}
\int_{\mathbb{R}}J_2^2(y, s)dy&\leq M^{\frac{1}{8(\mu_0-1)}}\int_{\mathbb{R}}|\partial_y^{\mu_0-1}W_n|^2(y, s)ds\\
&+M^{\frac{1}{16(\mu_0-1)}}e^{s}\sum\limits_{1<\beta<\mu_0-2}\|\partial_y^{\beta}W_n(\cdot, s)\|_{L^{\frac{2(\mu_0-2)}{\beta-1}}}^2\|\partial_y^{\mu_0-\beta}W_n(\cdot, s)\|_{L^{\frac{2(\mu_0-2)}{\mu_0-\beta-1}}}^2.
\end{aligned}
\end{equation}
Note that by Lemma \ref{lemA-3},
\begin{equation*}\begin{aligned}
&\|\partial_y^{\beta}W_n(\cdot, s)\|_{L^{\frac{2(\mu_0-2)}{\beta-1}}}\leq M^{\frac{1}{64(\mu_0-1)}}\|\partial_y W_n(\cdot, s)\|_{L^{\infty}}^{\frac{\mu_0-\beta-1}{\mu_0-2}}\|\partial_y^{\mu_0-1}W_n(\cdot, s)\|_{L^2}^{\frac{\beta-1}{\mu_0-2}},\\
&\|\partial_y^{\mu_0-\beta}W_n(\cdot, s)\|_{L^{\frac{2(\mu_0-2)}{\mu_0-\beta-1}}}\leq M^{\frac{1}{64(\mu_0-1)}}\|\partial_y W_n(\cdot, s)\|_{L^{\infty}}^{\frac{\beta-1}{\mu_0-2}}\|\partial_y^{\mu_0-1}W_n(\cdot, s)\|_{L^{2}}^{\frac{\mu_0-\beta-1}{\mu_0-2}}.
\end{aligned}
\end{equation*}
Then combining these two estimates with \eqref{ii-3}, \eqref{vi-35}, \eqref{v-22} and Lemma \ref{lemA-3} yields
\begin{equation}\label{v-221}\begin{aligned}
\int_{\mathbb{R}}J_2^2(y, s)dy&\leq 2M^{\frac{1}{8(\mu_0-1)}}\int_{\mathbb{R}}|\partial_y^{\mu_0-1}W_n|^2(y, s)ds\\
&\leq M^{\frac{1}{4}(\mu_0-1)}\|\partial_y W_n(\cdot, s)\|_{L^{2}}^{\frac{2}{\mu_0-1}}\|\partial_y^{\mu_0}W_n(\cdot, s)\|_{L^2}^{2\frac{\mu_0-2}{\mu_0-1}}\\
&\leq M^{-\frac{1}{8(\mu_0-2)}}\int_{\mathbb{R}}|\partial_y^{\mu_0}W_n|^2(y, s)dy+2M^{\frac{3}{8}}e^{-s}.
\end{aligned}
\end{equation}

For $J_3$, due to $\gamma_j(0)={\bf e}_j\ (1\leq j\leq n-1)$ in \eqref{i-71}, one then derives
from \eqref{ii-3} and \eqref{iv-1}-\eqref{iv-3} that
\begin{equation}\label{v-222}
|c_j(w)|\leq M\|W(\cdot, s)\|_{L^{\infty}}\leq M^4\varepsilon^{\frac{1}{3}}\leq \varepsilon^{\frac{1}{4}}\ (1\leq j\leq n-1).
\end{equation}
On the other hand, it follows from \eqref{v-2}, \eqref{v-16} and \eqref{v-222} that
\begin{equation}\label{v-23}
\int_{\mathbb{R}}J_3^2(y, s)dy\leq \varepsilon^{\frac{1}{6}} e^{-3s}.
\end{equation}
With respect to $J_4$, one has from \eqref{v-19} that
\begin{equation}\label{w-1}\begin{aligned}
J_4=&e^{\frac{s}{2}}\beta_{\tau}\sum\limits_{1\leq \beta\leq \mu_0}C_{\mu_0}^{\beta}|\partial_y^{\beta}(\mu_n(w)-W_n)\partial_y^{\mu_0-\beta+1}W_n|\\
\leq&e^{\frac{s}{2}}\beta_{\tau}\sum\limits_{1\leq\beta\leq\mu_0}C_{\mu_0}^{\beta}|\partial_{w_n}(\mu_n(w)-W_n)|I_{\beta 1 n}\\
+& e^{\frac{s}{2}}\beta_{\tau}\sum\limits_{1\leq\beta\leq\mu_0}C_{\mu_0}^{\beta}\sum\limits_{1\leq k\leq n-1}|\partial_{w_k}(\mu_n(w)-W_n)|I_{\beta 1 k}\\
+&e^{\frac{s}{2}}\beta_{\tau}\sum\limits_{1\leq\beta\leq\mu_0}\sum\limits_{2\leq q\leq \mu_0-\beta+1}c_{\beta q n}(w)I_{\beta q n}\\
:=&\sum\limits_{1\leq \beta\leq \mu_0}(J_{41 \beta}+J_{42 \beta}+J_{43 \beta}),
\end{aligned}
\end{equation}
where $|c_{\beta q n}(w)|\leq M$ due to \eqref{ii-3} and \eqref{iv-1}-\eqref{iv-3}. Since $\partial_{w_n}\mu_n(0)=1$,
by \eqref{ii-3} and \eqref{iv-1}-\eqref{iv-3}, we have
\begin{equation}\label{w-2}
|J_{41 \beta}|\leq M\varepsilon^{\frac{1}{6}}e^{\frac{s}{2}}I_{\beta 1 n}.
\end{equation}
In the similar and easier way,
\begin{equation}\label{w-3}
|J_{42 \beta}|+|J_{43 \beta}|\leq M\sum\limits_{1\leq k\leq n-1}I_{\beta 1 k}+M^2\sum\limits_{2\leq q\leq \mu_0-\beta+1}I_{\beta q n}.
\end{equation}
With the help of Lemma \ref{lemA-4}, Lemma \ref{lemA-5} and \eqref{v-2}, we obtain from \eqref{w-1}-\eqref{w-3} that
\begin{equation}\label{w-4}
\int_{\mathbb{R}}J_4^2(y, s)dy\leq \varepsilon^{\frac{1}{10}}e^{-s}.
\end{equation}
Analogously to the treatment of $J_4$ in \eqref{w-1}-\eqref{w-4}, one has
\begin{equation}\label{v-24}
\int_{\mathbb{R}}J_5^2(y, s)dy\leq \varepsilon^{\frac{1}{10}}e^{-s}.
\end{equation}
Therefore, \eqref{v-17} comes from \eqref{v-21}, \eqref{v-221}, \eqref{v-23}, \eqref{w-4}, \eqref{v-24} and the largeness of $M$.
\end{proof}

\subsection{Energy estimates of $W_{\mu_0}^n$}

First, we close the estimate of $W_{\mu_0}^n$ in \eqref{v-2}. Substituting \eqref{v-17} into \eqref{v-5} yields
\begin{equation}\label{v-25}
\frac{d}{ds}\int_{\mathbb{R}}|W_{\mu_0}^n|^2(y, s)dy+\int_{\mathbb{R}}(\frac{98\mu_0-102}{100}-\frac{5}{2}-\beta_{\tau}e^{\frac{s}{2}}\partial_y\mu_n(w))|W_{\mu_0}^n|^2(y, s)dy\leq M^{\frac{1}{2}}e^{-s}.
\end{equation}
When $\mu_0\geq 6$, it is derived from \eqref{i-7}, \eqref{ii-3}, \eqref{vi2-4}, \eqref{vi-35}, \eqref{vii-12}, \eqref{vii-16},
\eqref{viii-3} and \eqref{viii-5} that
\begin{equation}\label{v-26}
\frac{98\mu_0-102}{100}-\frac{5}{2}-\beta_{\tau}e^{\frac{s}{2}}\partial_y\mu_n(w)
\geq\frac{486}{100}-\frac{5}{2}-M e^{-s}-(1+\varepsilon^{\frac{1}{20}})>\frac{6}{5}.
\end{equation}
In addition, one has from \eqref{v-25}-\eqref{v-26} that
\begin{equation*}
\frac{d}{ds}\int_{\mathbb{R}}|W_{\mu_0}^n|^2(y, s)dy
+\frac{6}{5}\int_{\mathbb{R}}|W_{\mu_0}^n|^2(y, s)dy\leq M^{\frac{1}{2}}e^{-s}.
\end{equation*}
This yields
\begin{equation}\label{v-27}
\int_{\mathbb{R}}|W_{\mu_0}^n|^2(y, s)dy\leq e^{\frac{6}{5}(-\log\varepsilon-s)}\int_{\mathbb{R}}|W_{\mu_0}^n|^2(y, -\log\varepsilon)dy+5M^{\frac{1}{2}}e^{-s}.
\end{equation}
Then it follows from \eqref{v-27}, \eqref{ii-3}, \eqref{iii-9} and \eqref{vii-401} that
\begin{equation}\label{v-271}
\int_{\mathbb{R}}|W_{\mu_0}^n|^2(y, s)dy\leq 6M^{\frac{1}{2}} e^{-s}.
\end{equation}

\subsection{Energy estimates of $W_{\mu_0}^j\ (1\leq j\leq n-1)$}

We now close the estimates of $W_{\mu_0}^j\ (1\leq j\leq n-1)$ in \eqref{v-2}. Similarly to \eqref{v-26},
there exists a minimal positive integer $\mu_0\geq 6$ such that for all $1\leq m\leq n-1$,
\begin{equation}\label{v-28}
3\mu_0-\frac{5}{2}-\beta_{\tau}e^{\frac{s}{2}}\partial_y\mu_m(w)\geq 3\mu_0-\frac{5}{2}-M e^{-s}-(1+\varepsilon^{\frac{1}{4}})|\partial_{w_n}\mu_m(0)|\geq 4.
\end{equation}
This also ensures the assumption \eqref{iv-50} in turn.

In addition, it follows from \eqref{v-7} and \eqref{v-271} that
\begin{equation}\label{v-29}
\int_{\mathbb{R}}|\mathbb{F}_{m}^{\mu_0}|^2(y, s)dy\leq \frac{\delta}{2n}\sum\limits_{j=1}^{n-1}\int_{\mathbb{R}}|W_{\mu_0}^j|^2(y, s)dy+c^{*}\sum\limits_{j=1}^{n-1}\int_{\mathbb{R}}|\eta^{-\frac{1}{3}}(y)W_{\mu_0}^j|^2(y, s)dy+24nM^{\frac{3}{4}}e^{-3s},
\end{equation}
where $\delta=8n^2 M^{-\frac{1}{16}}$ and $c^*=2n M^{\frac{1}{16}}$.
On the other hand, there exist two positive constants $K$ and $K^{*}$ such that
\begin{equation}\label{v-30}
\eta^{-\frac{2}{3}}(y)\leq \frac{\delta}{2n c^{*}}\ (|y|\geq K),\ \eta^{-\frac{2}{3}}(y)\leq \frac{K^*}{nc^{*}}\ (|y|\leq K).
\end{equation}
It is derived from \eqref{v-29} and \eqref{v-30} that
\begin{equation}\label{v-31}
\int_{\mathbb{R}}|\mathbb{F}_{m}^{\mu_0}|^2(y, s)dy\leq \frac{\delta}{n}\sum\limits_{j=1}^{n-1}\int_{\mathbb{R}}|W_{\mu_0 j}|^2(y, s)dy+\frac{K^*}{n}\sum\limits_{j=1}^{n-1}\int_{-K}^K|W_{\mu_0 j}|^2(y, s)dy+24n M^{\frac{3}{4}} e^{-3s}.
\end{equation}

Next we determine the Lipschitz continuous function $q_m(y)\ (1 \leq m\leq n-1)$ in Lemma \ref{lem8-1}. Due to \eqref{i-7a}
and $\mu_n(0)=0$, when $\mu_m(0)<0\ (1\leq m\leq i_0-1)$, then $q_m(y)$ is defined as
\begin{equation}\label{v-32}
q_m(y):=q_m^-(y)=\begin{cases}0,\ y\leq -K,\\
\frac{y}{K}+1,\ -K\leq y\leq K,\\
2,\ y\geq K.
\end{cases}
\end{equation}
When $\mu_m(0)>0\ (i_0\leq m\leq n-1)$, $q_m(y)$ is defined as
\begin{equation}\label{v-33}
q_m(y):=q_m^+(y)=q_m^-(-y).
\end{equation}
According to the definitions \eqref{v-32}-\eqref{v-33}, we have
\begin{equation}\label{v-34}
0\leq q_m(y)\leq 2.
\end{equation}
On the other hand, $Q_m(y, s)$ in \eqref{v-3} satisfies
\begin{equation}\label{v-35}\begin{aligned}
Q_m(y, s)&\geq \left(\frac{1}{K}\beta_{\tau}e^{\frac{s}{2}}|\mu_m(w)-\dot{\xi}(t)|-\frac{3}{2}K\right)\chi_{\{|y|\leq K\}}\\
&\geq \left(\frac{|\mu_m(0)|}{2K}e^{\frac{s}{2}}-\frac{3}{2}K\right)\chi_{\{|y|\leq K\}}\geq K^*\chi_{\{|y|\leq K\}},
\end{aligned}
\end{equation}
where the last inequality comes from \eqref{i-7}, \eqref{ii-3}, \eqref{vi2-4}, \eqref{vii-12}, \eqref{viii-1} and the fact of $s\geq -\log\varepsilon$.

With \eqref{v-271}-\eqref{v-28}, \eqref{v-31}, \eqref{v-34}-\eqref{v-35} and the largeness of $M$, summing up $m$
on both sides of \eqref{v-4} from $1$ to $n-1$ yield
\begin{equation*}\label{v-36}\begin{aligned}
\frac{d}{ds}\sum\limits_{m=1}^{n-1}\int_{\mathbb{R}}|W_{\mu_0}^m|^2(y, s)dy
+\frac{7}{2}\sum\limits_{m=1}^{n-1}\int_{\mathbb{R}}|W_{\mu_0}^m|^2(y, s)dy\leq 14n^2 M^{\frac{3}{4}} e^{4-3s}.
\end{aligned}
\end{equation*}
This shows
\begin{equation}\label{v-36}
\sum\limits_{m=1}^{n-1}\int_{\mathbb{R}}|W_{\mu_0}^m|^2(y, s)dy\leq e^{\frac{7}{2}(-\log\varepsilon-s)}\sum\limits_{m=1}^{n-1}\int_{\mathbb{R}}|W_{\mu_0}^m|^2(y, -\log\varepsilon)dy+28n^2 M^{\frac{3}{4}}e^{4-3s}.
\end{equation}
Then it comes from \eqref{iii-9}, \eqref{vii-401} and \eqref{v-36}  that
\begin{equation}\label{v-37}
\sum\limits_{m=1}^{n-1}\int_{\mathbb{R}}|W_{\mu_0}^n|^2(y, s)dy\leq 30n^2 e^4 M^{\frac{3}{4}} e^{-3s}.
\end{equation}

{\bf Proof of Theorem \ref{thm8-1}:} It is derived from \eqref{v-271}, \eqref{v-37} and \eqref{vii-402} that
\begin{equation}\label{v-38}
\sum\limits_{j=1}^{n-1}\|\partial_y^{\mu_0}W_j|^2(\cdot, s)\|_{L^2(\mathbb{R})}\leq (n-1)\sum\limits_{k=1}^{n-1}\|W_{\mu_0}^{k}(\cdot, s)\|_{L^2(\mathbb{R})}\leq \sqrt{30}n^2 e^2 M^{\frac{3}{8}}e^{-\frac{3}{2}s}
\end{equation}
and
\begin{equation}\label{v-39}
\|\partial_y^{\mu_0}W_n(\cdot, s)\|_{L^2(\mathbb{R})}\leq \sum\limits_{k=1}^{n}\|W_{\mu_0}^k(\cdot, s)\|_{L^2(\mathbb{R})}\leq \sqrt{30}n^2 e^2 M^{\frac{3}{8}}e^{-\frac{3}{2}s}+\sqrt{6}M^{\frac{1}{4}}e^{-\frac{s}{2}}.
\end{equation}
Then the estimates in \eqref{v-1} come from \eqref{v-38}, \eqref{v-39} and the largeness of $M$.
Therefore, the proof of Theorem \ref{thm8-1} is completed.

\section{Proof of main theorems}\label{V}

In the section, we complete the proofs of Theorem \ref{thmii-1} and Theorem \ref{thmi-1}.
\subsection{Proof of Theorem \ref{thmii-1}}

Based on the local existence of \eqref{i-6} (see \cite{MA}), we utilize the continuous induction to prove Theorem \ref{thmii-1}.
To this end, under the induction assumptions \eqref{iv-1}-\eqref{iv-3} and \eqref{iv-6} for suitably large $M>16$,
the proof of Theorem \ref{thmii-1} is mainly reduced to recover the estimates \eqref{iv-1}-\eqref{iv-3}
and \eqref{iv-6} with the smaller coefficient bounds via the bootstrap arguments.

The induction assumptions of $\kappa(t), \tau(t)$ and $\xi(t)$ in \eqref{iv-1a} and their derivatives in \eqref{iv-1b} are recovered in \eqref{viii-5}-\eqref{viii-51}, \eqref{viii-11} and \eqref{viii-2}-\eqref{viii-3}, \eqref{viii-1} with $M$ replaced by the smaller ones $M^{\frac{1}{2}}, 2$ and $2M^{\frac{3}{4}}$ respectively.

In the similar way, the assumptions \eqref{iv-2} for $W_0$ as well as $\mathcal{W}$ are also recovered by \eqref{vi-24}, \eqref{vi-36}, \eqref{vi-35}, \eqref{vi-16} and \eqref{vi-27} with the coefficients replaced by the smaller ones accordingly.  In addition, the assumptions \eqref{iv-3} for $W_j\ (j\neq n)$ are also obtained by \eqref{vii-12}, \eqref{vii-17}, \eqref{vii5-2} and \eqref{vii-30} with the coefficient $M$ replaced by the smaller ones. In addition, the energy assumptions \eqref{iv-6} are obviously
derived by \eqref{v-1} in Theorem \ref{thm8-1} with $M$ replaced by $M^{\frac{1}{2}}$.

Therefore, Theorem \ref{thmii-1} is proved via the method of continuous induction.

\subsection{Proof of Theorem \ref{thmi-1}}

Due to \eqref{ii-2} and Theorem \ref{thmii-1}, in order to complete the proof of Theorem \ref{thmi-1},
we only need to verify the $C^{\frac{1}{3}}$ optimal regularity of $W_0$ and \eqref{i-13}.

First, we show that the optimal regularity of $W_0(y, s)$ is $C^{\frac{1}{3}}$ with respect to the spatial variable.

By \eqref{ii-2}, for any $t\geq -\varepsilon$ and $x_1, x_2\in\mathbb{R}$, set
\begin{equation}\label{V-1}
s=s(t), y_i=(x_i-\xi(t))e^{\frac{3s}{2}}\ (i=1, 2).
\end{equation}
Then for any $\alpha>0$, by \eqref{ii-3}, we arrive at
\begin{equation}\label{V-2}
\frac{|w_n(x_1, t)-w_n(x_2, t)|}{|x_1-x_2|^{\alpha}}=e^{\frac{3\alpha s}{2}-\frac{s}{2}}\frac{|W_0(y_1, s)-W_0(y_2, s)|}{|y_1-y_2|^{\alpha}}.
\end{equation}
When $\alpha=\frac{1}{3}$, it is derived from \eqref{V-2}, \eqref{ii-9b} and \eqref{vi-35} that
\begin{equation}\label{V-3}\begin{aligned}
&\sup\limits_{x_1, x_2\in\mathbb{R}, x_1\neq x_2}\frac{|w_n(x_1, t)-w_n(x_2, t)|}{|x_1-x_2|^{\alpha}}\\
\leq& \sup\limits_{y_1, y_2\in\mathbb{R}, y_1\neq y_2}\frac{|W_0(y_1, s)-W_0(y_2, s)|}{|y_1-y_2|^{\frac{1}{3}}}\\
=&\sup\limits_{y_1, y_2\in\mathbb{R}, y_1\neq y_2}\frac{\displaystyle|\int_{y_2}^{y_1}\partial_y W_0(z, s)dz|}{|y_1-y_2|^{\frac{1}{3}}}\\
\leq&2\sup\limits_{y_1, y_2\in\mathbb{R}, y_1\neq y_2}\frac{\displaystyle|\int_{y_2}^{y_1}\eta^{-\frac{1}{3}}(z)dz|}{|y_1-y_2|^{\frac{1}{3}}}\leq 12.
\end{aligned}
\end{equation}

When $\frac{1}{3}<\alpha<1$, with \eqref{ii-8}, \eqref{ii-14}, \eqref{vi-16} and $y_{1}^*=1, y_2^*=0$ (Denote ($x_1^*, x_2^*)$ by the
corresponding transformation \eqref{V-1} respectively), we have
\begin{equation}\label{V-4}\begin{aligned}
\frac{|W_0(y_1^*, s)-W_0(y_2^*, s)|}{|y_1^*-y_2^*|^{\alpha}}&\geq |\overline{W}(1)|-|\mathcal{W}(1, s)-\mathcal{W}(0, s)|\\
&\geq |\overline{W}(1)|-2\varepsilon^{\frac{1}{11}}>\frac{1}{2}|\overline{W}(1)|>0.
\end{aligned}
\end{equation}

Thus, \eqref{V-4}  shows that when $\frac{1}{3}<\alpha<1$, for any $\overline{M}>0$, there exists a
constant $s_0\geq -\log\varepsilon$ (and corresponding $t_0$ by \eqref{V-1}) such that when $s\geq s_0$,
\begin{equation}\label{V-5}
e^{\frac{3\alpha s}{2}-\frac{s}{2}}\frac{|W_0(y_1^*, s)-W_0(y_2^*, s)|}{|y_1^*-y_2^*|^{\alpha}}\geq \frac{1}{2}e^{\frac{3\alpha s_0}{2}-\frac{s_0}{2}}|\overline{W}(1)|>\overline{M}.
\end{equation}

Due to the arbitrariness of $\overline{M}$,  it implies from \eqref{V-2} and \eqref{V-5} that $w_n(x, t)\notin C^{\alpha}$
with $\alpha>\frac{1}{3}$. Combining this with \eqref{V-3} shows that the optimal regularity of $w_n(x, t)$
is $C^{\frac{1}{3}}$ with respect to the spatial variable.

Therefore, (1) and (2) in Theorem \ref{thmi-1}  are obtained from Theorem \ref{thmii-1} and the above
verification of $C^{\frac{1}{3}}$ regularity for $W_0$.

Next we prove \eqref{i-13}. It is derived from \eqref{ii-2}, \eqref{ii-3} and \eqref{ii-10} that
\begin{equation}\label{V-6}
\partial_x w_n(\xi(t), t)=e^{s}\partial_y W_0(0, s)=-e^{s}=\frac{1}{\tau(t)-t}.
\end{equation}
By the definition of $\tau(t)$ in \eqref{ii-1}, we have $\tau(T^*)=T^*$ and then
\begin{equation}\label{V-7}
|T^*-\tau(t)|=|\tau(T^*)-\tau(t)|\leq \int_t^{T^*}|\dot{\tau}(z)|dz\lesssim \varepsilon^{\frac{1}{3}}(T^*-t),
\end{equation}
where the last inequality comes from (2) in Theorem \ref{ii-1}.

Following \eqref{V-7}, one has
\begin{equation}\label{V-8}\begin{aligned}
&|\tau(t)-t|\leq |T^*-t|+|\tau(t)-T^*|\leq (1+\varepsilon^{\frac{1}{4}})|T^*-t|,\\
&|\tau(t)-t|\geq |T^*-t|-|\tau(t)-T^*|\geq (1-\varepsilon^{\frac{1}{4}})|T^*-t|.
\end{aligned}
\end{equation}

Then, we derive from \eqref{V-6}-\eqref{V-8} that
\begin{equation}\label{V-9}
-2<(T^*-t)\partial_x w_n(\xi(t), t)<-\frac{1}{2}.
\end{equation}

Collecting \eqref{V-7}, \eqref{V-9} and Theorem \ref{ii-1} yields \eqref{i-13} in Theorem \ref{thmi-1} and then the
proof of Theorem \ref{thmi-1} is completed.

\appendix

\renewcommand{\appendixname}{}

\section{Appendix}\label{A}

In the Appendix, we introduce a useful interpolation inequality
(see \cite{Ad}) and give its applications.
\begin{lem}\label{lemA-3} {\bf (Gagliardo-Nirenberg-Sobolev inequality).} {\it Let $u: \mathbb{R}^d\to\mathbb{R}$.
Fix $1\leq q, r\leq \infty$ and $j, m\in\mathbb{N}$, and $\frac{j}{m}\leq \alpha\leq 1$. If
\begin{equation*}
\frac{1}{p}=\frac{j}{d}+\alpha(\frac{1}{r}-\frac{m}{d})+\frac{1-\alpha}{q},
\end{equation*}
then one has
\begin{equation}\label{A2-1}
\|D^j u\|_{L^p}\leq C\|D^m u\|_{L^r}^{\alpha}\|u\|_{L^q}^{1-\alpha},
\end{equation}
where the positive constant $C$  depends on $d, m, r, q$ and $m$.} \end{lem}

Next we estimate the $L^2-$norm of  the terms ${I_{\beta q k}}'s\ (1\leq k\leq n)$ with the expression as
\begin{equation}\label{A2-2}
I_{\beta q k}=\sum\limits_{\gamma_1+\cdots+\gamma_q=\mu_0-\beta+1,\ \gamma_j\geq 1 (1\leq j\leq q)}|\partial_y^{\gamma_1}W|\cdots |\partial_y^{\gamma_q}W|\cdot|\partial_y^{\beta}W_k|.
\end{equation}

The estimates of ${I_{\beta q k}}'s$ are considered in two cases: $1\leq k\leq n-1$ and $k=n$.
\begin{lem}\label{lemA-4} {\it For $1\leq k\leq n-1$, we have
\begin{subequations}\label{A2-3}\begin{align}
&\sum\limits_{\beta=1, \mu_0}\int_{\mathbb{R}}I_{\beta 1 k}^2(y, s)dy\nonumber\\
&\qquad\qquad\leq  2e^{-s}\sum\limits_{j=1}^{n-1}\int_{\mathbb{R}}|\eta^{-\frac{1}{3}}(y)W_{\mu_0}^j|^2(y, s)dy+M^{\frac{1}{16}}e^{-3s}\int_{\mathbb{R}}|W_{\mu_0}^{n}|^2(y, s)dy+M^3 e^{-6s},\label{A2-3a}\\
&\sum\limits_{1< \beta<\mu_0}\int_{\mathbb{R}}I_{\beta 1 k}^2(y, s)dy\nonumber\\
&\qquad\qquad\leq 2 M^{-\frac{1}{8}}e^{-s}\sum\limits_{j=1}^{n-1}\int_{\mathbb{R}}|W_{\mu_0}^j|^2(y, s)dy+M^{\frac{1}{4}}e^{-3s}\int_{\mathbb{R}}|W_{\mu_0}^n|^2(y, s)dy+M^3 e^{-6s},\label{A2-3b}\\
&\int_{\mathbb{R}}I_{\beta q k}^2(y, s)\leq M^{3\mu_0+1}e^{-(4+\frac{1}{4})s}\ (q\geq 2).\label{A2-3c}
\end{align}
\end{subequations}
}
\end{lem}

\begin{proof} For the proof of \eqref{A2-3a},  by \eqref{ii-3},  \eqref{vi-35}, \eqref{vii-16}, the expansion in \eqref{ii-17}, \eqref{vii-401}-\eqref{vii-402} and \eqref{v-2}, we have
\begin{equation}\label{v-10}\begin{aligned}
&\int_{\mathbb{R}}(I^2_{\mu_0 1 k}+I^2_{1 1 k})(y, s)dy\\
\leq&\int_{\mathbb{R}}|\partial_y W|^2\cdot |\partial_y^{\mu_0}W_k|^2(y, s)dy+\int_{\mathbb{R}}|\partial_y W_k|^2|\partial_y^{\mu_0}W|^2(y, s)dy\\
\leq& 2e^{-s}\sum\limits_{j=1}^{n-1}\int_{\mathbb{R}}|\eta^{-\frac{1}{3}}(y)W_{\mu_0}^j|^2(y, s)dy+M^{\frac{1}{16}} e^{-3s}\sum\limits_{j=1}^{n}\int_{\mathbb{R}}|W_{\mu_0}^j|^2(y, s)dy.
\end{aligned}
\end{equation}
Combining \eqref{v-10} with \eqref{v-2} yields \eqref{A2-3a}.

 With respect to the case of $1<\beta<\mu_0$ and $q=1$ in \eqref{A2-2}, it is derived from \eqref{ii-3}, \eqref{vi-35}, \eqref{vii-402}, \eqref{vii-16}, \eqref{v-2}, H\"older inequality and Lemma \ref{lemA-3} that
\begin{equation}\label{v-11}\begin{aligned}
&\sum\limits_{1<\beta<\mu_0}\int_{\mathbb{R}}I_{\beta 1 k}^2(y, s)dy\\
\leq&\sum\limits_{1<\beta<\mu_0}\int_{\mathbb{R}}|\partial_y^{\mu_0+1-\beta}W|^2 |\partial_y^{\beta}W_k|^2(y, s)dy\\
\leq&\sum\limits_{1<\beta<\mu_0}\|\partial_y^{\beta}W_k(\cdot, s)\|_{L^{\frac{2(\mu_0-1)}{\beta-1}}}^2\|\partial_y^{\mu_0-\beta+1}W(\cdot, s)\|_{L^{\frac{2(\mu_0-1)}{\mu_0-\beta}}}^2\\
\leq&M^{\frac{1}{16}} \sum\limits_{1<\beta<\mu_0}\|\partial_y^{\mu_0}W_k(\cdot, s)\|_{L^2}^{2\frac{\beta-1}{\mu_0-1}}\|\partial_y W_k(\cdot, s)\|_{L^{\infty}}^{2\frac{\mu_0-\beta}{\mu_0-1}}\|\partial_y^{\mu_0}W(\cdot, s)\|_{L^2}^{2\frac{\mu_0-\beta}{\mu_0-1}}\|\partial_y W(\cdot, s)\|_{L^{\infty}}^{2\frac{\beta-1}{\mu_0-1}}\\
\leq&M^{\frac{1}{8}+\frac{\beta-1}{8(\mu_0-1)}}\|\partial_y W_k(\cdot, s)\|_{L^{\infty}}^2\|\partial_y^{\mu_0}W(\cdot, s)\|_{L^2}^2+M^{-\frac{1}{8}}\|\partial_y W(\cdot, s)\|_{L^{\infty}}^2\|\partial_y^{\mu_0}W_k(\cdot, s)\|_{L^2}^2\\
\leq& 2M^{-\frac{1}{8}} e^{-s}\sum\limits_{j=1}^{n-1}\int_{\mathbb{R}}|W_{\mu_0}^j|^2(y, s)dy+M^{\frac{1}{4}}e^{-3s}\sum\limits_{j=1}^{n}\int_{\mathbb{R}}|W_{\mu_0}^j|^2(y, s)dy.
\end{aligned}
\end{equation}
Then \eqref{A2-3b} comes from \eqref{v-11} and \eqref{v-2}.

For the easier cases of $I_{\beta q k}$ ($q\geq 2$) in \eqref{A2-2}, we
will apply the following three type estimates with the help of Lemma \ref{lemA-3}:
\begin{subequations}\label{v-13}\begin{align}
&\|\partial_y^{\gamma_j}W(\cdot, s)\|_{L^{\infty}}\leq M^{\frac{1}{16}}\|\partial_y^{\mu_0}W(\cdot, s)\|_{L^2}^{\frac{\gamma_j-1}{\mu_0-1-\frac{1}{2}}}\|\partial_y W(\cdot, s)\|_{L^{\infty}}^{\frac{\mu_0-\frac{1}{2}-\gamma_j}{\mu_0-1-\frac{1}{2}}},\label{v-13a}\\
&\|\partial_y^{\gamma_j}W(\cdot, s)\|_{L^{2}}\leq M^{\frac{1}{16}}\|\partial_y^{\mu_0}W(\cdot, s)\|_{L^2}^{\frac{\gamma_j-1-\frac{1}{2}}{\mu_0-1-\frac{1}{2}}}\|\partial_y W(\cdot, s)\|_{L^{\infty}}^{\frac{\mu_0-\gamma_j}{\mu_0-1-\frac{1}{2}}}\ (\gamma_j\geq 2),\label{v-13b}\\
&\|\partial_y^{\beta}W_k(\cdot, s)\|_{L^2}
\leq M^{\frac{1}{16}}\|\partial_y^{\mu_0}W_k(\cdot, s)\|_{L^2}^{\frac{\beta-1-\frac{1}{2}}{\mu_0-1-\frac{1}{2}}}\|\partial_y W_k(\cdot, s)\|_{L^{\infty}}^{\frac{\mu_0-\beta}{\mu_0-1-\frac{1}{2}}}\ (\beta\geq 2).\label{v-13c}
\end{align}
\end{subequations}

When $\beta\geq 2$ and $q\geq 2$, substituting \eqref{v-13a} and \eqref{v-13c} into \eqref{A2-2} yields
\begin{equation}\label{v-131}\begin{aligned}
&\int_{\mathbb{R}} I_{\beta q k}^2 (y, s)dy\\
=&\sum\limits_{\gamma_1+\cdots+\gamma_q=\mu_0-\beta+1,\ \gamma_j\geq 1 (1\leq j\leq q)}\|\partial_y^{\gamma_1}W(\cdot, s)\|_{L^{\infty}}^2\cdots\|\partial_y^{\gamma_q}W(\cdot, s)\|_{L^{\infty}}^2\|\partial_y^{\beta}W_k(\cdot, s)\|_{L^{2}}^2\\
\leq & M^{\frac{q+1}{4}}\|\partial_y^{\mu_0}W(\cdot, s)\|_{L^2}^{2\frac{\mu_0-\beta+1-q}{\mu_0-1-\frac{1}{2}}}\|\partial_y W(\cdot, s)\|_{L^{\infty}}^{2\frac{q(\mu_0-\frac{1}{2})-(\mu_0-\beta+1)}{\mu_0-1-\frac{1}{2}}}\|\partial_y^{\mu_0}W_k(\cdot, s)\|_{L^{2}}^{2\frac{\beta-1-\frac{1}{2}}{\mu_0-1-\frac{1}{2}}}\|\partial_y W_k(\cdot, s)\|_{L^{\infty}}^{2\frac{\mu_0-\beta}{\mu_0-1-\frac{1}{2}}}\\
\leq& M^{\frac{q+1}{2}} \left(\|\partial_y W(\cdot, s)\|_{L^{\infty}}^{2q}+\|\partial_y^{\mu_0}W(\cdot, s)\|_{L^2}^{2q}\right)\left(\|\partial_y^{\mu_0}W_k(\cdot, s)\|_{L^2}^2+\|\partial_y W_k(\cdot, s)\|_{L^{\infty}}^2\right).
\end{aligned}
\end{equation}
Combining this with \eqref{iv-6}, \eqref{vi-35} and \eqref{vii-16} derives
\begin{equation}\label{v-14}
\int_{\mathbb{R}}I_{\beta q k}^2 (y, s)dy\leq M^{3q+1}e^{-(q+3)s}\leq M^{3\mu_0+1}e^{-(q+3)s}\ (q\geq 2, \beta\geq 2).
\end{equation}

When $\beta=1$ and $q\geq 3$, similarly to \eqref{v-14}, we have
\begin{equation}\label{v-15}\begin{aligned}
&\int_{\mathbb{R}} I_{1 q k}^2 (y, s)dy\\
=&\sum\limits_{\gamma_1+\cdots+\gamma_q=\mu_0-\beta+1,\ \gamma_j\geq 1 (1\leq j\leq q)}\|\partial_y^{\gamma_1}W(\cdot, s)\|_{L^{\infty}}^2\cdots\|\partial_y^{\gamma_q}W(\cdot, s)\|_{L^{\infty}}^2\|\partial_y W_k(\cdot, s)\|_{L^{2}}^2\\
\leq & M^{\frac{q+1}{4}}\|\partial_y^{\mu_0}W(\cdot, s)\|_{L^2}^{2\frac{\mu_0-\beta+1-q}{\mu_0-1-\frac{1}{2}}}\|\partial_y W(\cdot, s)\|_{L^{\infty}}^{2\frac{q(\mu_0-\frac{1}{2})-(\mu_0-\beta+1)}{\mu_0-1-\frac{1}{2}}}\|\partial_y W_k(\cdot, s)\|_{L^{2}}^{2}\\
\leq& M^{\frac{q+1}{2}} \left(\|\partial_y W(\cdot, s)\|_{L^{\infty}}^{2q}+\|\partial_y^{\mu_0}W(\cdot, s)\|_{L^2}^{2q}\right)\|\partial_y W_k(\cdot, s)\|_{L^{2}}^2\\
\leq &M^{3q+1}e^{-(q+\frac{5}{4})s}\ (q\geq 3),
\end{aligned}
\end{equation}
where the last estimate comes from \eqref{iv-6} and \eqref{vii5-2} with $\nu=\frac{7}{24}$.

When $\beta=1$ and $q=2$, without loss of generality, we assume $\gamma_2\geq 2$ due to $\gamma_1+\gamma_2=\mu_0\geq 6$.
Then we apply \eqref{v-13a} and \eqref{v-13b} to control $\partial_y^{\gamma_1}W$ and $\partial_y^{\gamma_2}W$ respectively and
subsequently obtain the following estimate with \eqref{iv-6} and \eqref{vii-16}
\begin{equation}\label{v-151}\begin{aligned}
&\int_{\mathbb{R}}I_{1 q k}^2(y, s)dy\\
=&\|\partial_y^{\gamma_1}W(\cdot, s)\|_{L^{\infty}}^2\|\partial_y^{\gamma_2}W(\cdot, s)\|_{L^2}^2\|\partial_y W_k(\cdot, s)\|_{L^{\infty}}^2\\
\leq & M^{\frac{1}{4}}\|\partial_y W(\cdot, s)\|_{L^{\infty}}^{2\frac{\mu_0-\frac{1}{2}}{\mu_0-1-\frac{1}{2}}}\|\partial_y^{\mu_0}W(\cdot, s)\|_{L^{2}}^{2\frac{\mu_0-2-\frac{1}{2}}{\mu_0-1-\frac{1}{2}}}\|\partial_y W_k(\cdot, s)\|_{L^{\infty}}^2\\
\leq &M^5 e^{-5s}.
\end{aligned}
\end{equation}
Thus, we  get \eqref{A2-3c} from \eqref{v-14}-\eqref{v-151} and the proof of Lemma \ref{lemA-4} is finished.\end{proof}

\begin{lem}\label{lemA-5} {\it For $I_{\beta q n}$, one has
\begin{equation}\label{A-10}
\int_{\mathbb{R}}I_{\beta q n}^2(y, s)dy\leq M^{3q+1}e^{-(q+1)s}.
\end{equation}
}
\end{lem}

\begin{proof} As in Lemma \ref{lemA-4}, the estimates on the two terms $I_{\mu_0 1 n}$ and $I_{1 1 n}$ are crucial in the
proof of \eqref{A-10}. For these two terms, similarly to \eqref{v-10}, one has
\begin{equation}\label{A-111}
\int_{\mathbb{R}}(I_{\mu_0 1 n}^2+_{1 1 n}^2)(y, s)dy\leq 2\int_{\mathbb{R}}|\partial_y W|^2\cdot|\partial_y^{\mu_0}W|^2(y, s)dy
\leq 2\|\partial_y W(\cdot, s)\|_{L^{\infty}}^2\cdot\|\partial_y^{\mu_0}W(\cdot, s)\|_{L^2}^2.
\end{equation}

When $1<\beta<\mu_0$, similarly to \eqref{v-131}, we arrive at
\begin{equation}\label{A-112}
\int_{\mathbb{R}}I_{\beta q n}^2(y, s)dy\leq  M^{\frac{q+1}{2}}\left(\|\partial_y W(\cdot, s)\|_{L^{\infty}}^{2q}+\|\partial_y^{\mu_0}W(\cdot, s)\|_{L^2}^{2q}\right)\left(\|\partial_y W_n(\cdot, s)\|_{L^{\infty}}^2+\|\partial_y^{\mu_0}W_n(\cdot, s)\|_{L^2}^2\right).
\end{equation}
Combining \eqref{A-111}-\eqref{A-112} with \eqref{ii-3}, \eqref{iv-6}, \eqref{vi-35} and \eqref{vii-16} yields \eqref{A-10}
and the proof of Lemma \ref{lemA-5} is completed.
\end{proof}

\vskip 1 true cm

\end{document}